\documentclass{article}
\usepackage[a4paper]{geometry}
\usepackage{a4wide}
\usepackage{indentfirst}
    
\usepackage[english]{babel}

\usepackage{tikz}
\usepackage{tikz-cd}
\usetikzlibrary{positioning,calc,arrows.meta,decorations.pathreplacing}

\tikzset{
  mybox/.style={
    draw, very thick,
    minimum width=48mm,
    minimum height=16mm,
    align=center,
    font=\bfseries
  },
  myboxR/.style={
    draw, very thick,
    minimum width=55mm,
    minimum height=16mm,
    align=center,
    font=\bfseries
  },
  myoval/.style={
    draw, very thick,
    minimum width=40mm,
    minimum height=22mm,
    align=center,
    font=\bfseries
  },
  downarr/.style={-{Stealth[open,length=7mm,width=7mm]}, very thick},
  arr/.style={-{Stealth[open,length=6mm,width=6mm]}, very thick},
  topdash/.style={very thick, dash pattern=on 10mm off 8mm},
  vdot/.style={very thick, dotted}
}

\usepackage{amsthm}
\usepackage{amsmath}

\allowdisplaybreaks
\usepackage{amstext}
\usepackage{amssymb}
\usepackage{color}
\usepackage{graphicx}

\newcommand{\heart}{\ensuremath\heartsuit}
\usepackage{subcaption}
\usepackage{mathtools}
\usepackage{hyperref}
\usepackage{csquotes}
\usepackage{url}
\newcommand{\footremember}[2]{%
    \footnote{#2}
    \newcounter{#1}
    \setcounter{#1}{\value{footnote}}%
}

\usepackage{multicol}
\usepackage{lipsum}
\def\A{{\mathcal A}}

\usepackage{float}

\makeatletter
\newcommand{\proofpart}[2]{%
  \par
  \addvspace{\medskipamount}%
  \noindent\emph{Step #1: #2}\par\nobreak
  \addvspace{\smallskipamount}%
  \@afterheading
}
\makeatother

\newtheorem{definition}{Definition}
\newtheorem{theorem}{Theorem}
\newtheorem{corollary}{Corollary}

\newtheorem{proposition}{Proposition}
\newtheorem{lemma}{Lemma}
\newtheorem{notation}{Notation}
\newtheorem{remark}{Remark}
\newtheorem{assumption}{Assumption}
\date{}
\allowdisplaybreaks
\makeatother
\title{Approximation of Discrete-Time Infinite-Horizon Mean-Field
Equilibria via Finite-Horizon Mean-Field Equilibria\thanks{Research of first and second authors was supported in part by the AFOSR Grant FA9550-24-1-0152}}
\usepackage[all,arc]{xy}
\usepackage{enumerate}
\usepackage{mathrsfs}
\usepackage{setspace}
\usepackage{biblatex}
\bibliography{whole}
\author{%
  U\u{g}ur Ayd{\i}n\footremember{ua}{Department of Electrical and Computer Engineering, University of Illinois Urbana-Champaign, Urbana, IL, 61801, USA, uaydin2@illinois.edu}%
  \and Tamer Ba\c{s}ar\footremember{tb}{Department of Electrical and Computer Engineering, University of Illinois Urbana-Champaign, Urbana, IL, 61801, USA, basar1@illinois.edu}%
  \and Naci Saldi\footremember{ns}{Department of Mathematics, Bilkent University, \c{C}ankaya, Ankara, 06800, TURKEY, naci.saldi@bilkent.edu.tr}%
}

\begin{document}
\maketitle
\begin{abstract}
We address in this paper a fundamental question that arises in mean-field games (MFGs), namely whether mean-field equilibria (MFE)\footnotemark  for discrete-time finite-horizon MFGs can be used to obtain approximate stationary as well as non-stationary MFE for similarly structured infinite-horizon MFGs. We provide a rigorous analysis of this relationship, and show that any accumulation point of MFE of a discounted finite-horizon MFG constitutes, under weak convergence as the time horizon goes to infinity, a non-stationary MFE for the corresponding infinite-horizon MFG. Further, under certain conditions, these non-stationary MFE converge to a stationary MFE, establishing the appealing result that finite-horizon MFE can serve as approximations for stationary MFE. Additionally, we establish improved contraction rates for iterative methods used to compute regularized MFE in finite-horizon settings, extending existing results in the literature. As a byproduct, we obtain that when two MFGs have finite-horizon MFE that are close to each other, the corresponding stationary MFE are also close. As one application of the theoretical results, we show that finite-horizon MFGs can facilitate learning-based approaches to approximate infinite-horizon MFE when system components are unknown. Under further assumptions on the Lipschitz coefficients of the regularized system components (which are stronger than contractivity of finite-horizon MFGs), we obtain exponentially decaying finite-time error bounds-- in the time horizon--between finite-horizon non-stationary, infinite-horizon non-stationary, and stationary MFE. {Contraction conditions for the infinite-horizon non-stationary MFG over the space of bounded sequences are shown to be equivalent to that of stationary MFGs.} As a byproduct of our error bounds, we present a new uniqueness criterion for infinite-horizon nonstationary MFE beyond the available contraction results in the literature for stationary MFGs (and hence the contraction result we find for the infinite-horizon non-stationary MFGs).
\end{abstract}
\section{Introduction}
\footnotetext{When we use the acronym ``MFE'',``E'' stands for both singular (equilibrium) and plural (equilibria) versions, with the precise attribution to be clear from context.}

This work investigates the relationship between finite-horizon mean-field equilibria (MFE) in discounted finite-horizon mean-field games (MFGs) and infinite-horizon MFE, encompassing both stationary and non-stationary scenarios. In learning theory, when system components are unknown, Bayesian methods often employ finite-horizon models to construct sample priors that approximate the true system parameters \cite{huang2024statistical, pasztor2023efficient}. Similarly, adaptive learning techniques utilize finite-horizon MFGs to estimate and learn these true parameters \cite{ramponi2024imitation, light2025computing}. However, evaluations of finite-horizon models within the mean-field framework typically neglect infinite-horizon benchmarks \cite[Definition 7]{ramponi2024imitation}, \cite[Eq. (4)]{huang2024statistical}. Moreover, in the learning framework, the non-stationary case has so far been studied only for finite-horizon MFGs; see \cite{lauriere2022learning}. Therefore, approximating an infinite-horizon non-stationary MFE by finite-horizon non-stationary MFEs allows these results to be extended to the infinite-horizon setting. The infinite-horizon non-stationary case is typically more challenging than both the finite-horizon and stationary settings, due to the infinite forward--backward recursion linking the $Q$-functions and the state measures. This oversight prompts the crucial question of whether methods developed for finite-horizon scenarios are capable of accurately approximating infinite-horizon equilibria:
\begin{itemize}
\item[Q)] Can finite-horizon MFE effectively approximate infinite-horizon MFE (both stationary and non-stationary)?
\end{itemize}
If the answer to this question is affirmative, and if infinite-horizon equilibria are unique, then models derived using finite-horizon MFGs can also provide approximate MFE for infinite-horizon MFE obtained under the targeted system components.

To address this problem, we will prove that any accumulation point of finite-horizon MFE, under weak convergence as the time horizon increases, constitutes a non-stationary infinite-horizon MFE. We recall that a discrete-time non-stationary infinite-horizon MFE is characterized by optimal policies $(\pi_t)_{t \in \mathbb N}$ and state-measures $(\mu_t)_{t \in \mathbb N}$ obtained from an infinite sequence of $Q$-functions $(Q_t)_{t\in \mathbb N}$ such that the forward-backward recursions $Q_t=H_1(Q_{t+1},\mu_t)$ and $\mu_{t+1}=H_2(\pi_t,\mu_t)$ are satisfied for some time independent maps $H_1$ and $H_2$ \cite[Remark 7]{SaBaRaSIAM}. It follows that optimal policy $\pi_t$ at time $t$ depends on the entire trajectory of $Q$-functions $(Q_{j})_{j \ge t}$ and state-measures $(\mu_j)_{j \le t}$, which requires solving infinitely many equations at once to obtain an infinite-horizon non-stationary MFE, which is intractable in practice. Given that directly solving dynamic programming problems in the infinite-horizon context is typically intractable, our result not only advances the theoretical understanding of Bayesian learning within MFGs but also highlights a practical method for approximating infinite-horizon nonstationary MFE. Specifically, our result shows that finite-horizon MFE can effectively serve as approximations, thus providing a viable and computationally manageable strategy for both learning and approximating infinite-horizon non-stationary MFE. 

Finite-horizon MFE depend on both the length of the time horizon and the time parameter, which is bounded by the horizon length. Conversely, stationary MFE are independent of the length of the time horizon but require an ergodicity condition on the state-measure component. Consequently, approximating stationary MFE using finite-horizon MFE necessitates extending both the time horizon and the time parameter simultaneously, posing significant challenges. To overcome this obstacle, we will utilize infinite-horizon non-stationary MFE as intermediate terms. By investigating conditions under which infinite-horizon non-stationary MFE converge to stationary MFE, we establish a pathway for approximation. Provided that finite-horizon MFE converge first to infinite-horizon non-stationary MFE, which in turn converge to stationary MFE, we demonstrate that finite-horizon MFE can effectively approximate stationary MFE, which will partially answer the question above regarding the relation between stationary MFE and finite-horizon MFE.

To demonstrate the effectiveness of finite-horizon MFGs in learning compared to the infinite-horizon setting, we will provide improved contraction rates for the so-called mean-field equilibrium operators used for iterative methods to learn {finite-horizon} MFE under contraction \cite{anahtarci2020value}. {We show that the contraction condition for these iterative algorithms for infinite-horizon non-stationary MFGs over the space of bounded sequences, equipped with the uniform norm, coincides with the contraction condition for the stationary MFGs obtained in \cite{anahtarci2023q}. In particular, we demonstrate that a finite-horizon MFG may be contractive for every horizon length $T<\infty$, while the corresponding infinite-horizon (non-stationary and stationary) MFGs may fail to be contractive. Furthermore, under conditions weaker than the contraction condition for infinite-horizon MFGs, we show that there exists a unique infinite-horizon non-stationary MFG for any given initial state measure, and that the finite-horizon MFEs obtained under the same initial state measure converge to this infinite-horizon MFE as the horizon length $T \to \infty$.
}

We will provide experimental results as well as theoretical bounds showing that improved contractivity holds for the finite-horizon setting when the known results in the literature for the infinite-horizon setting fail.
We will mainly focus our contractivity results on the regularized setting developed in \cite{anahtarci2023q}, but we also expect our results to hold in other contractivity settings used in the literature \cite{cui2021approximately, guo2019learning} (provided that the state space is compact).

\subsection{Literature Review}
\subsubsection{Iterative Methods for Finding Equilibria of Discrete-time Mean-field Games}
In infinite-horizon stationary MFGs, most iterative methods for finding MFE focus on two approaches: contractive methods and monotonicity. The monotonicity condition for MFGs allows us to show that there exists a unique MFE for the system without any further restrictions on the iterations, such as small Lipschitz coefficients \cite{zhang2023learning}. In contrast, the contractive method requires small Lipschitz coefficients for the system components as well as access to a Lipschitz continuous optimal policy \cite{guo2019learning,subramanian2019reinforcement,anahtarci2020value}. In the current literature, the most common way to satisfy these restrictions is found in finite state and action space settings with Lipschitz continuous system components, where one perturbs the system objective with a \emph{regularizer} \cite{anahtarci2023q,cui2021approximately}, which causes a deviation from the true equilibria to obtain a Lipschitz continuous optimal policy. Despite the available prior work on stationary and finite-horizon non-stationary MFE, there is no fixed-point iteration algorithm studied in the case of infinite-horizon non-stationary MFE \cite{lauriere2022learning}.
\subsubsection{Function Approximation in Mean-Field setting}
Bayesian methods often utilize function approximations to choose a model from a given set of functions. In the mean-field setting, there is currently limited literature available regarding the use of function approximations. The work \cite{pasztor2023efficient} used reproducing kernel Hilbert spaces to perform function approximation over upper confidence intervals for finite-horizon mean-field control with near-deterministic transition functions and sub-Gaussian noises in general state and action spaces. The work \cite{huang2024statistical} provided sample complexity results for MFG and mean-field control settings under function approximation when the state and action spaces are finite in the finite-horizon setting. The work \cite{anahtarci2023learning} used function approximations for infinite-horizon MFGs in a model-free setting under regularization with finite states and actions. Recently, adaptive learning methods have also started to gain atraction in the MFG setting \cite{light2025computing}, \cite{ramponi2024imitation}. 
\subsubsection{Robustness of MFGs}
Approximation of infinite-horizon MFE with finite-horizon MFE can be considered as a form of robustness of the system with respect to the time horizon. Although there is no other work available in the literature regarding the robustness of MFE with respect to the time-horizon, several robustness results have appeared recently in the mean-field setting. The stability of Stackelberg MFGs has been studied in \cite{guo2022optimization}. For general MFGs, the robustness of MFE has been studied in \cite{aydin2023robustness} under model uncertainties. The proof of convergence from finite-horizon MFEs to infinite-horizon nonstationary MFEs in our work relies on techniques similar to those introduced in \cite{aydin2023robustness}. In continuous time, in \cite{bauso2016robust}, MFGs that incorporate uncertainty in both states and payoffs have been investigated. In \cite{moon2016linear}, the authors consider linear-quadratic risk-sensitive and robust mean-field games. For MDPs, the robustness of the value function has been studied in \cite{baker2016continuity} \cite{kara2019robustness}. Although convergence of policies is not considered in \cite{kara2019robustness}, they utilize the ``continuous convergence'' of the system components (see Assumption \ref{assump:2}), which also plays an important role in our work. 

\subsection{Contributions}

This work presents both qualitative and quantitative results regarding approximation of an infinite-horizon MFE by a finite-horizon MFE in the discrete-time setting.

\begin{enumerate}
    \item In Section \ref{sect:2}, we will review the discrete-time MFGs in the framework setting introduced in \cite{SaBaRaSIAM}. Since we are interested in the relations of MFE obtained in each of the finite-horizon, infinite-horizon non-stationary, and stationary settings, we provide a short description of each setting.
    \item In Section \ref{sect:fixpt}, we will present our quantitative approximation results. For this purpose, we will first study the fixed-point iteration for mean-field games under regularization in finite state and action spaces. We mainly base our analysis on the framework introduced in \cite{anahtarci2023q} by introducing a regularizer. {We further show that the contraction condition for infinite-horizon non-stationary MFGs over the space of bounded sequences coincides with that of stationary MFGs. To obtain these contraction conditions,} we analyze the fixed-point iteration by means of vector inequalities (Theorem \ref{thrm:f2}). It will be shown that the contraction condition in the finite-horizon setting is weaker than the one in infinite-horizon {stationary} setting \cite[Proposition 1]{anahtarci2023q} (Theorem \ref{prop:2} and Theorem \ref{thrm:a1}).
    
    We expect our techniques to be applicable to several other settings as well, such as the Boltzmann setting introduced in \cite{cui2021approximately}, and the nonregularized convex setting studied in \cite{anahtarci2020value}. We also believe that our techniques can yield sharper convergence-rate guarantees for algorithms on finite-horizon MFGs. 
    
    Finally, under assumptions stronger than contractivity {of finite-horizon MFGs (but weaker than the contraction of infinite-horizon MFGs)}, we establish error bounds between finite-horizon and infinite-horizon MFE (Theorem \ref{thrm:7}). These bounds, in turn, allow us to derive a new uniqueness result for infinite-horizon non-stationary MFE under discounted cost.
    \item In Section \ref{sect:approx}, we will present our qualitative approximation results under more general settings than those in Section \ref{sect:fixpt}. For this purpose, we will study the relationship between the finite-horizon discounted cost MFE and the infinite-horizon non-stationary and stationary discounted cost MFE under compact state and action spaces. As mentioned earlier, this will be done by showing the asymptotic convergence of finite-horizon MFE to infinite-horizon non-stationary MFE, without any explicit error bounds (Theorem \ref{thrm:1}). Since the tail of non-stationary MFE can be oscillatory in nature, by analyzing cases in which a non-stationary MFE has a stationary MFE as an accumulation point, we prove that finite-horizon discounted cost MFE can be used to approximate stationary MFE (Proposition \ref{prop:5}).

\end{enumerate}
\begin{figure}[h]
\centering
\resizebox{\linewidth}{!}{%
\begin{tikzpicture}

\node[font=\bfseries, align=center] (qt) at (-6,5.0)
{Quantitative Results\\(Finite state and action spaces)};
\node[font=\bfseries, align=center] (ql) at (6,5.0)
{Qualitative Results\\(Compact Polish state space and compact convex action space)};

\draw[topdash] (-10,4.55) -- (-2,4.55);
\draw[topdash] (  2,4.55) -- (10,4.55);

\draw[vdot] (0,5.2) -- (0,-4.4);

\node[mybox] (L1) at (-7,2.4) {Theorem \ref{thrm:main}:\\Explicit contraction condition\\for finite-horizon MFE.};
\node[mybox, below=18mm of L1] (L2) {Theorem \ref{thrm:7}:\\Finite-time error\\bound between finite\\horizon MFE and infinite\\horizon non-stationary MFE.};
\node[mybox, below=18mm of L2] (L3) {Theorem \ref{thrm:s2}:\\Finite-time error\\bound between infinite\\horizon non-stationary MFE\\and stationary MFE under\\contraction of stationary MFE.};

\draw[downarr,dashed] (L1.south) -- (L2.north);

\node[myoval] (O1) at (0,2.4) {Convergence of\\finite-horizon MFE\\to infinite horizon\\non-stationary MFE.};
\node[myoval, below=20mm of O1] (O2) {Convergence of\\infinite-horizon\\ non-stationary MFE\\to stationary MFE.};

\node[myboxR] (R1) at (7,2.4) {Theorem \ref{thrm:1}:\\Accumulation points of\\finite-horizon MFE\\are infinite-horizon\\non-stationary MFE.};
\node[myboxR, below=26mm of R1] (R2) {Theorem \ref{thrm:q1}:\\If there exists a unique\\policy that corresponds to\\state-measures then\\the tail of infinite-horizon\\non-stationary MFE\\converges to stationary MFE\\if and only if state-measures\\ are convergent.};

\draw[arr, dashed] (L2.east) to[bend left=18]
  node[midway, above, sloped, align=center, font=\bfseries]
  {convergence rate\\available}
  (O1.west);

\draw[arr, dashed] (L3.east) to[bend left=12]
  node[midway, above, sloped, align=center, font=\bfseries]{convergence rate\\available}
  (O2.west);

\draw[arr] (R1.west) to[bend left=12]
  node[midway, above, sloped, font=\bfseries]{asymptotic}
  (O1.east);

\draw[arr, dashed] (R2.west) to[bend left=10]
  node[midway, above, sloped, font=\bfseries]{asymptotic}
  (O2.east);
\end{tikzpicture}
}
\caption{Informal statements of the main results of this work. The dashed lines indicate that the corresponding earlier result provides only a sufficient condition for the stated convergence. In contrast, Theorems \ref{thrm:1} and \ref{thrm:q1} do not require a contraction assumption and are therefore applicable in a broader setting than Theorems \ref{thrm:7} and \ref{thrm:s2}. By combining the convergence of finite-horizon MFEs to infinite-horizon non-stationary MFEs with the convergence of infinite-horizon non-stationary MFEs to stationary MFEs, one also obtains an approximation scheme for stationary MFEs via finite-horizon MFEs.}
\end{figure}
\section{Preliminaries}\label{sect:2}
In this work, we will investigate the relationship between finite-horizon, infinite-horizon non-stationary, and stationary MFGs under discounted cost in discrete-time by means of their MFE. We will adopt the setting introduced in \cite{SaBaRaSIAM}.
\subsection{Finite-Horizon MFGs with Discounted Cost}
We denote a finite-horizon MFG with the tuple $(X,A,c,p,\mu_0,T)$, which we will often denote by MFG$_{\textrm {T}}$, where 
\begin{itemize}
    \item $X$ is a Polish state space,
    \item $A$ is a Polish action space,
    \item $c:X \times A \times \mathcal P(X) \to \mathbb R$ is the one-stage cost function and $p: X \times A \times \mathcal P(X) \to \mathcal P(X)$ denotes the transition probability of the next state given a state-action pair and a state-measure, where $\mathcal P(X)$ is the space of probability measures over the state space $X$,
    \item $\mu_0\in \mathcal P(X)$ is a given initial state-measure,
    \item $T$ represents the length of the horizon of the MFG.
\end{itemize}

To relate finite-horizon equilibria to infinite-horizon equilibria, we will need the discounted cost structure, and thus we will assume that all of our MFG$_{\mathrm T}$ are under discounted cost throughout the paper without mentioning it explicitly. With this convention, in MFG$_{\mathrm T}$, the components of the tuple represent a single-player who seeks to find a flow of policies, $\pmb \pi^T =(\pi^T_t)_{t=0}^T$, $\pi^T_t: X \to \mathcal P(A),$ that minimizes the discounted objective function under a fixed discount factor $0\leq \beta <1$:
\[
J(\pmb \pi^T) = E^{\pmb { \pi}^T}\left[ \sum_{t=0}^{T-1}\beta^tc(x_t,a_t,\mu^T_t)\right],
\]
where $\Pi$ is the space of Markov policies \cite[Proposition 3.2]{SaBaRaSIAM}. The flow $\pmb \mu^{\pmb T} = (\mu^T_t)_{t=0}^{T-1} \in \prod_{t=0}^{T-1} \mathcal P(X) =: \mathcal P(X)^{T}$ satisfies
\[
\mu^T_{t+1}(\cdot) = \int_X p(\cdot|x,a,\mu^T_t)\pi^T_t(da|x)\mu^T_t(dx).
\]
In this model, the evolutions of the states and actions are given by
\[
x(0) \sim \mu_0, \quad x(t) \sim p(\cdot | x(t-1), a(t-1), \mu^T_t), \quad t \geq 1, \quad a(t) \sim \pi^T_t(\cdot | x(t)), \quad t \geq 0.
\] A pair $(\pmb \pi^{\pmb T}, \pmb \mu^{\pmb T})=(\pi^T_t,\mu^T_t)_{t=0}^{T-1}$ that satisfies these properties is referred to as a \emph{mean-field equilibrium} of MFG$_{\mathrm T}$. We will also refer to $(\pmb \pi^{\pmb T},\pmb \mu^{\pmb T})$ as a MFE flow. For the most part, we will be interested in the convergence of the families of joint probability measures $(\pi^T_t \otimes \mu^T_t)_t$, where $\pi^T_t\otimes \mu^T_t(da,dx) := \pi^T_t(da|x)\mu^T_t(dx)$ denotes the joint probability measure constructed using the pair $(\pi^T_t,\mu^T_t)$. In the case of a MFE $(\pmb \pi^{\pmb T}, \pmb \mu^{\pmb T})$, we have $\pmb \pi^{\pmb T} \otimes \pmb \mu^{\pmb T} := (\pi^T_t \otimes \mu^T_t)_{t=0}^{T-1}.$ Clearly, the disintegration of the flow $\pmb \pi^{\pmb T} \otimes \pmb \mu^{\pmb T}$ provides a MFE for MFG$_{\mathrm T}$. If a pair $(\pi^T_t,\mu^T_t)_{t=0}^{T-1}$ induces a MFE, we will call such a pair a \emph{MFE flow}. We will often denote a MFE flow obtained from MFG$_{\mathrm T}$ as $\pmb {\pi^T} \otimes \pmb{\mu^T}$ explicitly when there is potential confusion.
\subsection{Infinite-horizon non-stationary MFGs}
We represent an infinite-horizon non-stationary MFG with the tuple $(X,A,c,p,\mu_0)$, and denote it by MFG$_{\textrm {ns}}$ as a shorthand. The only difference between the infinite-horizon MFGs and those of finite-horizon in the non-stationary case is that we are mainly interested in countable flows $\pmb \pi = (\pi_t)_{t=0}^{\infty}$ and $\pmb \mu = (\mu_t)_{t=0}^{\infty}$ such that $\pmb \pi$ minimizes the objective function
\[
J(\pmb \pi) =E^{\pmb { \pi}}\left[ \sum_{t=0}^{\infty}\beta^tc(x_t,a_t,\mu_t)\right]
\]
and $\pmb \mu$ evolves according to
\begin{equation}\label{eq:stat-evolve}
\mu_{t+1}(\cdot) = \int_X p(\cdot|x,a,\mu_t)\pi_t(da|x)\mu_t(dx).
\end{equation}
Here, the pair $(x_t,a_t)$ admits a similar evolution as in the finite-horizon case.
\subsection{Stationary MFGs}
In the stationary setting, we are interested in time-independent evolutions. For this reason, the system description does not include an initial state-flow $\mu_0$ as the evolution of the population dynamics should be time-independent. Thus, the description will be given by the tuple $(X,A,c,p)$ instead. As a shorthand, we will refer to stationary MFGs as MFG$_{\textrm {s}}$. A mean-field equilibrium in the stationary case is a time-independent tuple $(\pi, \mu)$ such that $\pi$ minimizes the $\beta$-discounted cost
\[
J(\pi) = E^{\pi}\left[ \sum_{t=0}^{\infty}\beta^tc(x_t,a_t,\mu)\right]
\]
and $\mu$ is an invariant distribution
\[
\mu(\cdot) = \int_X p(\cdot|x,a,\mu)\pi(da|x)\mu(dx).
\]
Here, compared to the non-stationary case, we replace $\mu_t$ with $\mu$ in the dynamics and the one-stage cost function for all $t \in \mathbb N,$ which also includes the initial state-measure. In particular, if $(\pi,\mu)$ is a stationary MFE, then $\pmb \pi :=(\pi,\pi,\cdots)$ and $\pmb \mu = (\mu,\mu,\cdots)$ induces an infinite-horizon non-stationary MFE starting from the initial state-measure $\mu$.
\section{Fixed-Point Iteration for MFGs}\label{sect:fixpt}

In this section, we assume that \(X\) and \(A\) are finite. Our objective is to employ contraction arguments to derive quantitative approximation results relating finite-horizon and infinite-horizon MFE. As fixed-point iterations will serve as the main analytical prototype throughout this section, we begin by introducing a fixed-point iteration for \(\mathrm{MFG}_T\) under regularization. Regularization is standard in the recent literature on learning mean field games, as it ensures the Lipschitz continuity of policies, a property that is essential for the convergence analysis of fixed-point schemes \cite{guo2019learning,anahtarci2023q,cui2021approximately}. We also expect that the arguments developed here can be extended to the non-regularized framework considered in \cite{anahtarci2020value}.

We then establish a negative result showing that these techniques do not extend to the infinite-horizon setting in any immediate or straightforward manner. More precisely, we show that the fixed-point iteration associated with a finite-horizon mean field game may be contractive even when the corresponding fixed-point iteration in the infinite-horizon setting is not. In addition, we prove that fixed-point iterations depending only on \(Q\)-functions and state measures yield contraction conditions in the finite-horizon case that differ from those arising in the infinite-horizon setting.

Under additional assumptions guaranteeing contractivity in the infinite-horizon setting, we proceed to derive finite-time error bounds, comparing finite-horizon MFE, infinite-horizon non-stationary MFE, and stationary MFE under regularization. As a consequence of these error estimates, we obtain a new uniqueness condition for infinite-horizon non-stationary MFE corresponding to an arbitrary initial state measure.

\begin{enumerate}
    \item In Subsection \ref{sect:3.1}, we establish a qualitative contraction result for finite-horizon MFGs based exclusively on the iteration of the state measures. The argument relies on representing these state-measure iterations in matrix form. The resulting contraction criterion is then characterized in terms of the spectral radius of an explicitly constructed matrix $A_T$.
    \item In Subsection \ref{sect:b}, we use spectral analysis of the matrix $A_T$ to derive a quantitative contraction criterion that implies the qualitative contraction result established in Subsection \ref{sect:3.1}. Moreover, we show that this criterion is sharp, in the sense that the asymptotes of the spectral radius of $A_T$ converges to the established upper bound.

    \item In Subsection \ref{sect:3.2}, to motivate our quantitative approximation results, we show that the contraction conditions for infinite-horizon non-stationary MFEs and stationary MFEs coincide. {In particular, we will establish a contraction principle for infinite-horizon non-stationary MFG in the space of bounded sequences $(\ell_{\infty},\|\cdot\|_{\infty})$ with an explicit sharp contraction condition by using the spectral analysis of infinite-dimensional Toeplitz operators. It will be shown that this contraction condition (Theorem \ref{prop:2}) is the same one that is available for the stationary MFGs in the literature \cite{anahtarci2023q}, which is a strictly stronger condition than the contraction conditions that we will find for finite-horizon MFG in Subsection \ref{sect:b}.}
    \item In Subsection \ref{sect:3.5}, we derive finite-time error bounds between finite-horizon MFEs and infinite-horizon MFEs under regularization, building on the results of Subsection \ref{sect:b} and the structure of the left Perron eigenvectors of $A_T$ established in Subsection \ref{sect:b}. As a byproduct, we also obtain a uniqueness condition for infinite-horizon non-stationary MFEs that goes beyond the contraction criterion presented in Subsection \ref{sect:3.2}.
\end{enumerate}

\subsection{Fixed-Point Iteration for Finite-horizon Discounted Cost MFGs}\label{sect:3.1}

In this subsection, we develop the tools needed for our quantitative approximation results, closely following \cite{anahtarci2023q}. These results depend on the Lipschitz properties of the maps governing the updates of the $Q$-functions and state measures that characterize finite-horizon MFEs. To this end, we introduce the regularization procedure for discrete-time MFGs developed in \cite{SaBaRaSIAM} and establish the relevant Lipschitz properties needed to control the evolution of the $Q$-functions and state measures. In addition, we present a general qualitative contraction principle for finite-horizon MFEs that will play an important role in our approximation analysis. The sharp asymptotic behavior of this contraction condition requires spectral analysis and is deferred to the next subsection.

Let MFG$_{\mathrm T}$ be the finite-horizon MFG $(X,A,c,p,\mu_0,T).$
We recall that the total variation norm between two probability measures $\mu$ and $\nu$ over $X$ is defined as 
\[
\| \mu - \nu \|_{\mathrm{TV}} := \frac 12 \sum_{x \in X} | \mu(x) - \nu(x)| =: \frac 12 \| \mu - \nu \|_1.
\]Throughout this subsection (and the rest of Section \ref{sect:fixpt}), we will make the following Lipschitz continuity assumption on our system components of MFG$_{\mathrm T}$:
\begin{assumption}\label{assump:1}
    \begin{enumerate}[(a)]
    \item The one-stage reward function \( c \) satisfies the following Lipschitz bound:
    \[
    \left| c(x,a,\mu) - c(\hat{x},\hat{a},\hat{\mu}) \right|
    \leq L_1 \left( {1}_{\{x \neq \hat{x}\}} + 2 \, {1}_{\{a \neq \hat{a}\}} + 2 \, \|\mu - \hat{\mu}\|_{\mathrm{TV}} \right),
    \]
    for all \( x, \hat{x} \in X \), all \( a, \hat{a} \in A \), and all \( \mu, \hat{\mu} \in \mathcal P(X) \).

    \item The stochastic kernel \( p(\cdot | x,a,\mu) \) satisfies the following Lipschitz bound:
    \[
    \left\| p(\cdot | x,a,\mu) - p(\cdot | \hat{x},\hat{a},\hat{\mu}) \right\|_{\mathrm{TV}}
    \leq \frac{K_1}2 \left( 1_{\{x \neq \hat{x}\}} + 2 \, 1_{\{a \neq \hat{a}\}} + 2 \, \|\mu - \hat{\mu}\|_{\mathrm{TV}} \right),
    \]
    for all \( x, \hat{x} \in X \), all \( a, \hat{a} \in A \), and all \( \mu, \hat{\mu} \in \mathcal P(X) \).
\end{enumerate}
\end{assumption}

Since $c$ is continuous over $\mathcal P(X)$, it follows that $c$ is bounded by a constant, say $M$.
For $u \in \mathcal P(A)$, using the transformations,
\[
C(x,u,\mu) := \sum_{a \in A} c(x,a,\mu)u(a)
\]
and
\[
P(\cdot|x,u,\mu):= \sum_{a \in A} p(\cdot|x,a,\mu)u(a),
\]
we can transform the MFG$_{\textrm {T}}$ to $(X,\mathcal P(A),C,P,\mu_0,T)$, which is defined over a compact convex action space $\mathcal P(A)$ that is isomorphic to a closed convex subset of $\mathbb R^{|A|}$. The newly obtained system components $C$ and $P$ satisfy the following Lipschitz conditions \cite[Proposition 1]{anahtarci2023q}:
\begin{lemma}\label{lem:a}
    Under Assumption \ref{assump:1}, \( P \) and \( C \) satisfy the following Lipschitz bounds:

\[
|C(x, u, \mu) - C(\tilde{x}, \tilde{u}, \tilde{\mu})| \leq L_1
\left( 1_{\{x \neq \tilde{x}\}} + \|u - \tilde{u}\|_1 + 2 \, \|\mu - \tilde{\mu}\|_{\mathrm{TV}} \right),
\]

\[
\|P(\cdot|x, u, \mu) - P(\cdot|\tilde{x}, \tilde{u}, \tilde{\mu})\|_{\mathrm{TV}} \leq \frac{K_1}2
\left( 1_{\{x \neq \tilde{x}\}} + \|u - \tilde{u}\|_1 + 2 \, \|\mu - \hat{\mu}\|_{\mathrm{TV}} \right),
\]

for all \( x, \tilde{x} \in X \), \( u, \tilde{u} \in \mathcal P(A) \), and \( \mu, \tilde{\mu} \in \mathcal P(X) \).
\end{lemma}
\begin{proof}
    The bounds follow from \cite[Proposition 1]{anahtarci2023q}.
\end{proof}
To have strongly convex $Q$-functions in the dynamic programming formulation of the objective function (value iteration in the case of stationary MFGs), one often perturbs the cost function $C$ with a $\rho$-strongly convex function $\Omega:\mathcal P(A) \to \mathbb R$ under the $\|\cdot \|_1$ norm i.e., $\Omega(u)-\frac{\rho}2\|u\|^2_1$ is convex over $\mathcal P(A)$. We call the resulting MFG $(X,\mathcal P(A),C+\Omega,P,\mu_0,T)$ a \emph{regularized MFG}. With a slight abuse of notation, by MFG$_{\mathrm T}$ we will denote the regularized MFG $(X,\mathcal P(A),C+\Omega,P,\mu_0,T)$. We refer to \cite{anahtarci2023q} for further details on regularized MFGs. By perturbing the cost function $C$ with $\Omega$, we obtain Lipschitz continuous optimal policies at the cost of a deviation from the MFE of the system $(X,\mathcal P(A),C+\Omega,P,\mu_0,T)$, which will be essential for our analysis as the optimal policies directly enter the state measure updates. In what follows, we will establish a contraction principle to calculate the MFE of the regularized system $(X,\mathcal P(A),C,P,\mu_0,T)$.

To prevent potential confusion regarding our terminology, when we use the terms ``nonnegative'' (resp. ``positive'') in the context of a vector (or a matrix) in this subsection, we will mean that all entries of the vector (or the matrix) are nonnegative (resp. positive).

Our fixed-point iterations will be done by consecutive updates of $Q$-functions and state measures. To handle the updates of the $Q$-functions, for a continuous function $Q$ over $X\times \mathcal P(A)$ and a probability measure $\mu \in \mathcal P(X)$, we let
\begin{equation}\label{eq:g0}
H_{1,t}(Q,\mu)(x,u) := C(x,u,\mu) +\Omega(u) + \beta \sum_{y \in X} \min_{b \in \mathcal P(A)} Q(y,b) \,P(y|x,u,\mu),
\end{equation}
for $T> t\geq 1$ to define a general iteration of $Q$-functions and define $H_{1,0}(Q):= H_{1,1}(Q,\mu_0)$ to account for the evolution at time $t=0,$ and $H_{1,T}(\mu)(x,u) = C(x,u,\mu)+\Omega(u),$ which will determine the value that our value function takes at the terminal time $t=T$. Note that
\[
H_{1,t}(Q,\mu)(x,u) = \langle h(Q,\mu)(x,a) , u \rangle + \Omega (u),
\]
where $h(Q,\mu)(x,a):=c(x,a,\mu)+\beta\sum_{y \in X}\min_{b\in A}Q(y,b)p(y|x,a,\mu)$. Hence, $Q$-functions that we will obtain via the operators $(H_{1,t})_t$ will be $\rho$-strongly convex under the metric $\| \cdot \|_1$;
that is, for any $x \in X$, $H_{1,i}(Q,\mu)(x,\cdot)$ is strongly convex over $\mathcal P(A)$; thus, for any given $x \in X$, any output of $H_{1,i}(Q,\mu)(x,\cdot)$ admits a unique minimizer. In particular, for a $Q$-function obtained from $H_{1,t}$, the optimal policy is uniquely given by $\delta_{\mathrm{argmin}_{b \in \mathcal P(A)} Q(x,b)}(\cdot)$ and is well defined, where $\delta$ denotes to the Dirac measure. For the updates of our state measures, to capture the evolution \eqref{eq:stat-evolve}, we define
\[
H_{2,t}(Q,\mu)(\cdot) := \sum_{x \in X}\int_{\mathcal P(A)} P(\cdot|x,u,\mu)\delta_{\textrm{argmin}_{b \in \mathcal P(A)} Q(x,b)}(du)\mu(x).
\]
for all $t.$
For a given family $\pmb \mu= (\mu_t)_{t=1}^{T-1}$, one can find a family of optimal $Q$-functions that \emph{correspond} to the flow $\pmb{\mu}$ via the recursive relation $Q^{\pmb \mu}_{T-1}:=H_{1,T-1}(\mu_{T-1}),$ and $Q^{\pmb \mu}_t := H_{1,t}(Q_{t+1},\mu_t)$ for all $t$ in the finite-horizon setting. By $Q^{\pmb \mu}=(Q^{\mu}_t)_{t=0}^{T-1}$, we will denote the flow generated by these $Q$-functions. Then, we would like to update the family $(\mu_t)_{t=1}^{T}$ by setting the recursive relation $\hat \mu_{t+1} := H_{2,t+1}(Q^{\pmb {\mu}}_t,\mu_t)$ for all $t=0,1,\cdots,T-1,$ starting from $(\mu^0_t)_{t=1}^{T-1}=(\mu_t)_{t=1}^{T-1}$. For a flow $\pmb{\mu}^T = (\mu_t^T)_t \in \{\mu_0\}\times \mathcal P(X)^{T-1}$, we will denote this iteration by the operator 
\[
\pmb H_T(\pmb \mu^T)=(\mu_0,H_{2,1}(Q^{\pmb \mu^T}_0,\mu_0),H_{2,2}(Q^{\pmb \mu^T}_1,\mu^T_1),\cdots,H_{2,T-1}(Q^{\pmb \mu^T}_{T-2},\mu_{T-2}))
\]
for shorthand.

\begin{figure}
\begin{tikzcd}[column sep=large, row sep=large]
\mu_0 \arrow[dr,"H_{1}"] \arrow[ddr, bend right=35,"H_{2}"] & \mu_1 \arrow[dr,"H_{1}"] \arrow[ddr, bend right=35,"H_2"] & \mu_2 \arrow[dr,"H_1"] \arrow[ddr, bend right=35,"H_2"] & \mu_3 \arrow[dr,"H_1"] \arrow[ddr, bend right=35,"H_2"] & \cdots \arrow[dr,"H_1"] \arrow[ddr, bend right=35,"H_2"]& \mu_{T-1} \arrow[dr,"H_{1}"] & \mathrm{Input}\\
& Q^{\pmb \mu}_0 \arrow[r,leftarrow,"H_1"] \arrow[d,"H_2"] & Q^{\pmb \mu}_1 \arrow[r,leftarrow,"H_1"] \arrow[d,"H_2"] & Q^{\pmb \mu}_2 \arrow[r,leftarrow,"H_1"] \arrow[d,"H_2"] & \cdots \arrow[r,leftarrow,"H_1"] & Q^{\pmb \mu}_{T-2} \arrow[d,"H_2"]\arrow[r,leftarrow,"H_1"]& Q^{\pmb \mu}_{T-1} \\
\mu_0 & \mu^{\pmb{new}}_1 & \mu^{\pmb{new}}_2 & \mu^{\pmb{new}}_3 & \cdots & \mu^{\pmb{new}}_{T-1} & \mathrm{Output} \arrow[uu, bend right=35]
\end{tikzcd}
\caption{A diagram that summarizes the iterations procedure that we use while defining the map $\pmb H_T$. We suppressed the index $t$ in the maps $H_1$ and $H_2$.}
\label{fig:h}
\end{figure}

\begin{remark}
By a possible abuse of notation, since the initial state-measure $\mu_0$ is fixed, we will ignore the first component of $\pmb H_T$ in what follows.    
\end{remark}

Our aim is to establish a criterion that guarantees the convergence of a family $\{(\mu^{T,k}_t)_{t=1}^{T-1}:k \in \mathbb N\}$ as $k\to\infty$ obtained from recursions of $\pmb H_T$ to some $\pmb {\tilde \mu}^T = (\tilde \mu^T_t)_{t=1}^{T-1}$ that satisfies the property $\tilde \mu_{t}^T = H_{2,t}(Q^{\pmb {\tilde \mu}^T}_{t-1},\tilde \mu_{t-1}^T)$ for all $t = 1,\cdots,T-1$, which in turn will give us a finite-horizon MFE. To achieve this, our aim is to show that $\pmb H_T$ is a contraction operator. For this, we will equip $\mathcal P(X)^{T-1} := \prod_{t=1}^{T-1}\mathcal P(X)$ with the following norm:
\[
\| \pmb \mu^T- \pmb {\tilde \mu}^T \|_{\mathrm {T-1},\mathrm{TV}} := \sum_{i=1}^{T-1}\| \mu_i^T -\tilde \mu_i^T \|_{\mathrm{TV}},
\]
which metrizes the product topology.
It is easy to see that the convergence in the norm $\|\cdot \|_{T,TV}$ is equivalent to the convergence of the vectors
\[ V(\pmb \mu^T, \pmb{\tilde \mu}^T) := (\| \mu_1^T-\tilde\mu_1^T \|_{\mathrm{TV}},\| \mu_2^T-\tilde \mu^T_2 \|_{\mathrm{TV}},\cdots,\|\mu_{T-1}^T-\tilde\mu^T_{T-1}\|_{\mathrm{TV}}),
\]
in a norm over $\mathbb R^{T-1}$, since $\| \pmb \mu^T- \pmb {\tilde \mu}^T \|_{\mathrm{T-1},\mathrm{TV}}$ is just $V(\pmb \mu^T, \pmb{\tilde \mu}^T)$ evaluated under the $1-$norm over $\mathbb R^{T-1}$. In general, for a vector norm over $\mathbb R^{T-1}$, say $\| \cdot \|$, we also have that $\| V(\cdot,\cdot) \|$ is a norm over $\mathcal P(X)^{T-1}$, which will be useful in the following discussion to determine a norm over $\mathcal P(X)^{T-1}$ that will yield an improved contraction property.

First, we will calculate the Lipschitz coefficients that arise from the variations of $(H_{2,t})_t$ over different flows of state-measures, which will heavily depend on the inequalities established in \cite{anahtarci2023q} for the stationary setting. We will adjust them to the finite-horizon setting. To achieve this, we will first identify a compact subset of the space of functions in which our $Q$-functions will live on. 

\begin{lemma}\label{lem:f0}
    Suppose Assumption \ref{assump:1} holds. Let $\pmb{\mu}^T=(\mu^T_t)_{t=0}^{T-1}\in \mathcal \{\mu_0\}\times \mathcal P(X)^{T-1}$. For $1 \leq t <T-1$, for $Q$-functions over $X \times \mathcal P(A)$ that are $\frac{L_1}{1-\frac{\beta K_1}2}$ Lipschitz over $X$ under any action $u \in \mathcal P(A),$ we have
    \[
    \sup_{u \in \mathcal P(X)} | H_{1,t}(Q,\mu^T_t)(x,u) - H_{1,t}(Q,\mu^T_t)(\tilde x,u) | \leq \frac{L_1}{1-\frac{\beta K_1}2} \, 1_{\{x \not = \tilde x\}},
    \]
    and 
    \[
    \sup_{u \in \mathcal P(A)}| H_{1,T-1}(\mu^T_{T-1})(x,u) - H_{1,T-1}(\mu^T_{T-1})(\tilde x,u) | \leq L_1\,1_{\{x \not = \tilde x\}}.
    \]
    Furthermore
    \begin{align*}
    &\| \mathrm{argmin} _{u \in \mathcal P(A)}H_{1,t}(Q,\mu^T_t)(x,u)-\mathrm{argmin}_{u \in \mathcal P(A)}H_{1,t}(\tilde Q,\tilde \mu^T_t)(\tilde x,u) \|_{1} 
    \\&\leq \frac{\beta}{\rho} \| Q -\tilde Q\|_{\infty} + \frac{L_1}{\rho(1-\frac{\beta K_1}2)}\left(1_{\{x \not = x\}} + 2 \,  \| \mu^T_t - \tilde \mu^T_t\|_{\mathrm{TV}} \right)
    \end{align*}
\end{lemma}
\begin{proof}
    See Appendix \ref{sect:lem:f0}.
\end{proof}

{\begin{remark}\label{rem:discount}
For any $T \in \mathbb N$, and $\pmb \mu^T \in \prod_{t=0}^{T-1}\mathcal P(X)$, the quantity $\frac{L_1}{1-\frac{\beta K_1}{2}}$ serves as a uniform Lipschitz bound for the family $(Q_t^{\pmb \mu^T})_t$ over $X$. To obtain this Lipschitz condition, we use the Lipschitz properties of $C$ and $P$ over $X$. For this bound to be well defined, one requires $\frac{\beta K_1}{2}<1$. Since $|X|<\infty$, however, this condition is generally not restrictive because the total variation norm is bounded above by $1$, and one may always choose $\frac{K_1}{2}\leq 1$. To be precise, we can separate the Lipschitz condition for $P$ we have established in Lemma \ref{lem:a} as
\begin{align}\label{eq:r1}
\|P(\cdot|x, u, \mu) - P(\cdot|\tilde{x}, \tilde{u}, \tilde{\mu})\|_{\mathrm{TV}} \leq \frac{K_1}2 1_{\{x \neq \tilde{x}\}} + \frac{\hat K_1}2 \left(\|u - \tilde{u}\|_1 + 2 \, \|\mu - \hat{\mu}\|_{\mathrm{TV}} \right).
\end{align}
Since $\|P(\cdot|x, u, \mu) - P(\cdot|\tilde{x}, \tilde{u}, \tilde{\mu})\|_{\mathrm{TV}} \le 1$, we always have $\frac{K_1}{2}\le 1$.

In the borderline case $\frac{K_1}{2}=1$, it is further necessary to assume $\beta<1$ in order to obtain a well defined Lipschitz constant that is uniform in $T$, which is the motivation behind the discounted structure used in this subsection. Thus, by assuming $\frac{K_1}2<1$, one can also obtain similar results for non-discounted finite-horizon MFG. To make the presentation of terms we obtain less cumbersome we will not do the separation of the Lipschitz coefficients as done in \eqref{eq:r1}; instead, we will simply suppose that $\beta<1$ throughout this subsection.
\end{remark}
}

\begin{remark}
 Let $\pmb{\mu}^T\in \{\mu_0\}\times \mathcal P(X)^{T-1}$. Then, Lemma \ref{lem:f0} implies that the $Q$-functions $(Q^{\pmb \mu^T}_t)_t$ satisfy the Lipschitz property
 \[
|Q^{\pmb \mu^T}_t(x,u)-Q^{\pmb \mu^T}_t(\tilde x,u)| \le \frac{L_1}{1-\frac{\beta K_1}2}\,1_{\{x \not = \tilde x\}}
 \]
 for $0 \le t <T-1,$ and
 \[
|Q^{\pmb \mu^T}_{T-1}(x,u)-Q^{\pmb \mu^T}_{T-1}(\tilde x,u)| \le L_1 \,1_{\{x \not = \tilde x\}}.
 \]
\end{remark}

We next establish Lipschitz properties of the maps $Q^{\pmb \mu^T}_t(=H_{1,t}(Q^{\pmb \mu^T}_{t+1},\mu^T_t))$, which is needed as our iterations of $H_{2,t}$ depend of these $Q$-functions. Before proceeding to the statement, we define the constant $\bar L= \frac{L_1}{1-\frac{\beta K_1}{2}}$.

\begin{lemma}\label{lem:f2}
    Suppose Assumption \ref{assump:1} holds. Let $\pmb {\tilde \mu}^T,\pmb {\mu}^T \in \{\mu_0\}\times \mathcal P(X)^{T-1}$. Recall that for all $0\le t \le T$, we have $Q^{\pmb \mu^T}_t=H_{1,t}(Q^{\pmb \mu^T}_{t+1},\mu_t^T)$ and $Q^{\tilde{\pmb \mu}^T}_t=H_{1,t}(Q^{\pmb {\tilde \mu}^T}_{t+1},\tilde \mu_t^T)$. Then, we have
    \[
    \| Q^{\pmb {\mu}^T}_0 - Q^{\pmb {\tilde \mu}^T}_0\|_{\infty} \leq \beta \| Q^{\pmb {\mu}^T}_1 - Q^{\pmb {\tilde \mu}^T}_1 \|_{\infty},
    \]
    \[
    \| Q^{\pmb { \mu}^T}_t - Q^{\pmb {\tilde \mu}^T}_t \| _{\infty} \leq 2\bar L\| \mu_t^T - \tilde \mu_t^T \|_{\mathrm{TV}} + \beta \| Q^{\pmb {\mu}^T}_{t+1} - Q^{\pmb {\tilde \mu}^T}_{t+1} \|_{\infty},
    \]
    for $T-1>t>0,$
    and
    \[
    \| Q^{\pmb {\mu}^T}_{T-1} - Q^{\pmb {\tilde \mu}^T}_{T-1} \|_{\infty} \leq 2L_1\| \mu_{T-1} - \tilde \mu_{T-1} \|_{\mathrm{TV}}.
    \]
\end{lemma}
\begin{proof}
    See Appendix \ref{sect:lem:f2}.
\end{proof}

 Let $\bar K =\frac{3K_1}{2}+\frac{K_1 \bar L}{2\rho(1-\beta)}.$ With this notation, using Lemma \ref{lem:f2}, we obtain Lipschitz properties of the maps $(H_{2,t})_t$ under the $Q$-functions $Q^{\pmb \mu}$.

\begin{lemma}\label{lem:f1}
    Suppose Assumption \ref{assump:1} holds. Let $\pmb {\tilde \mu},\pmb {\mu} \in \{\mu_0\}\times \mathcal P(X)^{T-1}$. Then, we have
    \[
    \| H_{2,1}(Q^{\pmb { \mu}^T}_0,\mu_0) - H_{2,1}(Q^{\pmb {\tilde \mu}^T}_0,\mu_0) \|_{\mathrm{TV}} \leq \frac {K_1}{2\rho}\| Q^{\pmb {\mu}^T}_0 - Q^{\pmb {\tilde \mu}^T}_0 \|_{\infty},
    \]
    and
    \[
        \| H_{2,t}(Q^{\pmb {\mu}^T}_{t-1},\mu^T_{t-1}) - H_{2,t}(Q^{\pmb {\tilde \mu}^T}_{t-1},\tilde \mu_{t-1}^T) \|_{\mathrm{TV}} \leq \frac {K_1}{2\rho}\| Q^{\pmb {\mu}^T}_{t-1} - Q^{\pmb {\tilde \mu}^T}_{t-1} \|_{\infty} + \bar K\| \mu_{t-1}^T - \tilde \mu_{t-1}^T\|_{\mathrm{TV}}
    \]
    for $T-1\geq t \ge 2$.
\end{lemma}
\begin{proof}
    See Appendix \ref{sect:lem:f1}.
\end{proof}

The next lemma is the key observation for our (eventual) contraction result for $\pmb H_T$ over $(\mathcal P(X)^{T-1},\|\cdot\|_{\mathrm T,\mathrm{TV}})$, which essentially combines the variations of the state-flow measures we have obtained under $\pmb H_T$ in Lemmas \ref{lem:f2} and \ref{lem:f1} in a matrix inequality format.

\begin{lemma}\label{lem:f4}
     Suppose Assumption \ref{assump:1} holds. Recall that $\bar L = \frac{L_1}{1-(\beta K_1/2)}$ and $\bar K =\frac{3K_1}{2}+\frac{K_1 \bar L}{2\rho(1-\beta)}.$ With this notation, for any $\pmb \mu^T, \pmb {\tilde \mu}^T \in \{\mu_0\}\times \mathcal P(X)^{T-1}$, we have the following matrix inequality
    \begin{align}
    V( \pmb H_T(\pmb \mu^T) , \pmb H_T(\pmb {\tilde \mu}^T) ) \leq A_T V(\pmb \mu^T , \pmb {\tilde \mu}^T),
    \end{align}
    where the inequality is defined term by term, and the $T \times T$ matrix $A_T$ is given by
    \small
    \begin{equation*}
    A_T = \begin{bmatrix}
        \frac{\bar LK_1}{\rho}\beta  & \frac{\bar LK_1}{\rho}\beta^2 & \cdots  & \frac{\bar LK_1}{\rho} \beta^{T-1} \\
        \bar K+\frac{\bar LK_1}{\rho} & \frac{\bar LK_1}{\rho}\beta & \cdots & \frac{\bar LK_1}{\rho} \beta^{T-2} \\
        0 & \bar K +\frac{\bar LK_1}{\rho} & \cdots & \frac{\bar LK_1}{\rho} \beta^{T-3} \\
        \vdots & \vdots & \ddots & \vdots \\
        0 & 0 & \cdots & \frac{\bar LK_1}{\rho}\beta
    \end{bmatrix}.
    \end{equation*}
\end{lemma}
\begin{proof}
    By Lemmas \ref{lem:f2} and \ref{lem:f1}, we have
    \begin{align*}
    \|H_{2,1}(Q^{\pmb \mu^T}_0,\mu_0) - H_{2,1}(Q^{\pmb{\tilde \mu}^T}_0,\mu_0)\|_{\mathrm{TV}} 
    &\leq \frac{K_1}{2\rho}\beta \| Q^{\pmb \mu^T}_1
     - Q^{\pmb{ \tilde \mu }^T}_1 \|_{\infty} 
     \\& \leq \frac{K_1}{\rho}\bar L \beta \|\mu_1^T - \tilde \mu_1^T \|_{\textrm{TV}} + \frac{K_1}{2\rho}\beta^2\| Q^{\pmb \mu^T}_2
     - Q^{\pmb{ \tilde \mu }^T}_2 \|_{\infty}
     \\& \qquad \vdots
     \\& \leq \frac{K_1}{\rho}\bar L\sum_{i=1}^{T-1}\beta^i\|\mu_i^T - \tilde \mu_i^T\|_{\mathrm{TV}} + \frac{K_1}{2\rho}\beta^T \| Q^{\pmb \mu^T}_T
     - Q^{\pmb{ \tilde \mu }^T}_T\|_{\infty}
     \\& \leq \frac{K_1}{\rho}\bar L\sum_{i=1}^{T-1}\beta^i\|\mu_i^T - \tilde \mu_i^T\|_{\mathrm{TV}} + \frac{K_1}{\rho}\bar L_1\beta^{T-1}\|\mu_{T-1}^T - \tilde \mu_{T-1}^T\|_{\mathrm{TV}},
    \end{align*}
    and
    \begin{align*}
        &\| H_{2,t}(Q^{\pmb {\mu}^T}_{t-1},\mu_{t-1}^T) - H_{2,t}(Q^{\pmb {\tilde \mu}^T}_{t-1},\tilde \mu_{t-1}^T) \|_{\mathrm{TV}} 
        \\& \leq \left(\bar K +\frac{K_1}{\rho}\bar L\right)\|\mu_{t-1}^T -\tilde \mu_{t-1}^T\|_{\mathrm{TV}} + \frac{K_1}{2\rho}\beta \| Q^{\pmb {\mu}^T}_{t} - Q^{\pmb {\tilde \mu}^T}_{t} \|_{\infty}
        \\& \vdots
        \\& \leq \left(\bar K +\frac{K_1}{\rho}\bar L\right)\|\mu_{t-1}^T -\tilde \mu_{t-1}^T\|_{\mathrm{TV}} + \frac{K_1}{\rho}\bar L\sum_{i=t}^{T-2}\beta^{i-t+1}\|\mu_i^T -\tilde \mu_i^T \|_{\mathrm{TV}} 
        \\& \qquad + \frac{K_1}{\rho}L_1 \beta ^{T-t} \|\mu_{T-1}^T - \tilde \mu^T_{T-1}\|_{\mathrm{TV}}.
    \end{align*}
    Using the inequalities established above, {and the bound $L_1 \le \bar L$,} we directly obtain the desired result.
\end{proof}

Our contraction result for $\pmb H_T$ relies on the spectral properties of the matrix $A_T$. To this end, before presenting this contraction result, we recall some basic results and concepts from the spectral theory of finite matrices that will be used in our proof.

It is straightforward to see that the $(T-1)^{\textrm{th}}$ power of $A_T$, $A^{T-1}_T,$ has all strictly positive entries. Hence, $A_T$ is a nonnegative irreducible matrix. Thus, the Perron-Frobenius theorem is applicable to $A_T$, meaning that the largest eigenvalue of $A_T$ is positive. By $\rho(A_T)>0,$ let us denote the largest eigenvalue of $A_T$ (in magnitude), i.e. the \emph{spectral radius} of $A_T$. As a consequence of Gelfand's formula \cite[Corollary 5.6.14]{horn2012matrix}, we have that $\lim_{n \to \infty} \|A^n_T \|^{1/n}_{\mathrm{op}} = \rho(A_T),$ where $\| \cdot \|_{\mathrm{op}}$ can be an arbitrary operator norm induced by some vector norm over the Euclidean space $\mathbb R^{T-1}$. However, this convergence is only asymptotic for most operator norms. For instance, under the operator norm generated by the $\ell_2$-norm in $\mathbb R^{T-1}$, $\|\cdot\|_{2,\mathrm{op}}$, we have $\| A^n_T\|^{1/n}_{2,\mathrm{op}} > \rho(A_T)$ for any $n$ as $A_T$ has strictly positive entries on the upper triangular part, cf. \cite[Theorem]{goldberg}.

 As a consequence of the Perron-Frobenius theorem for nonnegative irreducible matrices \cite[Theorem 8.4.4]{horn2012matrix}, the matrix $A_T$ admits a unique positive (left and right) eigenvector (up to a constant multiplicity) that corresponds to its largest eigenvalue, which is also unique on its own and corresponds to $\rho(A_T)$. The same also holds for the transpose of $A_T$, $A^*_T$. Since the largest eigenvalue of $A_T$ is a real number, it is also an eigenvalue of $A^*_T$, the transpose of $A_T$.

By $r$ we denote a corresponding positive right eigenvector to $\rho(A_T)$ for the matrix $A^*_T$, which are also called a Perron eigenvector. Note that the map $(\pmb \mu^T,\pmb{\tilde \mu}^T) \mapsto \langle r , V( \pmb \mu^T, \pmb{\tilde \mu}^T)   \rangle =: \|  V(\pmb \mu^T,\pmb{\tilde \mu}^T)\|_r$ is a norm over $\mathcal P(X)^{T-1}$. Since $r$ is a positive vector, for any $\pmb \mu, \tilde{\pmb \mu} \in \mathcal P(X)^{T-1}$, we have the following:
\[
\langle r, V(\pmb H_T(\pmb \mu) , \pmb H_T(\pmb {\tilde \mu}) ) \rangle \leq \langle r, A_T V( \pmb \mu , \pmb {\tilde \mu}) \rangle \leq \langle A^*_Tr, V( \pmb \mu , \pmb {\tilde \mu}) \rangle \le\rho(A_T) \langle r, V(\pmb \mu , \pmb {\tilde \mu} ) \rangle.
\]
This will allow us to show that the operator $\pmb H_T$ is a contraction with contraction rate $\rho(A_T)$. When $\rho(A_T)=1$, using the norm $\| V(\cdot,\cdot)\|_{\mathrm r}$ over $\mathcal P(X)^{T-1}$, we can still have convergence of iterations of the matrix $A_T$, when $\mathcal P(X)$ is a compact subset of a Banach space, using Ishikawa's theorem \cite{MR412909}, for which we will need the norm  $\| V(\cdot,\cdot)\|_{\mathrm r}$. However, we note that, even though $\| A_{T} \|_{\mathrm {op}} \geq \rho(A_T)$ always holds for an arbitrary matrix norm $\| \cdot \|_{\mathrm {op}},$ it does not necessarily satisfy the relation $\| A_{T} \|_{\mathrm {op}} = \rho(A_T) = 1$ as we have noted above. Thus, as Gelfand's formula will be strictly asymptotic at most for most norms, we will not be able to show the convergence of iterations of the operator $\pmb H_T$ due to lack of non expansiveness under arbitrary matrix norms.

With these observations, we now state our main result of this subsection, which provides a sufficient convergence condition for the iterations of the state-flows in finite-time horizon to converge to an MFE.

\begin{theorem}\label{thrm:f2}
    Suppose $\beta <1$. If $\rho(A_T)<1$, then iterations of the operator $\pmb H_T$ in $\{\mu_0\}\times \mathcal P(X)^{T-1}$ converge to a unique fixed point in $(\{\mu_0\}\times \mathcal P(X)^{T-1},\| V(\cdot,\cdot)\|_{\mathrm r})$. If $\rho(A_T) =1$, then iterations of the operator $\lambda \pmb I + (1-\lambda) \pmb H_T$ converge to a fixed point of $\pmb H_T$ for all $\lambda \in (0,1)$ in $(\{\mu_0\}\times \mathcal P(X)^{T-1},\| V(\cdot,\cdot)\|_{\mathrm r})$, where $\pmb I$ is the identity operator over $\{\mu_0\}\times \mathcal P(X)^{T-1}$.
\end{theorem}
\begin{proof}
    We will consider the case $\rho(A_T)<1$ first. Let $\pmb{\mu},\pmb{\tilde \mu} \in \{\mu_0\}\times\mathcal P(X)^{T-1}$. Using Lemma \ref{lem:f4} and the discussion above, we have
    \begin{align*}
    \langle r , V( \pmb{H}_T(\pmb{\mu}^T) , \pmb{H}_T(\pmb{\tilde \mu}^T) )\rangle &
    \leq \langle r , A_T V( \pmb{\mu}^T , \pmb{\tilde \mu}^T )\rangle = \rho(A_T) \langle r , V( \pmb{\mu}^T , \pmb{\tilde \mu}^T )\rangle.
    \end{align*}  
    Thus, using the Banach contraction mapping theorem with the norm $\| V(\cdot,\cdot) \|_{\mathrm r}$ over $\{\mu_0\}\times \mathcal P(X)^{T-1}$, and the operator $\pmb H_T,$ we see that $\pmb H_T$ has a unique fixed point and its iterations converge to that fixed point. The uniqueness of the fixed point also follows from the Banach fixed-point theorem in this case.

    Let us consider the case $\rho(A_T)=1$ now. We will use \cite[Theorem 1]{MR412909} to establish the convergence of the operator $\lambda \pmb I + (1- \lambda)\pmb H_T,$ for which we need to justify the existence of a fixed point first. As a consequence of Assumption \ref{assump:1}-(c), for all $\pmb \mu^T$, for all $t,$ $Q^{\pmb \mu^T}_t$ is bounded by $\frac{M}{1-\beta}$. Since $X$ and $A$ are discrete, it also follows that all such $Q^{\pmb \mu^T}_t$ are also Lipschitz continuous with the same Lipschitz constants. Thus, for all $\pmb {\mu}^T$, the families $(Q^{\pmb \mu^T}_t)_{t=0}^T$ all belong to a compact set of functions as a consequence of the Arzel\`a-Ascoli theorem. The space $\mathcal P(X)$ is compact under the total variation norm as $X$ is a finite space. Therefore, since the operator $\pmb H_T$ is continuous, and maps a compact set to itself, it admits a fixed point by Schauder's fixed point theorem. 
    
    Since the space of finite signed measures is a Banach space under the total variation norm, we have that $\mathcal P(X)$ is a compact subset of the space of finite signed measures when $X$ finite. Let $\mathcal M(X)$ denote the space of finite signed measures. Consequently, $\prod_{t=1}^{T-1} \mathcal M(X)$ is also a Banach space under the norm $\| V(\cdot,\cdot)\|_r.$ It then follows that $\mathcal P(X)^{T-1}$ is a compact subset of a Banach space under the norm $\| V(\cdot,\cdot)\|_{\mathrm r}.$ Then, applying \cite[Theorem 1]{MR412909} to $\lambda \pmb I + (1-\lambda) \pmb H_T$ under the norm $\| V(\cdot,\cdot)\|_{\mathrm r}$, we see that iterations of the operator $\lambda \pmb I + (1-\lambda) \pmb H_T$ converge to a fixed point of it under $\| V(\cdot,\cdot)\|_{\mathrm r}.$ The fixed point of $\lambda \pmb I + (1-\lambda) \pmb H_T$ also gives us a fixed point for $\pmb H_T$, and this completes the proof.
\end{proof}

Since we endow each section of the product space $\mathcal P(X)^{T-1}$ with the norm $\| \cdot \|_{\mathrm{TV}}$, and that every norm on $\mathbb R^{T-1}$ is equivalent to every other norm, we conclude convergence of the iterations of $\pmb H_T$ taken in the norm $\| V(\cdot,\cdot)\|$, for an arbitrary norm $\| \cdot \|$ on $\mathbb R^{T-1}$, also ensures the convergence of the iterations of $\pmb H_T$ (resp. $\lambda \pmb I+(1-\lambda)\pmb H_T$) whenever $\rho(A_T) <1$ (resp. $\rho(A_T)=1$). Recall that if $\pmb \mu^T \in \{\mu_0\}\times \mathcal P(X)^{T-1}$ is a fixed point of $\pmb H_T$, then $\pmb \mu^T$ induces a finite-horizon MFE. Thus, the condition $\rho(A_T)<1$ also implies that there exists a unique finite-horizon MFE for the system $(X,\mathcal P(A),C+\Omega,P,\mu_0,T)$ for any given $\mu_0 \in \mathcal P(X)$.

Theorem \ref{thrm:f2} is qualitative in nature and, as stated, does not provide enough information to determine whether the contraction result in the finite-horizon case is weaker than in the infinite-horizon case. In the next subsection, we study the spectral properties of $A_T$ in order to derive sharp quantitative asymptotics for $\rho(A_T)$.

\subsection{Spectral Properties of $A_T$}\label{sect:b}

In the previous subsection, we showed that $\rho(A_T)<1$ serves a qualitative contraction condition for $\pmb H_T$, whose fixed points induce a finite-horizon MFE to the regularized MFG $(X,\mathcal P(A),C+\Omega,P,\mu_0,T)$. In this subsection, we will provide tight lower and upper bounds for $\rho(A_T)$ for sufficiently large $T$. Asymptotic behavior of $\rho(A_T)$ and the structure of Perron eigenvectors of $A_T$ will be essential in our finite-time error bounds between finite-horizon and infinite-horizon MFE. Furthermore, we will provide a characterization for the positive eigenvectors that correspond to $\rho(A_T)$, which will be used for constructing Lyapunov functions to obtain the rate of convergence  between finite-horizon MFE and infinite-horizon MFE.  Our finite-time error bounds will also provide us a new uniqueness condition for infinite-horizon non-stationary MFE; thus, it is desirable to have a sharp quantitative description of the asymptotes of $(\rho(A_T))_T$.

The main result of this subsection is the following result, providing tight bounds on $\rho(A_T)$:
\begin{theorem}\label{thrm:main}
Suppose that Assumption \ref{assump:1} holds. Suppose further that $\beta<1$. Then, the following results hold:
\begin{itemize}
    \item[a)]  For any sufficiently large $T$, we must have
\begin{align*}
    &\sqrt{ \hat K \beta} \cos \left( \frac{\pi}{T} \right) + \sqrt{ \left( \hat K\cos^2\left( \frac{\pi}{T} \right)-\bar K\right) \beta} \leq \sqrt{\rho(A_T)} 
    \\&\leq \sqrt{ \hat K \beta} \cos \left( \frac{\pi}{2T-1} \right) + \sqrt{ \left( \hat K\cos^2\left( \frac{\pi}{2T-1} \right)-\bar K\right) \beta},
\end{align*}
{where $\hat K :=\bar K + \frac{K_1}{ \rho} \bar L$.}
Moreover, for any sufficiently large $T$, for some $\theta_T \in (0,2\pi)$ we have $\rho(A_T) = -\bar K \beta + 2 \sqrt{\rho(A_T) \beta \hat K} \cos( \theta_T).$ Then, for any sufficiently large $T,$ we must have
\[
\cos\left( \frac{\pi}{T}\right) \leq \cos( \theta_T) \leq \cos\left( \frac{\pi}{2T-1}\right).
\]
\item[b)] Let $\rho(A_T) = -\bar K \beta + 2 \sqrt{\rho(A_T) \beta \hat K} \cos( \theta_T).$ Then, for sufficiently large $T$, the right Perron eigenvectors of $A$ are of the form $v=(v_1,\cdots,v_T)$, where
\[
v_j = \left( \sqrt\frac{\hat K}{\rho(A_T)\beta}\right)^j \sin\left(j\theta_T\right).
\]
\end{itemize}
\end{theorem}

The proof of Theorem \ref{thrm:main} is based on a contradiction argument and therefore establishes only the existence of a threshold $T$ for which the above bounds hold. The main steps are as follows. First, we construct a perturbed tridiagonal matrix that has $\rho(A_T)$ as an eigenvalue and shares the same Perron eigenvectors as $A_T$. Next, using the recursion relations induced by this matrix, we derive a polynomial that characterizes its eigenvalues, and then analyze the asymptotic behavior of $\rho(A_T)$ through this polynomial. The construction of this polynomial depends essentially on the corresponding Perron eigenvectors of $A_T$.

Before turning to the proof of Theorem \ref{thrm:main}, we make several observations about its implications. We begin by comparing Theorem \ref{thrm:main} with the contraction conditions for stationary MFGs. To this end, we first establish some simple upper bounds on $\rho(A_T)$. Applying Gershgorin's Circle Theorem \cite[Corollary 6.1.5]{horn2012matrix} to $A_T$, one can obtain the following bound on the spectral radius of $A_T$:
\begin{equation}\label{eq:ref1}
\rho(A_T) \leq \max \left(\bar K + \frac{K_1}{\rho}\bar L \frac{1-\beta^{T-1}}{1-\beta}+\frac{K_1}{\rho}\bar L \beta^T, \frac{K_1}{\rho}\bar L \frac{1-\beta^{T}}{1-\beta}+\frac{K_1}{\rho}\bar L \beta^T\right),
\end{equation}
which shows that the contraction rate $\rho(A_T)$ in the worst case is comparable to the ones used in the literature for stationary MFE \cite{anahtarci2023q}.
Indeed, when $T$ is sufficiently large, since $\beta <1,$ we obtain the following time-horizon independent bound on $\rho (A_T)$ for all $T$:
\begin{equation}\label{eq:ref2}
\rho(A_T) < \bar K + \frac{K_1}{\rho}\frac{\bar L}{1-\beta}.
\end{equation}
In general, Gershgorin's Circle Theorem is tight in the sense that the bound we have found in \eqref{eq:ref2} can be asymptotically optimal for some parameters $\bar K,\bar L,K_1,\rho,\beta$. However, Theorem \ref{thrm:main} shows that the bound \eqref{eq:ref1} is usually not sharp. {Indeed, Theorem \ref{thrm:main} implies that
\begin{equation}\label{eq:ref3}
\limsup_{T \to \infty} \sqrt{\rho(A_T)} = \sqrt{\beta}\left(\sqrt{\hat K}+\sqrt{\hat K-\bar K}\right).
\end{equation}
Combining \eqref{eq:ref2} and \eqref{eq:ref3}, we obtain
\[
\beta\left(\sqrt{\hat K}+\sqrt{\hat K-\bar K}\right)^2 \le \bar K + \frac{K_1}{\rho}\frac{\bar L}{1-\beta},
\]
which can be shown to be strict in most cases (see Remark \ref{rem:1} below for an explicit example).}

As a consequence of \eqref{eq:ref2}, we have the following corollary, which could have been obtained if we had treated the vectors in Lemma \ref{lem:f4} under the max norm.

\begin{corollary}\label{cor:ask}
    Suppose that Assumption \ref{assump:1} holds. If $\bar K + \frac{K_1}{\rho}\frac{\bar L}{1-\beta} \leq 1$, then there exists a unique fixed point of $\pmb H_T$ in $\{\mu_0\}\times\mathcal P(X)^{T-1}$ for all sufficiently large $T$ values.
\end{corollary}
\begin{proof}
    The result directly follows from Theorem \ref{thrm:f2} and \eqref{eq:ref2}.
\end{proof}

Thus, for sufficiently large $T$, the map $\pmb H_T$ is a contraction under the condition $ \bar K + \frac{K_1}{\rho}\frac{\bar L}{1-\beta} \leq 1$, which is the same contraction condition for the mean-field equilibrium operator for stationary MFE presented in \cite[Theorem 1]{aydin2023robustness}. In Subsection \ref{sect:3.2}, we will also show that fixed-point iteration algorithm with our methods results in the contraction condition $\bar K + \frac{K_1}{\rho}\frac{\bar L}{1-\beta} <1$ for the infinite-horizon non-stationary MFG. As expected, Corollary \ref{cor:ask} combined with Theorem \ref{thrm:main} shows that the contraction condition we have in finite-horizon games is more relaxed than that in the infinite-horizon case.

As stated, Theorem \ref{thrm:main} only provides the aforementioned bounds for sufficiently large $T$. The last observation that we will make is that the asymptotic bounds obtained via Theorem \ref{thrm:main} are indeed applicable for all $T$.

We will make some simple observations on the behavior of the map $T \mapsto\rho(A_T)$ to make a conclusion about the contractivity of $\pmb H_T$ for all $T$ based on Theorem \ref{thrm:main}. We note that the spectral radius of $A_T$ increases as $T$ increases due to the newly added terms. Indeed, if we extend $A_T$ to a $T \times T$ dimensional matrix as 
\[
\tilde A_T = \begin{bmatrix}
    A_T & 0 \\
    0& 0
\end{bmatrix},
\]
where $0$'s are matrices of appropriate dimensions with entries all zero, then it is easy to see that the eigenvalues of $\tilde A_T$ and $A_T$ are the same. Furthermore, for any nonnegative vector $v \in \mathbb R^{\mathrm T}$, we have $\| \tilde A^n_{T}v\|_1 \leq \|A^n_{T+1}v\|_1$ for all $n$. Hence, if we take $v$ as $A_T$'s positive right eigenvector that corresponds to $\rho(A_T)$ and set $\tilde v = \begin{bmatrix}
    v \\
    0
\end{bmatrix}$, then we obtain that $\rho(A_T)^n \| \tilde v\|_1 \leq \| A^n_{T+1} \|_{1,\mathrm {op}} \|\tilde v\|_1$ and hence, $\rho(A_T) \leq \rho(A_{T+1})$ as a consequence of Gelfand's formula. Thus, if one takes $T \to \infty$ in the inequality presented in Theorem \ref{thrm:main}, one obtains an upper bound for $\rho(A_T)$ that is applicable for all $T$. {This observation, combined with \eqref{eq:ref3} and Corollary \ref{cor:ask}, yields the following result:

\begin{corollary}\label{cor:22}
Suppose Assumption \ref{assump:1} holds. If
\[
\sqrt\beta\left( \sqrt{\hat K}+\sqrt{\hat K-\bar K}\right)<1,
\]
then MFG$_{\mathrm T}$ is contractive for all \(T<\infty\).
\end{corollary}

\begin{proof}
This follows from the discussion above.
\end{proof}
}

\begin{remark}\label{rem:1}
    By Theorem \ref{thrm:main} and Corollary \ref{cor:22}, we have $\rho(A_T)<1$ for all $T \in \mathbb N$ if, and only if, it holds that
    \[
    \sqrt{\beta}\left(\sqrt{\hat K}+\sqrt{\hat K-\bar K} \right)<1.
    \]
    Let $\beta = 0.9833,$ and $\bar K = 0.6, \frac{K_1\bar L }{\rho}=0.01$. Then $\sqrt{\hat K \beta } + \sqrt{(\hat K-\bar K)\beta} \sim 0.873$, and thus MFG$_{\mathrm T}$ is contractive for all $T \in \mathbb N$. However, the contraction condition for MFG$_{\mathrm s}$ found in \cite{anahtarci2023q}, which is $\bar K + \frac{\bar LK-1}{\rho(1-\beta)}<1,$ does not hold in this case. Indeed, we have
     $\bar K + \frac{K_1}{\rho}\frac{\bar L}{1-\beta} \sim 1.199$ in this configuration. 
    This shows that $\rho(A_T)<1$ is possible for all $T$ while the stationary MFG fails to be contractive. {In Subsection \ref{sect:3.2}, we will also show that when the spectral analysis done in this section is extended to the $T=\infty$ case, one obtains a contraction condition equivalent to $\bar K + \frac{K_1}{\rho}\frac{\bar L}{1-\beta} <1$ for MFG$_{\mathrm{ns}}$ over the space of bounded sequences, $\ell_{\infty}$.} 
\end{remark}

We now proceed to the proof of Theorem \ref{thrm:main}. To understand the asymptotic behavior of the eigenvalue $\rho(A_T)$, as mentioned, we will study the eigenvalues of a perturbed tridiagonal matrix $T(\rho(A_T))$ that has $\rho(A_T)$ as an eigenvalue. This will allow us to characterize a polynomial that has $\rho(A_T)$ as a root of a polynomial that is constructed by using the recursive relation that arises from $T(\rho(A_T))$. The following lemma  gives us a construction for such $T(\rho(A_T))$.

\begin{lemma}\label{lem:g1}
    Suppose $\beta <1.$ If $A_Th=\rho(A_T)h$, then we also have
    \begin{align*}
     T(\rho(A_T))h :=   \begin{bmatrix}
    -\bar K \beta & \rho(A_T)\beta & 0 & \cdots & 0 & 0\\
    \hat K  & - \bar K \beta & \rho(A_T)\beta & \cdots & 0 & 0\\
    0& \hat K & -\bar K \beta & \cdots &0 & 0\\
    \vdots & \vdots & \vdots & \ddots & \vdots & \vdots \\
    0 & 0 & 0 & \cdots &-\bar K \beta & \rho(A_T)\beta\\
    0 & 0 & 0 & \cdots & \hat K  & - \bar K \beta + r
\end{bmatrix} h = \rho(A_T)h,
    \end{align*}
    where $r := \left( \bar K \beta + \frac{K_1}{\rho} \bar L \beta\right)=\hat K \beta$. Furthermore, eigenvectors of $T(\rho(A_T))$ that correspond to $\rho(A_T)$ are the Perron eigenvectors of $A_T$.
\end{lemma}
\begin{proof}
    See Appendix \ref{sect:lem:g1}.
\end{proof}

Let $(e_i)_i$ be the canonical basis of $\mathbb R^{T-1}$. Let $e^{\top}_i$ denote the transpose of $e_i$. Then note that $e_{T-1}e^{\top}_{T-1}$ is a $(T-1)\times (T-1)$ dimensional matrix of all $0$'s except the $(T-1,T-1)$ entry. Heuristically, we expect $\rho(A_T)$ to be the largest positive eigenvalue of the tridiagonal matrix $T(\rho(A_T)) - re_{T-1}e_{T-1}^{\top}$ as $T \to \infty$, where $(e_i)_i$ is the canonical basis of $\mathbb R^{T-1}$ due to the increasing number of terms. In what follows, this will be proved rigorously. Furthermore, even if we treat $\rho(A_T)$ as the largest positive eigenvalue of $T(\rho(A_T)) - re_{T-1}e_{T-1}^{\top},$ there will be two possible choices of $\rho(A_T)$ asymptotically. Our analysis will show that only one of these choices must hold asymptotically.

\begin{remark}
The heuristic mentioned above is false in general; see \cite{yueh2005eigenvalues}. As shown in \cite{yueh2005eigenvalues}, establishing this heuristic typically requires $r$ to be sufficiently small. In our setting, the analysis is more delicate because the perturbed tridiagonal matrix depends on one of its own eigenvalues. Nevertheless, we will show that a similar procedure can still be applied in our case.
\end{remark}

We prove the lower and upper bounds for $\rho(A_T)$ presented in Theorem \ref{thrm:main} separately. To establish the lower bound, we use a perturbation argument on $T(\rho(A_T))$ by minorizing it with the tridiagonal matrix $T(\rho(A_T)) - r e_{T-1} e_{T-1}^{\top}$. To obtain the desired lower bound in Theorem \ref{thrm:main}, we show that $\rho(A_T)$ is the largest positive eigenvalue of $T(\rho(A_T))$, and then prove that the largest positive eigenvalue of $T(\rho(A_T)) - r e_{T-1} e_{T-1}^{\top}$ is smaller than that of $T(\rho(A_T))$ and conclude by using the explicit formula for the eigenvalues of $T(\rho(A_T)) - r e_{T-1} e_{T-1}^{\top}$.

To do this comparison, first, we show that $\rho(A_T)$ is the largest positive eigenvalue of $T(\rho(A_T))$. Note that perturbations of the form $T(\rho(A_T)) + a I$ shift the eigenvalues of $T(\rho(A_T))$ by $a$. Furthermore, for any $\varepsilon>0$, the matrix $T(\rho(A_T)) + (\bar K \beta  + \varepsilon)I$ is a nonnegative irreducible matrix, and hence satisfies the conditions of the Perron-Frobenius theorem. Since the space of right Perron eigenvectors is a one-dimensional vector space by the Perron- Frobenius theorem, we obtain that $\rho(A_T)$ is the largest positive eigenvalue of $T(\rho(A_T))$:

\begin{lemma}\label{lem:g2}
    $\rho(A_T)$ is the largest positive eigenvalue of $T(\rho(A_T))$.
\end{lemma}
\begin{proof}
    See Appendix \ref{sect:lem:g2}.
\end{proof}

Using Lemma \ref{lem:g2}, in the next proposition, we will provide bounds for $\sqrt{\rho(A_T)}$ by using the minorization of $T(\rho(A_T)) - r e_{T-1} e_{T-1}^{\top}$ on $T(\rho(A_T))$, which is a (true) tridiagonal matrix and for which a closed-form expression is available for the corresponding eigenvalues \cite[Theorem 2.4]{bottcher2005spectral}. These bounds later will be used to establish tight asymptotics for $(\sqrt{\rho(A_T)})_{T \ge 0}$.

\begin{proposition}\label{prop:g1}
    Suppose $\beta <1$. For any natural number $T$ such that $\hat K \cos^2\left(\frac{\pi}{ T+1}\right)-\bar K >0,$ we have either
    \begin{equation}\label{eq:g23}
        \sqrt{\hat K \beta}\cos\left(\frac{\pi}{T+1}\right)+\sqrt{ \left(\hat K \cos^2\left(\frac{\pi}{T+1}\right) - \bar K \right)\beta }  \leq \sqrt{\rho(A_T)}
    \end{equation}
    or 
    \begin{equation}\label{eq:g24}
\sqrt{\rho(A_T)} \leq \sqrt{\hat K \beta}\cos\left(\frac{\pi}{T+1}\right)-\sqrt{ \left(\hat K \cos^2\left(\frac{\pi}{T+1}\right) - \bar K \right)\beta }.
\end{equation}
\end{proposition}
\begin{proof}
    Let $(e_i)_i$ be the canonical basis of $\mathbb R^{T-1}$. For any $\varepsilon>0$, we have the componentwise inequality 
\[
0\leq T(\rho(A_T))+(\bar K\beta +\varepsilon)I - re_{T-1}e_{T-1}^{\top} \leq T(\rho(A_T))+(\bar K\beta +\varepsilon)I
\]
and thus, by \cite[Theorem 2.2]{eriksson2023perron}, we obtain that
\[
\rho\left(  T(\rho(A_T))+(\bar K\beta +\varepsilon)I - re_{T-1}e_{T-1}^{\top}  \right) \leq \rho \left( T(\rho(A_T))+(\bar K\beta +\varepsilon)I  \right) .
\]
By Lemma \ref{lem:g2}, we obtain that
\[
\rho \left( T(\rho(A_T))+(\bar K\beta +\varepsilon)I  \right) = \rho(A_T) + \bar K \beta + \varepsilon
\]
and \cite[Theorem 2.4]{bottcher2005spectral} implies
\begin{align*}
    \rho\left(  T(\rho(A_T))+(\bar K\beta +\varepsilon)I - re_{T-1}e_{T-1}^{\top}  \right) = \varepsilon +2 \sqrt{\hat K \beta \rho(A_T)}\cos\left(\frac{\pi}{T+1}\right).
\end{align*}
It then follows that we have
\[
-\bar K \beta +2 \sqrt{\hat K \beta \rho(A_T)}\cos\left(\frac{\pi}{T+1}\right) \leq \rho(A_T).
\]
Treating $\sqrt{\rho(A_T)}$ as a variable, using the quadratic formula, we obtain either of the following bounds for all sufficiently large $T$:
\[
\sqrt{\hat K \beta}\cos\left(\frac{\pi}{T+1}\right)+\sqrt{ \left(\hat K \cos^2\left(\frac{\pi}{T+1}\right) - \bar K \right)\beta }  \leq \sqrt{\rho(A_T)}
\]
or
\[
\sqrt{\rho(A_T)} \leq \sqrt{\hat K \beta}\cos\left(\frac{\pi}{T+1}\right)-\sqrt{ \left(\hat K \cos^2\left(\frac{\pi}{T+1}\right) - \bar K \right)\beta },
\]
since we have that $\sqrt{\hat K \beta}-\sqrt{ \left(\hat K  - \bar K \right)\beta }>0$.
\end{proof}

Our aim now will be to prove that, for any sufficiently large $T$, \eqref{eq:g24} does not hold for $\rho(A_T)$. This will be done while finding an asymptotically sharp upper bound for \eqref{eq:g23}, which will allow us to obtain Theorem \ref{thrm:main}.

To obtain a sharp upper bound for $\rho(A_T)$, first, we would like to characterize the eigenvectors of the matrix $T(\rho(A_T))$, which is inspired by \cite{yueh2005eigenvalues}. In \cite{yueh2005eigenvalues}, a symbolic calculus over rings is used to analyze the eigenvalues of specific rank-two perturbations of tridiagonal matrices. In our case, our tridiagonal matrices also depend on one of their eigenvalues, so instead, we will assume a slightly different representation for the eigenvectors of $T(\rho(A_T))$ that will lead to a different analysis and will be more useful in our setting.
\begin{lemma}\label{lem:222}
    Suppose $\beta <1$. Any eigenvalue $\lambda$ of $T(\rho(A_T))$ can be written as
    \[
        \lambda = -\bar K \beta + \sqrt{\rho(A_T)\beta \hat K}\left(z + z^{-1}\right),
    \]
    where $z$ is a (potentially complex) root of the polynomial
    \[
    p(z) := z^{2T} - \frac{ r }{\sqrt{ \rho(A_T) \beta \hat K}} z^{2T-1} + \frac{ r }{\sqrt{ \rho(A_T) \beta \hat K}} z -1.
    \]
    Furthermore, if we write $\rho(A_T) = -\bar K\beta + \sqrt{\rho(A_T)\beta\hat K}\left(z + z^{-1}\right)$,
    then right Perron eigenvectors of $A_T$ are of the form $v=(v_i)_{i=1}^T$ (up to a constant multiplicity in $\mathbb C$), where
    \[
    v_j = \left( \sqrt\frac{\hat K}{\rho(A_T)\beta}\right)^j\left( z^j - z^{-j}\right).
    \]
\end{lemma}
\begin{proof}
    See Appendix \ref{sect:lem:222}.
\end{proof}

The degree {of the polynomial} $p$ we defined in Lemma \ref{lem:222} depends on $T.$ In what follows, when we refer to the index $T$, we will also refer to the degree of $p$ although we do not explicitly have $T$ in the notation. In the next lemma, we will provide a sufficient condition that holds under \eqref{eq:g23} so that the roots of the polynomial $p$ all lie on the unit circle $\{z \in \mathbb C :|z|=1\}$, which will be crucial for establishing a tight upper bound for \eqref{eq:g23}.
\begin{lemma}\label{lem:h24}
    If $\frac{r}{\sqrt{ \rho(A_T)\beta \hat K}} \leq 1,$ then the roots of the polynomial $p$ defined in Lemma \ref{lem:222} are on the complex unit circle $\{z \in \mathbb C: |z|=1\}$.
\end{lemma}
\begin{proof}
    See Appendix \ref{sect:lem:h24}.
\end{proof}

To provide a sharp upper bound for $(\rho(A_T))_T$, in our next result, we show that for any sufficiently large $T$ we must indeed have 
\[
\frac{r}{\sqrt{\rho(A_T)\beta \hat K}} \leq 1.
\]
Then, using Lemma \ref{lem:222} and Lemma \ref{lem:h24}, we provide a sharp upper bound for $\rho(A_T)$ and show that \eqref{eq:g23} is a lower bound for $\rho(A_T)$ for sufficiently large $T$. This will amount to a large fraction of the proof of Theorem \ref{thrm:main}. The main caveat of the proof will be that we have to argue by contradiction to exhaust all the possibilities of the asymptotes of $T \mapsto \rho(A_T)$, which results in tight asymptotes for $\rho(A_T)$ only for sufficiently large $T$. The proof only provides us the existence of such a threshold for $T$.

 \begin{theorem}\label{thrm:a1}
Suppose that Assumption \ref{assump:1} holds. Further suppose $\beta <1$. For any sufficiently large $T$ we must have:
\begin{equation}\label{eq:c}
    \sqrt{ \hat K \beta} \cos \left( \frac{\pi}{T+1} \right) + \sqrt{ \left( \hat K\cos^2\left( \frac{\pi}{T+1} \right)-\bar K\right) \beta} \leq \sqrt{\rho(A_T)} \leq \sqrt{ \hat K \beta} + \sqrt{ \left( \hat K-\bar K\right) \beta}.
\end{equation}
In particular, for all sufficiently large $T$, we must have that $r \leq \sqrt{ \rho(A_T)\hat K \beta}.$
\end{theorem}
\begin{proof}
    Recall that $z + z^{-1}$ is real if and only if $z \in \{ z \in \mathbb C : |z|=1\}$ or $z \in \mathbb R_{\not =0}.$ 
    Since the root of the polynomial $p$ that corresponds to $\rho(A_T)$ must be a real number in this case, we will study the behavior of the polynomial $p$ on the real line. If we can eliminate the possibility that $|z| \not = 1$, the result follows. We note that, if $z$ is a root of $p$, $z^{-1}$ will also be a root of $p$, for our purposes it will be sufficient to investigate the real roots of $p$ that are greater than $1$. In particular, we want to understand where the largest real root $z$ of $p$ may lie on. 
    
    For the sake of contradiction, suppose that \eqref{eq:g24} can hold for arbitrarily large $T>0$. If for some $T>0$ we have $\frac{r}{\sqrt{\rho(A_T)\hat K \beta}}\leq 1$, then {$r=\hat K\beta$ implies that for any sufficiently large $T$} we have $\sqrt{\rho(A_T)} \geq \sqrt{\hat K \beta}$, which is a contradiction to \eqref{eq:g24}. So we must have $\frac{r}{\sqrt{\rho(A_T)\hat K \beta}}>1$ for large $T$ values that satisfy \eqref{eq:g24}. {Since $(\rho(A_T))_T$ is monotone, $\lim_{T \to \infty}\rho(A_T)>0$ exists.} Note that if $\lim_{T \to \infty} \frac{ r }{\sqrt{\rho(A_T)\hat K \beta}} =1$, we must have $\lim_{T \to \infty} \sqrt{\rho(A_T)} \ge \sqrt{\hat K\beta}$, which is a contradiction to \eqref{eq:g24} for all large $T$. Thus, we must have $\frac{r}{\sqrt{\rho(A_T)\hat K \beta}}>1+\delta$ for some $\delta >0$ for all sufficiently large $T$ due to the monotonicity of $\rho(A_T)$. 
    
    In this case, note that for all sufficiently large $T$, we have that $p\left(\frac{r}{\sqrt{\rho(A_T)\hat K \beta}}\right)>0.$ Furthermore, for all sufficiently small $\delta >\varepsilon >0,$
    \[
    p\left(\frac{r}{\sqrt{\rho(A_T)\hat K \beta}}-\varepsilon \right) = \left(\frac{r}{\sqrt{\rho(A_T)\hat K \beta}}-\varepsilon \right)^{2T+1}(-\varepsilon) + \left(\frac{r}{\sqrt{\rho(A_T)\hat K \beta}}\right)^2-\varepsilon\frac{r}{ \sqrt{\rho(A_T)\hat K \beta}}-1<0
    \]
    for sufficiently large $T$ as we must have $\left(\frac{r}{\sqrt{\rho(A_T)\hat K \beta}}-\varepsilon \right)>1+\tilde \varepsilon$ in this case for some $\tilde \varepsilon$.

    For all $k \geq \frac{r}{\sqrt{\rho(A_T)\hat K \beta}}$, 
    \[
    p'(k)= k^{2T}\left( 2Tk - \frac{r}{\sqrt{\rho(A_T)\hat K \beta}}(2T-1) \right)+ \frac{r}{\sqrt{\rho(A_T)\hat K \beta}}>0,
    \]
    and thus it follows that 
    $p$ does not change sign over the real line after $\frac{r}{\sqrt{\rho(A_T)\hat K \beta}}>1.$ Since $z + z^{-1}$ is a real number if and only if $|z|=1$ or $z \in \mathbb R_{\not = 0}$, it follows that for all sufficiently large $T$, there is a root $z'$ of $p$ between $\left(\frac{r}{\sqrt{\rho(A_T)\hat K \beta}}-\varepsilon \right)$ and $\frac{r}{\sqrt{\rho(A_T)\hat K \beta}}$ that corresponds to $\rho(A_T).$ For all sufficiently small $\varepsilon$, for all sufficiently large $T$, as the function $x \mapsto x+\frac 1x$ is increasing when $x>1$, it must then hold that
    \begin{align}
    &- \bar K \beta + \sqrt{\rho(A_T)\hat K \beta }\left( \frac{r}{\sqrt{\rho(A_T)\hat K \beta}} + \frac{\sqrt{\rho(A_T)\hat K \beta}}{r} -\varepsilon \right) 
    \\&= \frac{K_1\bar L\beta}{\rho} + \rho(A_T) - \varepsilon\sqrt{\rho(A_T)\hat K \beta } \leq \rho(A_T).
    \end{align}
    {Therefore, for all sufficiently large $T$ we must have
    \[
    \frac{K_1\bar L}{\rho}\beta \leq \varepsilon \sqrt{\rho(A_T)\hat K \beta}.
    \]
    This leads to a contradiction to $\sqrt{\rho(A_T)}$ being strictly positive and bounded above (see \eqref{eq:ref2})} since we can take $\varepsilon \to 0$ by choosing larger and larger values of $T$. It then follows that \eqref{eq:g24} cannot hold for all sufficiently large $T$.

    Since \eqref{eq:g23} must hold for all sufficiently large $T$, it follows directly that for any sufficiently large $T$ we must have $\frac{r}{\sqrt{\beta \hat K}} \leq \sqrt{\rho(A_T)}$. Then, by Lemma \ref{lem:h24}, eigenvalues of $T(\rho(A_T))$ are  of the form 
    $\lambda =- \bar K \beta + \sqrt{ \rho(A_T) \beta \hat K} ( z + \bar z)$ for $|z|=1$. Thus, $z=e^{i \theta}$ for some $\theta \in \mathbb R$, and hence $z+ \bar z \in \mathbb R$. Furthermore, $|z+\bar z| \leq 2$. Hence, 
    \[
    \rho(A_T) \leq - \bar K \beta + 2\sqrt{\rho(A_T)\beta \hat K}.
    \]
    We can obtain the upper bound from the inequality above by applying the quadratic formula to the variable $x = \sqrt{\rho(A_T)}$.
\end{proof}

We want to highlight that unlike the bound \eqref{eq:ref2}, the bound in \eqref{eq:c} only considers the discounted cost in the quotient in the expression of $\bar L$. Moreover, the condition $\rho(A_T)<1$ already forces $K_1 \leq \frac 23$, and the contraction rate we have for MFGs with a finite-horizon does not blow up as $\beta \rightarrow 1,$ in contrast to the infinite-horizon scenario as we will show in the next subsection. 

Using Theorem \ref{thrm:a1}, we now proceed to complete the proof of Theorem \ref{thrm:main}. This will amount to a further perturbation of the matrix $T(\rho(A_T))$ for all sufficiently large $T$.

\begin{proof}[Proof of Theorem \ref{thrm:main}]
    We observe that, as $r \leq \sqrt{ \rho(A_T) \hat K \beta}$ must hold for sufficiently large $T$ by Theorem \ref{thrm:a1}, the following majorization for $T(\rho(A_T))$ must hold for all such sufficiently large $T$:
\[
T(\rho(A_T)) \leq 
\begin{bmatrix}
    -\bar K \beta & \rho(A_T)\beta & 0 & \cdots & 0 & 0\\
    \hat K  & - \bar K \beta & \rho(A_T)\beta & \cdots & 0 & 0\\
    0& \hat K & -\bar K \beta & \cdots &0 & 0\\
    \vdots & \vdots & \vdots & \ddots & \vdots & \vdots \\
    0 & 0 & 0 & \cdots &-\bar K \beta & \rho(A_T)\beta\\
    0 & 0 & 0 & \cdots & \hat K  & - \bar K \beta + \sqrt{\rho(A_T)\beta \hat K}
\end{bmatrix} =: \hat T(\rho(A_T)).
\]
Repeating Lemma \ref{lem:222} for $\hat T(\rho(A_T))$ (simply by plugging $\sqrt{\rho(A_T)\beta \hat K}$ in the place of $r$), we see that eigenvalues of $\hat T (\rho(A_T))$ must come from the polynomial 
    \[
    \tilde     p(z) := z^{2T} -  z^{2T-1} +  z -1 = (z^{2T-1}+1)(z-1).
    \]
The roots of $\tilde p$ that contribute to the eigenvalues of $\hat T(\rho(A_T))$ are merely the $(2T-1)^{\mathrm{th}}$ roots of unity. In particular, since the entries of $\hat T(\rho(A_T))$ are greater than those of $T(\rho(A_T))$, we must have $\rho(\hat T(\rho(A_T))) = -\bar K \beta + 2 \sqrt{\rho(A_T) \beta \hat K}\cos( \frac{\pi}{2T-1}) \geq \rho(A_T)$ as a consequence of the spectral radius monotonicity for nonnegative matrices. 

The structure of Perron eigenvectors of $A_T$ follows from Lemma \ref{lem:222}, Lemma \ref{lem:h24}, and Theorem \ref{thrm:a1} directly, as we must have that $|z|=1$ for any sufficiently large $T$. 
\end{proof}

For the purpose of our finite-error bounds, we will also need the left Perron eigenvectors of the matrix $A_T$.
\begin{proposition}\label{prop:a}
    Suppose that Assumption \ref{assump:1} holds. For any sufficiently large $T$, left Perron eigenvectors of $A_T$ are of the form $v=(v_j)_j$, where
    \[
    v_j = \left( \sqrt\frac{\rho(A_T)\beta}{\hat K}\right)^j2i\sin((T-j)\theta_T),
    \]
    where $\theta_T$ satisfies the relation
    \[
    \rho(A_T) = -\bar K \beta + \sqrt{\rho(A_T)\beta \hat K}\cos(\theta_T).
    \]
\end{proposition}
\begin{proof}
    Follows after repeating the proof of Theorem \ref{thrm:main} for the transpose of $A_T$, $A^*_T$.
\end{proof}

\begin{figure}[h]
  \centering
  \begin{minipage}{0.5\textwidth}
    \centering
    \includegraphics[width=\linewidth]{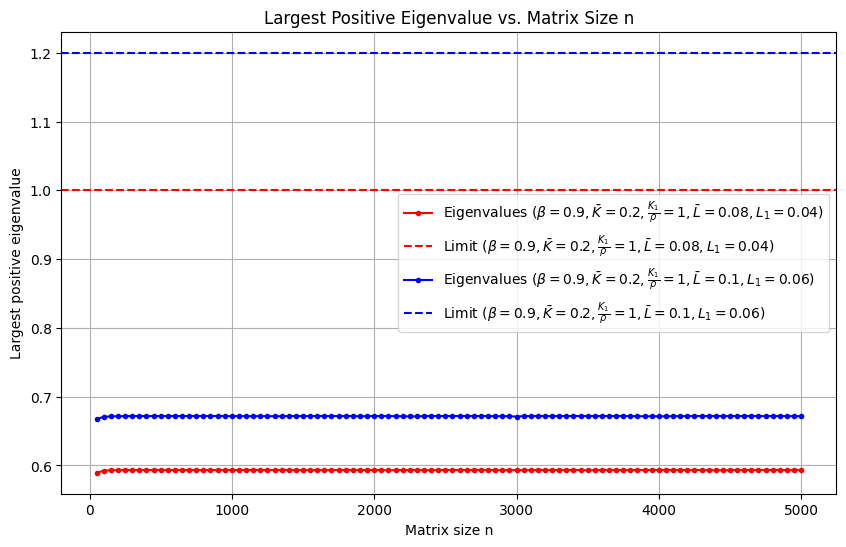}
    \\[2mm]
    \label{fig:hehe1}
  \end{minipage}
  \hfill
  \begin{minipage}{0.5\textwidth}
    \centering
    \includegraphics[width=\linewidth]{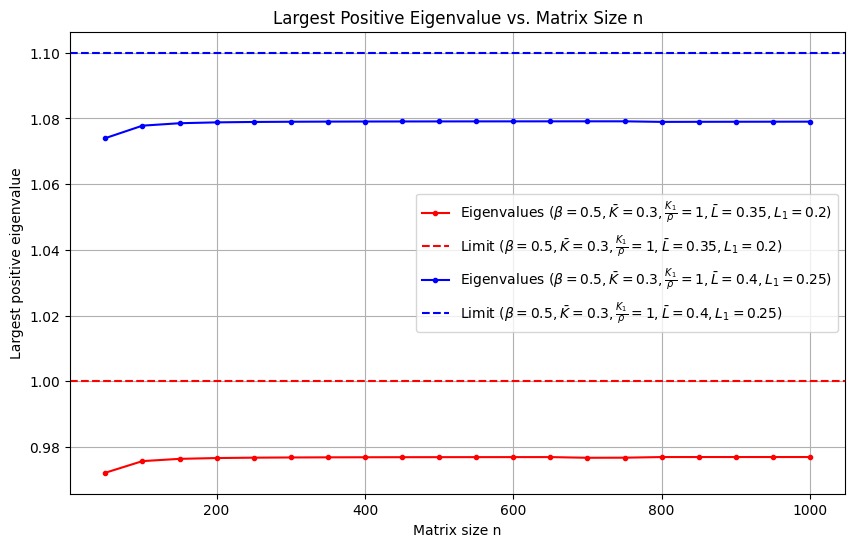}
    \\[2mm]
    \label{fig:hehe3}
  \end{minipage}
  \caption{The limit values represent the upper bound predicted by Gershgorin's Circle Theorem for the matrix $A_T$, while the eigenvalues are estimates of $\rho(A_T)$ obtained by searching $k$ that satisfies the quantity $\mathrm{det}(T(\rho(A_T))-kI)=0$, where $I$ is an identity matrix of appropriate dimension.}
  \label{fig:1}
\end{figure}

\subsection{Fixed-Point Iteration for Infinite-Horizon MFGs}\label{sect:3.2}

Let MFG\(_{\text{ns}}=(X,\mathcal P(A),C+\Omega,P,\mu_0)\) and MFG\(_{\text s} =(X,\mathcal P(A),C+\Omega,P)\). In this subsection, we will study the contraction rates for fixed-point iterations for $\mathrm{MFG}_{\mathrm{ns}}$ and \(\mathrm{MFG}_{\mathrm{s}}\). {In the case of \(\mathrm{MFG}_{\mathrm{s}}\), the contraction condition proved in \cite{anahtarci2023q} reads 
\[
\bar K+\frac{\bar L K_1}{\rho(1-\beta)}<1.
\]
As discussed in Remark \ref{rem:1}, MFG$_{\mathrm T}$ can enjoy a strictly more relaxed contraction condition for all $T<\infty$ than $\bar K+\frac{\bar L K_1}{\rho(1-\beta)}<1$. This raises the question whether using spectral analysis tools can provide us more relaxed contraction conditions for MFG$_{\mathrm{ns}}$ than the contraction condition of MFG$_{\mathrm{s}}$. For this purpose, we will extend the analysis done in the previous subsection to the case $T=\infty$.} Our aim will be to establish a similar {contraction} result for $\mathrm{MFG}_{\mathrm{ns}}$. {In particular,} we will show that the contraction rate one obtains for the fixed-point iteration for $\mathrm{MFG}_{\mathrm{ns}}$ in the space of bounded sequences, \((\ell_{\infty},\|\cdot\|_{\infty}),\) is the same as the one found for $\mathrm{MFG}_{\mathrm s}$ in \cite{anahtarci2023q}.

The main result of this section is the following:
\begin{theorem}\label{prop:2}
    Suppose that Assumption \ref{assump:1} holds. If $\bar K + \frac{\bar LK_1}{\rho(1-\beta)} <1$, then there exist unique MFEs of regularized $\mathrm {MFG}_{\mathrm{ns}}$ and regularized $\mathrm {MFG}_{\mathrm{s}}$ under the regularization parameter $\rho$.
\end{theorem}

The proof of Theorem \ref{prop:2} relies on constructing a contractive mapping over the space of bounded sequences, $(\ell_{\infty},\|\cdot\|_{\infty})$ in the $\mathrm{MFG}_{\mathrm{ns}}$ case. {As done in Lemma \ref{lem:f4}, we will argue by majorization. The dynamics of the majorization in the infinite-horizon non-stationary case will be given by an infinite-dimensional one-sided Toeplitz matrix, $\mathcal A$, over $\ell_{\infty}$.} It will be shown that $\bar K + \frac{\bar LK_1}{\rho(1-\beta)}$ is the {spectral radius of $\mathcal A$ over $\ell_{\infty}$. The remaining $\mathrm{MFG}_{\mathrm s}$ case is covered in \cite{anahtarci2023q} as discussed in Subsection \ref{sect:b}, and hence we will not include the details for the stationary case. 

Before proceeding with the proof of Theorem \ref{prop:2}, a few further remarks are in order concerning the gap between the contraction conditions in the cases \(T<\infty\) and \(T=\infty\). The matrices \(A_T\) used in Subsection \ref{sect:b} may be viewed, for each \(T\in\mathbb N\), as finite-rank operators on \(\ell_{\infty}\). It is well known that the space \((\ell_{\infty},\|\cdot\|_{\infty})\) has the \emph{bounded approximation property}; that is, when \(\ell_{\infty}\) is equipped with the supremum norm, finite-rank operators are dense in the compact operators on \(\ell_{\infty}\) \cite[p.~256]{MR551623}. It is also well known that the only compact Toeplitz operator on \(\ell_{\infty}\) is the zero operator. Therefore, the matrices \((A_T)_T\) cannot approximate its \(T=\infty\) counterpart in the operator norm on \(\ell_{\infty}\). In particular, one should not expect, in general, the spectral radius of \(A_T\) to converge to those of their respective \(T=\infty\) counterparts in $\ell_{\infty}$, which explains the potential gap between the $T<\infty$ and $T=\infty$ cases.

\begin{remark}
    Spectral radius of infinite-dimensional operators are norm and space specific. Thus, the discussion above does not imply that $A_T$ and $B_T$ cannot approximate the spectral radius of their $T=\infty$ counterpart in some other space.
\end{remark}

For the sake of simplicity, we provide the proof of Theorem \ref{prop:2} only for the \(T=\infty\) counterpart of \((A_T)_T\), namely \(\mathcal A\).
To set up the proof of Theorem \ref{prop:2}, we will introduce infinite-horizon non-stationary counterparts of Lemma \ref{lem:f2} and Lemma \ref{lem:f4}, which were fundamental in the finite-horizon case. Recall that Lemma \ref{lem:f4} depends on the Lipschitz properties of the $Q$-functions in the finite-horizon case, i.e., Lemma \ref{lem:f2}. First, we will state an infinite-horizon analogue of Lemma \ref{lem:f2}, where we show that $Q$-functions we obtain in the infinite-horizon case are $\bar L$-Lipschitz over the state space $X$.

\begin{lemma}\label{lem:10}
    Suppose that Assumption \ref{assump:1} holds. Let $\pmb \mu \in \mathcal P(X)^{\infty}$. Then, there exists a unique family of $Q$-functions $(Q_t)_{t \in \mathbb N}$ such that
    \[
    Q_t(x,u) = C(x,u,\mu_t) + \rho\Omega(u)+\beta\sum_{y \in X} \min_{v \in \mathcal P(A)}Q_{t+1}(y,v)P(y|x,u,\mu_t)
    \]
    for all $t \ge0$. In this case, we say that the $Q$-functions $(Q_t)_{t=0-}^{\infty}$ \emph{correspond} to the state-measure flow $\pmb \mu$. Furthermore, it holds that 
    \[
    \sup_{v \in \mathcal P(A)} \sup_{x,x' \in \mathcal P(X)} |Q_t(x,v)-Q_t(\tilde x,v)| \le \bar L
    \]
    and if \(\|C\|_{\infty} + \rho\|\Omega\|_{\infty} \le M\), we have \(\|Q_t\|_{\infty} \le \frac{M}{1-\beta}.\) We will denote $(Q_t)_{t \in \mathbb N}$ as $(Q^{\pmb \mu})_t$ to emphasize the corresponding state-measure flow.
\end{lemma}
\begin{proof}
    See Appendix \ref{sect:lem:f6} for an argument in the nonregularized case. The regularized case follows similarly, and thus we will not include the details.
\end{proof}

Next, we derive an infinite-horizon analogue of Lemma \ref{lem:f4}.

\begin{lemma}\label{lem:11}
    Suppose that $\pmb \mu, \tilde {\pmb \mu} \in \{\mu_0\} \in \mathcal P(X)^{\infty}$. Let $(Q_t)_{t \in \mathbb N}$ and $(\tilde Q_t)_{t \in \mathbb N}$ be the corresponding $Q$-functions. Then, for all $t \ge2$, and $T >t$, we have
    \begin{align}
    \| \mu_t - \tilde \mu_t\|_{\mathrm {TV}} & \leq \left(\bar K +\frac{K_1}{\rho}\bar L\right)\|\mu_{t-1}-\tilde \mu_{t-1}\|_{\mathrm{TV}} + \frac{K_1}{\rho}\bar L\sum_{i=t}^{T-1}\beta^{i-t+1}\|\mu_i -\tilde \mu_i \|_{\mathrm{TV}} 
        \\& \qquad + \frac{K_1}{\rho}\bar L \beta ^{T-t+1} \|\mu_{T} - \tilde \mu_{T}\|_{\mathrm{TV}}+\frac{K_1}{2\rho}\beta^{T+1-t}\|Q_{T+1}-\tilde Q_{T+1}\|_{\infty}\label{eq:yyyy}.
    \end{align}
    In particular, as $T \to \infty$, we have
    \begin{equation}\label{eq:rv100}
    \| \mu_t -\tilde \mu_t\|_{\mathrm{TV}} \le \left(\bar K +\frac{K_1}{\rho}\bar L\right)\|\mu_{t-1} - \tilde \mu_{t-1}\|_{\mathrm{TV}}+ \frac{K_1}{\rho}\bar L\sum_{i=t}^{\infty}\beta^{i-t+1}\|\mu_i -\tilde \mu_i \|_{\mathrm{TV}}.
    \end{equation}
\end{lemma}
\begin{proof}
This result follows from Lemma \ref{lem:f1} by repeatedly applying Lemma \ref{lem:f2} to the \(Q\)-functions of the infinite-horizon MFE. Let \((Q_t)_{t=0}^{\infty}\) and \((\tilde Q_t)_{t=0}^{\infty}\) be the \(Q\)-functions obtained under the state-measure flows \((\mu_t)_{t=0}^{\infty}\) and \((\tilde \mu_t)_{t=0}^{\infty}\), respectively. Lemma \ref{lem:f2} can be adjusted to the infinite-horizon case after applying Lemma \ref{lem:10}, which yields
\begin{equation}\label{eq:rv101}
    \|Q_t-\tilde Q_t\|_{\infty} = \|H_{1,t}(Q_{t+1},\mu_t)-H_{1,t}(\tilde Q_{t+1},\tilde \mu_t)\|_{\infty} \le 2\bar L \|\mu_t-\tilde \mu_t\|_{\mathrm{TV}} + \beta \| Q_{t+1}-\tilde Q_{t+1}\|_{\infty},
\end{equation}
as desired. The bound \eqref{eq:rv100} can be directly obtained from \eqref{eq:rv101} because the $Q$-functions are bounded. In particular,
\(
\beta^T \|Q_T - \tilde Q_T\|_{\infty} = O(\beta^T) \to 0\) as \(T \to \infty,
\)
since the $Q$-functions are bounded above by $\frac{M}{1-\beta}$.
\end{proof}

\begin{proof}[Proof of Theorem \ref{prop:2}]
We will show that the contraction in the $\mathrm{MFG}_{\mathrm{ns}}$ case can be deduced by the spectral properties of a one-sided infinite Toeplitz operator. Generally, the spectrum of finite-dimensional Toeplitz matrices does not describe the spectrum of their infinite-dimensional counterparts \cite[Section 10.3]{bottcher2005spectral}. To analyze what happens in the infinite-horizon case, first observe that $A_T$ is a Toeplitz-like matrix as each of its diagonal strips has the same constant value except for its last column. Thus, taking the time horizon $T \to \infty$ in the inequalities presented in Lemma \ref{lem:f4} by using Lemma \ref{lem:10} and Lemma \ref{lem:11}, we get the following formal expression that bounds the variation of state-measure flows in the infinite-horizon setting:
\begin{equation}\label{eq:f27}
    v_1:=\begin{bmatrix}
         \| H_{2,1}(Q^{\pmb {\mu}}_0,\mu_0) - H_{2,1}(Q^{\pmb {\tilde \mu}}_0,\mu_0) \|_{\mathrm{TV}} \\
        \| H_{2,2}(Q^{\pmb {\mu}}_1,\mu_1) - H_{2,2}(Q^{\pmb {\tilde \mu}}_1,\tilde \mu_1) \|_{\mathrm{TV}} \\
        \vdots 
    \end{bmatrix} \leq \mathcal A \begin{bmatrix}
     \| \mu_1 - \tilde \mu_1 \|_{\mathrm{TV}} \\
     \| \mu_2 - \tilde \mu_2 \|_{\mathrm{TV}} \\
     \vdots \\
    \end{bmatrix}=: \mathcal Av_2,
    \end{equation}
where $\mathcal A$ is the following \emph{bona-fide one-sided infinite Toeplitz matrix} \cite[Eq. 3.15.62]{barry2015operator}
    \[
    \A = \begin{bmatrix}
        \frac{\bar LK_1}{\rho}\beta  & \frac{\bar LK_1}{\rho}\beta^2 & \frac{\bar LK_1}{\rho} \beta^3 & \cdots\\
        \bar K+\frac{\bar LK_1}{\rho} & \frac{\bar LK_1}{\rho}\beta & \frac{\bar LK_1}{\rho} \beta^2 & \cdots \\
        0 & \bar K +\frac{\bar LK_1}{\rho} & \frac{\bar LK_1}{\rho}\beta & \cdots\\
        \vdots & \vdots & \vdots & \ddots
    \end{bmatrix}.
    \]
The expression \eqref{eq:f27} can be made rigorous as follows. Each component of $v_1$ and $v_2$ is bounded by $4,$ since the total variation norm is bounded by $2$, so $v_1,v_2 \in \ell_{\infty}$, where $\ell_{\infty}$ is the space of bounded sequences. Furthermore, each row of $\mathcal A$ is bounded because $\beta<1$; therefore, $\mathcal A$ is indeed an infinite Toeplitz matrix over $\ell_{\infty}$ \cite[Proposition 1.1]{bottcher2005spectral}. Thus, \eqref{eq:f27} can be interpreted as a relation over the vector space $\ell_{\infty}$.

Let $\bar a = (a_1, a_2,\cdots) \in \ell_{\infty}$ be a bounded sequence. Denote by $\| \mathcal A \|_{\ell_{\infty},\textrm{op}}$ the operator norm of $\mathcal A$ obtained under the supremum norm. Then, it is easy to check the following inequality: 
\[
\| \mathcal A \bar a \|_{\ell_{\infty}} \leq \left(\bar K+\frac{\bar LK_1}{\rho(1- \beta)}\right)\| \bar a \|_{\ell_{\infty}}.
\]
Thus, the spectral radius of $\mathcal A$ under the norm $\| \cdot \|_{\ell_{\infty}}$, $\rho(\mathcal A)$, satisfies $\rho(\mathcal A)\leq \| \mathcal A \|_{\ell_{\infty},\textrm{op}} \leq \bar K+\frac{\bar LK_1}{\rho(1- \beta)}$. Furthermore, the \emph{essential spectrum} of $\mathcal A$ \cite[p. 193]{barry2015operator} is the same as the range of the \emph{symbol} of $\mathcal A$ \cite[p. 213]{barry2015operator} over the unit circle in $\mathbb C$ provided that the symbol of $\mathcal A$ is continuous over the unit circle. The symbol that corresponds to the operator $\mathcal A$ is the function $\phi : \mathbb C \to \mathbb C$ defined by the Laurent polynomial
\[
\phi(z) := \frac{\bar LK_1}{\rho}\beta  + \sum_{i=1}^{\infty} \frac{\bar LK_1}{\rho} \beta^{i+1} z^i + \left( \bar K + \frac{\bar LK_1}{\rho} \right)z^{-1}.
\]
Since $\phi$ is continuous on the unit circle over $\mathbb C$, we can use the aforementioned result to calculate the maximum over the essential spectrum of $\mathcal A$, which provides a lower bound for its spectral radius \cite[Theorem 3.15.22]{barry2015operator}. 

A direct inspection shows that the maximum of $\phi$ over the unit circle is attained at $z=1$. Therefore, $\phi(1) = \bar K+\frac{\bar LK_1}{\rho(1- \beta)} \leq \rho(\mathcal A) \leq \bar K+\frac{\bar L K_1}{\rho(1- \beta)}$ since the essential spectrum is contained within the spectrum and the operator norm is an upper bound for the spectral radius. Hence, $\rho(\mathcal A) = \bar K+\frac{\bar L K_1}{\rho(1- \beta)}$, and thus we do not obtain any improvement in the infinite-horizon setting.

   The uniqueness of the MFE of MFG\(_{\mathrm{ns}}\) follows from the discussion above {and the Banach fixed-point theorem}. The stationary case directly follows from \cite{anahtarci2023q}, and thus we will not include the details.
\end{proof}

\begin{remark}
    We note that the operator $\mathcal A$ (or its adjoint) might not have any eigenvalue in $(\ell_{\infty},\|\cdot\|_{\infty})$, and therefore the trick we did in the proof of Theorem \ref{thrm:f2} is not possible to construct a norm that yields convergence of the iterations of \(\alpha \pmb I + (1-\alpha)\mathcal A\) when \(\bar K + \frac{K_1}{\rho(1-\beta)}\bar L=1\) for MFG\(_{\mathrm{ns}}\). We also note that since our proof relies on a majorization argument to preserve the order of the vector inequalities one needs an even stronger property than having an eigenvalue, namely, existence of a positive eigenvector.
\end{remark}

In the next subsection, we show that a unique non-stationary MFG exists even when the condition
\[
\bar K + \frac{\bar L K_1}{\rho(1-\beta)} < 1
\]
fails. Moreover, we show that this unique infinite-horizon non-stationary MFE can be approximated by finite-horizon MFEs.

\subsection{Finite-time Error Bounds Between Finite-Horizon Equilibria and Infinite-Horizon Equilibria}\label{sect:3.5}

In the previous subsection, we showed that the contraction condition for infinite-horizon non-stationary MFGs over $(\ell_{\infty}, \|\cdot\|_{\infty})$ coincides with the contraction condition for stationary MFGs established in \cite{anahtarci2023q}, {which is a more restrictive condition than what we have for finite-horizon MFGs as discussed in Remark \ref{rem:1}.} Furthermore, the corresponding iteration procedure established for infinite-horizon non-stationary MFG is intractable as we have to iterate infinitely many vectors at once. To obtain a tractable iterative scheme for approximating infinite-horizon non-stationary MFE, we show in this subsection that, if the map $\pmb H_T$ is contractive and an additional condition on the Lipschitz coefficients of $(C,P)$ is satisfied, then there exists a unique infinite-horizon non-stationary MFE, which can be approximated by finite-horizon MFEs as $T \to \infty$. We also provide explicit convergence rates. Finally, using the error bounds between the infinite-horizon non-stationary MFE and the finite-horizon MFEs as an intermediate step, we establish finite-time error bounds between the finite-horizon MFEs and the stationary MFE under the conditions of Theorem \ref{prop:2}.

For this purpose, we work within the setting of the previous subsection, i.e., our MFGs are regularized. The following additional assumption on the Lipschitz coefficients of the system components will be necessary for our purposes:
\begin{assumption}\label{assump:4}
There exists $1>\varepsilon>0$ such that we have
\[
\frac{ \sqrt {\hat K} }{ \sqrt{\hat K} + \sqrt{\hat K - \bar K} } > \beta ^{1-\varepsilon}.
\]
\end{assumption}

\begin{remark}\label{rem:ssss}
 It is easy to find examples where Assumption \ref{assump:4} is satisfied but $\bar K + \frac{\bar LK_1}{\rho(1-\beta)} >1,$ or where $\bar K + \frac{\bar LK_1}{\rho(1-\beta)} <1$ but Assumption \ref{assump:4} is not satisfied. This remark shows that MFG$_{\mathrm T}$ can be contractive {for all $T \in \mathbb N$} and Assumption \ref{assump:4} can be satisfied while $\bar K + \frac{\bar LK_1}{\rho(1-\beta)} >1$, i.e., contraction condition fails in the infinite-horizon setting (Theorem \ref{prop:2}). This observation is important, as {we will demonstrate} that when MFG$_{\mathrm T}$ are contractive {for all $T \in \mathbb N$} and Assumption \ref{assump:4} is satisfied, finite horizon MFE necessarily converge to an infinite-horizon non-stationary MFE as $T \to \infty$, which in turn will imply the uniqueness of infinite-horizon non-stationary MFE. Under our assumptions, the existence of an infinite-horizon non-stationary MFE follows from \cite[Theorem 3.3]{SaBaRaSIAM}.

     Let $\bar K= 1,$ $\hat K - \bar K = 0.04$, and $\beta =0.1$. Then it holds that $\sqrt{\hat K} \sim 1.012,$ $\sqrt{\hat K - \bar K} = 0.2,$ and thus we have $\sqrt{\hat K} \sim 1 > \beta ( \sqrt{\hat K} + \sqrt{\hat K - \bar K} ) \sim 0.1012,$ and Assumption \ref{assump:4} is satisfied. Furthermore,
    \[
    \sqrt{\beta} \left( \sqrt{\hat K - \bar K} + \sqrt{ \hat K}\right) \sim 0.322 <1,
    \]
    which implies that MFG$_{\mathrm T}$ is contractive for all $T$ by Theorem \ref{thrm:main}. However, the contractivity condition presented in Theorem \ref{prop:2} is not satisfied, since it reads as $\bar K + (\hat K - \bar K)/(1-\beta) > 1$ in this configuration; thus, corresponding (non-stationary and stationary) infinite-horizon MFG might not be contractive when we merely have that finite-horizon MFGs are contractive and Assumption \ref{assump:4} holds.
 \end{remark}

Assumption \ref{assump:4} essentially amounts to restricting the left Perron eigenvectors of the matrix $A_T$ to the case where each component is (eventually) exponentially decreasing. This will allow us to control the gap between infinite-horizon and finite-horizon MFE as $T \to \infty$.

Before proceeding to the statement of the main result of this subsection, we recall that $f(x) = O(g(x))$ if $|f(x)| \leq M |g(x)|$ for all $x \geq x_0$ for some $x_0$. We also want to point out that as a consequence of Assumption \ref{assump:1}, the cost function $c$ is bounded, say by $M$. It then follows that the cost function $C$ is also bounded by $M$. Accounting the perturbation caused by the regularizer $\Omega$, and abusing the notation slightly, we will assume that all the $Q$-functions obtained under a MFE, both in finite-horizon and infinite-horizon cases are bounded by the constant $\frac{M}{1-\beta}$.

\begin{theorem}\label{thrm:7}
    Suppose that Assumptions \ref{assump:1} and \ref{assump:4} hold. Further assume that $\sqrt{ \hat K \beta } + \sqrt{ ( \hat K - \bar K)\beta}<1$. Let $(\pmb \pi , \pmb \mu) = ( (\pi_t)_{t \in \mathbb N},(\mu_t)_{t\in \mathbb N}) \in \mathrm {MFE}_{\mathrm {ns}}$ and $(\pmb \pi^{\pmb T}, \pmb \mu^{\pmb T}) = ( (\pi^T_t)_{t=0}^{T-1}, (\mu^T_t)_{t=0}^{T-1}) \in \mathrm{MFE}_{\mathrm T}$ such that $\mu_0 = \mu^T_0$ for all $T \in \mathbb N$. Then, for any sufficiently large $T$,
     it holds that
    \begin{equation}\label{eq:ssss34444}
    \| \mu_t - \mu^T_t \|_{\mathrm {TV}} \leq O \left(  \frac {2T-1}{T-t}\left(\sqrt{ \frac{\hat K}{ \rho(A_T) \beta}}\right)^t \beta^{T\varepsilon}\right).
    \end{equation}
    In particular,  MFG$_{\mathrm{ns}}$ has a unique MFE under the assumptions above.
\end{theorem}

\begin{proof}[Proof of Theorem \ref{thrm:7}]
By $\pmb \mu|_T$, we denote the restriction of the flow $\pmb \mu$ to its first $T-1$ components. For all $t \in \mathbb N$, let $(Q_t)_{t \in \mathbb N}$ be a family of $Q$-functions such that $Q_t = H_{1,t}(Q_{t+1},\mu_t)$.
Using the finite-horizon approximation scheme for infinite-horizon dynamic programming for infinite-horizon MFGs \cite[Lemma A.1]{SaBaRaSIAM}, it can be shown that for all $t \in \mathbb N$, we have
\[
\sup_{u \in \mathcal P(A)}\sup_{x,x'\in X} | Q_t(x,u) - Q_t(\tilde x,u) | \le \bar L,
\]
i.e., for all $t \in \mathbb N$ and $u \in \mathcal P(A)$, $Q_t(\cdot,u)$ is $\bar L$-Lipschitz over $X$.
The proof of this observation follows closely that of Lemma \ref{lem:f0} because the finite-horizon approximations of the functions $(Q_t)_t$ will satisfy the Lipschitz conditions presented in Lemma \ref{lem:f0} by construction; see Appendix \ref{sect:lem:f6} for details of this argument.

With the observation above, arguing as done in the proof of Lemma \ref{lem:f4}, for any $T \in \mathbb N$ and $1\le t <T$, we obtain the following inequality:
\begin{align*}
        &\|\mu_t - \mu^T_t\|_{\mathrm{TV}}=\| H_{2,t}(Q_{t-1},\mu_{t-1}) - H_{2,t}(Q^{\pmb \mu^T}_{t-1},\mu_{t-1}^T) \|_{\mathrm{TV}} 
        \\& \leq \left(\bar K +\frac{K_1}{\rho}\bar L\right)\|\mu_{t-1} -\mu_{t-1}^T\|_{\mathrm{TV}} + \frac{K_1}{2\rho}\beta \| Q^{\pmb {\mu}^T}_{t-1} - Q_{t-1} \|_{\infty}
        \\& \vdots
        \\& \leq \left(\bar K +\frac{K_1}{\rho}\bar L\right)\|\mu_{t-1}^T -\mu_{t-1}\|_{\mathrm{TV}} + \frac{K_1}{\rho}\bar L\sum_{i=t}^{T-1}\beta^{i-t+1}\|\mu_i^T -\mu_i \|_{\mathrm{TV}} 
        \\& \qquad + \frac{K_1}{\rho}\bar L \beta ^{T-t} \|\mu_{T}^T - \mu_{T}\|_{\mathrm{TV}}+\frac{K_1}{2\rho}\beta^{T+1-t}\|Q_{T+1}\|_{\infty},
    \end{align*}
where the extra remainder terms with $\|Q_{T+1}\|_{\infty}$ are due to the fact that the $Q$-functions of infinite-horizon non-stationary MFE does not terminate at $t=T$, while the finite-horizon ones do.
In particular, for all $T\in \mathbb N$, arguing as in Lemma \ref{lem:f4}, we must have
\begin{align*}
    V(\pmb \mu|_T,\pmb \mu^T) \le A_TV(\pmb \mu|_T,\pmb \mu^T) + \frac{K_1}{2\rho}\|Q_{T+1}\|_{\infty}\begin{bmatrix}
        \beta^{T} \\
        \beta^{T-1}\\
        \vdots \\
        \beta
    \end{bmatrix}.
\end{align*}
 Let $v^T$ be a positive left Perron eigenvector of $A_T$. Then, we obtain the averaged norms relation
\begin{equation}\label{eq:bb}
\langle v^T,V(\pmb \mu|_T,\pmb \mu^T) \rangle \le \langle v^T, A_T V(\pmb \mu|_T,\pmb \mu^T)\rangle + \frac{K_1}{2\rho}\|Q_{T+1}\|_{\infty}\langle v^T , \begin{bmatrix}
        \beta^{T} \\
        \beta^{T-1}\\
        \vdots \\
        \beta
    \end{bmatrix}\rangle.
\end{equation}
For all $T \in \mathbb N$, the function $\langle v^T, V(\cdot,\cdot)\rangle$ will serve as a Lyapunov function for us. Our aim is to write \eqref{eq:bb} as explicitly as possible.

Since the left eigenvectors of $\rho(A_T)$ are right eigenvectors of $A^*_T$, the transpose of $A_T$, note that the relationship $A^*_T v = \rho(A_T)v$ implies the following recursion relation:
    \[
    \begin{bmatrix}
        -\bar K \beta +r & \hat K & 0 & \cdots & 0 & 0 \\
        \beta\rho(A_T) & -\bar K \beta & \hat K & \cdots & 0 & 0 \\
        0 & \beta \rho(A_T) & -\bar K \beta & \cdots & 0 & 0\\
        \vdots & \vdots & \vdots & \ddots & \vdots & \vdots\\
        0 & 0 & 0 & \cdots & - \bar K \beta & \hat K \\
        0 & 0 & 0& \cdots & \beta \rho(A_T) & - \bar K \beta
    \end{bmatrix} v = \rho(A_T) v.
    \]
    Then, arguing as in the proof of Theorem \ref{thrm:main} (see Proposition \ref{prop:a}), for all sufficiently large $T$, we obtain that eigenvectors $r^T$ that correspond to $\rho(A_T)$ for the matrix above are of the form $r^T=(r^T_1,r^T_2,\cdots,r^T_{T-1})$, where
    \[
    r^T_j = \left( \sqrt\frac{\rho(A_T)\beta}{\hat K}\right)^j2i\sin((T-j)\theta_T).
    \]
    Then, to obtain a positive eigenvector $r^T$, we have to multiply it with a constant $C \in \mathbb C$ such that $|C|=1$ that will give us $Cr^T_i>0$ for all $i$, {whose existence is guaranteed by the Perron-Frobenius theorem.}

 With this setting, for any sufficiently large $T,$ \eqref{eq:bb} can be written as
\begin{align*}
    \langle Cr^T , V(\pmb \mu|_T , \pmb \mu^{\pmb T} )\rangle \leq \rho(A_T) \langle Cr^T, V(\pmb \mu|_T , \pmb \mu^{\pmb T} )\rangle + \frac{K_1}{2\rho}\beta^{T}\sum_{j=1}^{T} \beta^{-j}Cr^T_{j}\| Q_{T+1}\|_{\infty},
\end{align*}
which gives us
\[
\langle Cr^T , V(\pmb \mu|_T , \pmb \mu^{\pmb T} )\rangle \le \frac{K_1}{2\rho(1-\rho(A_T))}\beta^{T}\sum_{j=1}^{T} \beta^{-j}Cr^T_{j}\| Q_{T+1}\|_{\infty}.
\]
For sufficiently large $T$, by Proposition \ref{prop:a}, $\theta_T$ used in $r^T$ should be such that $\theta_T \sim 0$; see the bounds for $\theta_T$ in Theorem \ref{thrm:main}. Thus, as $\|Q_T\| \leq \frac{M}{1-\beta}$ for all $T\in \mathbb N$, we obtain that
\begin{align}
&\beta^{T}\sum_{j=1}^{T-1} \beta^{-j}Cr^T_{j}\| Q_{T+1}\|_{\infty} \le \beta^T \sum_{j=1}^{T-1}\beta^{-j}\left( \sqrt\frac{\rho(A_T)\beta}{\hat K}\right)^j\left|\sin((T+1-j)\theta_T)\right|\|Q_{T+1}\|_{\infty}
\\&\leq \beta^{T} \sum_{j=1}^{T-1} \beta^{-j} 2\left( \sqrt{ \frac{\rho(A_T)\beta}{\hat K}}\right)^j \frac M{1-\beta} = \star.
\end{align}
By Theorem \ref{thrm:a1}, for all $T \in \mathbb N$, it follows that
\begin{equation}\label{eq:jjj}
\sqrt{\rho(A_T)} \le \sqrt{\beta}\left(\sqrt{\hat K}+\sqrt{\hat K-\bar K} \right).
\end{equation}
We then use \eqref{eq:jjj} on $\star$ to obtain the bound
\begin{align}
    \star & \leq 2\beta^{T} \sum_{j=1}^{T-1}  \left(\frac{ \sqrt{\hat K}+\sqrt{ \hat K -\bar K}}{ \sqrt {\hat K}}\right)^j \frac M{1-\beta}
\\&\leq 2 \beta^{T} \frac{ \left(\frac{ \sqrt{\hat K}+\sqrt{ \hat K -\bar K}}{ \sqrt {\hat K}}\right)^{T}-1}{\left(\frac{ \sqrt{\hat K}+\sqrt{ \hat K -\bar K}}{ \sqrt {\hat K}}\right) -1} \frac M{1-\beta} \leq  2 \frac{ \beta^{T\varepsilon}-\beta^{T}}{\frac{ \sqrt{\hat K}+\sqrt{ \hat K -\bar K}}{ \sqrt {\hat K}}-1}{} \frac M{1-\beta},
\end{align}
where the last inequality follows from Assumption \ref{assump:4}.
It follows that for any sufficiently large $T$, and for a fixed $t<T$, we have
\begin{align}
    Cr^T_t\|\mu_t-\mu^T_t\|_{\mathrm{TV}} = |Cr^T_t| \|\mu_t-\mu^T_t\|_{\mathrm {TV}} &\leq \langle Cr^T , V(\pmb \mu|_T,\pmb \mu^{\pmb T})\rangle 
    \\&\leq \frac{1}{1-\rho(A_T)}\frac{ \beta^{(T+1)\varepsilon}-\beta^{(T+1)}}{\frac{ \sqrt{\hat K}+\sqrt{ \hat K -\bar K}}{ \sqrt {\hat K}}-1} \frac {2M}{1-\beta}\frac{K_1}{2\rho}\label{eq:sss3}.
\end{align}

The terms $r^T_t$ depend on $\theta_T$ inside a trigonometric function; thus, to control these terms, we want to rotate $\theta_T$ to the interval $(0,\pi)$ without violating \eqref{eq:sss3} so that we can use $\arccos$ to find a lower bound on $\theta_T$ by using Theorem \ref{thrm:main}. Since $\lim_{T \to \infty} \theta_T = 0,$ we define
\[
\mathcal T_1 = \{ T \in \mathbb N: \theta_T \in (0,\pi) \}, \text{ and } \mathcal T_2 = \{ T \in \mathbb N: \theta_T \in (\pi, 2\pi)  \}.
\]
We have $|\sin(j \theta_T)| = | \sin (2\pi-j\theta_T) | $ and $\cos(\theta_T) = \cos( 2\pi - \theta_T)$ for all $T \in \mathbb N$. Let $\tilde \theta_T = \theta_T$ if $\theta_T \in \mathcal T_1$ and $\tilde \theta_T = 2\pi - \theta_T$ if $T \in \mathcal T_2,$ which modifies the sequence $(\theta_T)_T$ for all sufficiently large $T$ to be in the interval $(0,\pi)$.

As a consequence of Theorem \ref{thrm:main}, for any sufficiently large $T$ we have
\begin{align*}
    \cos\left( \frac{\pi}{T}\right) \leq \cos( \theta_T) = \cos(\tilde \theta_T) \leq \cos\left( \frac{\pi}{2T-1}\right).
\end{align*}
Since $\tilde \theta_T \in(0,\pi)$ for all sufficiently large $T$ by construction, we can use the fact that $\arccos$ is a decreasing (and well-defined) function over $(0,\pi)$ to obtain
\begin{equation}\label{eq:eq1}
    \frac{\pi}{2T-1} \leq \tilde \theta_T \leq  \frac{\pi}{T}
\end{equation}
for any sufficiently large $T$.

Note that since $\cos(\theta_T)$ is strictly increasing to $1$, we must have that $(\tilde \theta_T)_T$ is a strictly monotone sequence that decreases to $0$. Thus, for any sufficiently large $T$ for which \eqref{eq:eq1} holds, it also holds that $\sin\left(t\tilde \theta_T \right) \not = 0.$ Furthermore, since $|Cr^T_t| = \left( \sqrt\frac{\rho(A_T)\beta}{\hat K}\right)^t\left|\sin((T-t)\theta_T)\right|,$ as a consequence of the inequality \eqref{eq:sss3} we obtain
\begin{align}\label{eq:sss1}
\|\mu_t -\mu^T_t\|_{\mathrm{TV}} &\leq \left(\sqrt{\frac{ \hat K}{ \rho(A_T) \beta}}\right)^t \frac{ \frac{(T-t) \tilde \theta_T}{ |\sin((T-t) \tilde \theta_T)|} \left(\beta^{T\varepsilon}-\beta^{T}\right)}{(T-t)\tilde\theta_T\left(\frac{ \sqrt{\hat K}+\sqrt{ \hat K -\bar K}}{ \sqrt {\hat K}}-1\right)}  \frac {2M}{1-\beta}\frac{K_1}{2\rho}\frac{1}{1-\rho(A_T)}.
\end{align}

In particular, using the lower bound \eqref{eq:eq1} in the inequality \eqref{eq:sss1}, for all sufficiently large $T$ we obtain
\begin{align}\label{eq:sss33}
\|\mu_t -\mu^T_t\|_{\mathrm{TV}} &\leq\frac{\frac{(T-t) \tilde \theta_T}{ \sin((T-t) \tilde \theta_T)} (2T-1)\left(\beta^{T\varepsilon}-\beta^{T}\right)}{(T-t) \pi} \frac{\frac{1}{1-\rho(A_T)}\left(\sqrt{ \frac{\hat K}{ \rho(A_T) \beta}}\right)^t2K_1M}{2\rho\left(\frac{ \sqrt{\hat K}+\sqrt{ \hat K -\bar K}}{ \sqrt {\hat K}}-1\right)\left(1-\beta\right)}=\heart _{\mathrm T}.
\end{align}

It should be clear from the expression above that to control the error bound $\heart _{\mathrm T}$, it remains to show that asymptotes of the term $\frac{(T-t) \tilde \theta_T}{ \sin((T-t) \tilde \theta_T)}$ is dominated by $O(\beta^{T\varepsilon})$. To prove this claim, we will use \eqref{eq:eq1} to control the rate of converge of the sequence $((T-t)\tilde \theta_T)_{T}$ to its potential accumulation points.
Indeed, the inequality \eqref{eq:eq1} gives us that
\begin{equation}\label{eq:eq6}
\frac{(T-t)\pi}{2T+1} \le (T-t)\tilde \theta_T \le \frac{(T-t)\pi}{T+1} \implies \frac{\pi}{2} \le \limsup_{T \to \infty} (T-t)\tilde \theta_T \le \pi.
\end{equation}
Since sine vanishes under $\pi$, we will consider the cases $\limsup_{T \to \infty} (T-t)\tilde \theta_T<\pi$ and $\limsup_{T \to \infty} (T-t)\tilde \theta_T=\pi$ separately.

We start with the case $\limsup _{T \to \infty} (T-t)\tilde \theta_T = \pi.$ Write $(T_n-t)\tilde \theta_{T_n} = \pi - a_{T_n}$. Then, along any sequence $(T_n)_{n \in \mathbb N}$ such that $\lim_{n \to \infty} (T_n-t)\tilde \theta_{T_n} = \pi$, it holds that
\begin{align*}
\frac{(T_n-t)\tilde \theta_{T_n}}{\sin((T_n-t)\tilde \theta_{T_n})} = \frac{\pi-a_{T_n}}{\sin(a_{T_n})} = O\left(\frac{\pi}{a_{T_n}}\right) = O\left( \frac{\pi}{\pi - (T_n-t)\tilde \theta_{T_n}}\right) \underbrace{=}_{(I)} O\left( \frac{\pi}{\pi-\frac{(T_n-t)\pi}{T_n+1}}\right),
\end{align*}
where we have used \eqref{eq:eq6} on $(I)$.
Therefore, we obtain
\begin{align*}
\limsup_{T \to \infty} \heart _{\mathrm T} &= \limsup_{T \to \infty} \frac{ O\left(\frac{(T+1)\pi}{1+t}\right)(2T-1)\left(\beta^{T\varepsilon}-\beta^{T}\right)}{(T-t) \pi} \frac{\frac{1}{1-\rho(A_T)}\left(\sqrt{ \frac{\hat K}{ \rho(A_T) \beta}}\right)^t2K_1M}{2\rho\left(\frac{ \sqrt{\hat K}+\sqrt{ \hat K -\bar K}}{ \sqrt {\hat K}}-1\right)\left(1-\beta\right)}
\\& = \limsup_{T \to \infty} O(\beta^{T\varepsilon})< \infty,
\end{align*}
from which the following desired bound follows:
\[
\|\mu_t -\mu^T_t\|_{\mathrm{TV}} \leq O \left(  \frac 1{(T-t)}\left(\sqrt{ \frac{\hat K}{ \rho(A_T) \beta}}\right)^t (2T-1)\beta^{T\varepsilon}\right).
\]

It remains to consider the remaining asymptote for $(T-t)\tilde \theta_T$. Suppose $\limsup _{T \to \infty} (T-t)\tilde \theta_T = \pi-\delta$ for some $\pi>\delta>0$. In this case, since $\limsup_{T \to \infty} \sin((T-t)\tilde \theta_T) > 0$, under any sequence $(T_n)_n$ that realizes $\limsup$, it directly follows that
\[
\frac{(T_n-t)\tilde \theta_{T_n}}{\sin((T_n-t)\tilde \theta_{T_n})}  = O\left( (T_n-t)\tilde \theta_{T_n}\right).
\]
Thus, for all sufficiently large $T$, we also obtain the same bound in this case:
\[
\|\mu_t -\mu^{T}_t\|_{\mathrm{TV}} \leq O \left(  \frac 1{(T-t)}\left(\sqrt{ \frac{\hat K}{ \rho(A_T) \beta}}\right)^t (2T-1)\beta^{T\varepsilon}\right),
\]
which concludes the proof of the desired finite-time error bound.

To conclude the proof, it remains to show that this finite-time error bound implies that there exists a unique infinite-horizon non-stationary MFE under any given initial state-measure.
Now, since $(\pmb \pi, \pmb \mu)$ was an arbitrary non-stationary MFE, and we know that $\lim_{T\to \infty} \mu^T_t$
is unique as a consequence of the error
bound above for any $t$, for any $(\tilde {\pmb \pi},\tilde {\pmb \mu}), (\pmb \pi, \pmb \mu) \in \mathrm{MFE}_{\mathrm{ns}}$ {starting from the same initial state-measure} we must have $\pmb \mu = \pmb {\tilde \mu}$. The uniqueness of the
corresponding $Q$-functions is now enough to conclude that $\pmb \pi = \pmb {\tilde \pi}$ as the policies are
point-mass measures under regularization. Therefore, there exists a unique MFE for MFG$_{\mathrm{ns}}$.
\end{proof}

By Remark \ref{rem:ssss}, the contractivity condition for MFG$_{\textrm{ns}}$ established in Theorem \ref{prop:2} need not hold under the assumptions of the theorem above. Nevertheless, a unique non-stationary MFE still exists and can be approximated by finite-horizon MFE.

\begin{remark}
    As a consequence of \eqref{eq:eq1}, the constant in the big $O$ term in \eqref{eq:ssss34444} can be taken independent of all $T$. Thus, there is a universal constant $C$ available such that for all $0<t \le T$ and $T \in \mathbb N$ we have 
    \[
    \| \mu_t - \mu^T_t \|_{\mathrm {TV}} \leq C \frac {1}{T-t}\left(\sqrt{ \frac{\hat K}{ \rho(A_T) \beta}}\right)^t \beta^{T\varepsilon}.
    \]
\end{remark}

As a consequence of \eqref{eq:eq1}, we see that the constraint on the parameter $t$ is bounded linearly in terms of $T$ when $T$ is sufficiently large. Under the same constraints as in Theorem \ref{thrm:7}, we can also obtain a bound on the $Q$-functions.
\begin{corollary}\label{cor:vv}
    Suppose that the assumptions of Theorem \ref{thrm:7} hold and suppose that $T$ is large enough so that \eqref{eq:eq1} holds. Let $T>s>t$ be natural numbers.
    Then, for all sufficiently large $T$ we have
    \[
    \|Q_t-Q^T_t \|_{\infty} \leq O \left(\max\left( \left(\sqrt{ \frac{\hat K}{ \rho(A_T) \beta}}\right)^s,\left(\sqrt{ \frac{\hat K}{ \rho(A_T) \beta}}\right)^t \right) (2T+1)\beta^{(T+1)\varepsilon} + \beta^{s-t}\right).
    \]
\end{corollary}
\begin{proof}
    Since Lemma \ref{lem:10} is applicable to the $Q$-functions $(Q_t)_{t \in \mathbb N}$, using the fact that $Q_t=H_{1,t}(Q_{t+1},\mu_t)$ and $Q^T_t=H_{1,t}(Q^T_{t+1},\mu^T_t)$, the desired bound can be obtained by applying the finite-time error bound presented in Theorem \ref{thrm:7} to the recursion that follows from Lemma \ref{lem:f2}:
    \begin{align*}
        \|Q_t-Q^T_t\|_{\infty} 
        &\leq 2\bar L \|\mu_t-\mu^T_t \|_{\mathrm{TV}} + \beta \|Q_{t+1}-Q^T_{t+1}\|_\infty 
        \\&\leq 2\bar L \sum_{k=t}^s \beta^{s-k}\| \mu_k-\mu^T_k \|_\mathrm{TV} + \beta^{(s+1-t)} \frac{ M}{1-\beta}.
    \end{align*}
\end{proof}

\begin{remark}
We note that, since $T$ increases linearly, we can choose larger values of $s$ as $T$ increases without violating the convergence of
\(
\sqrt{\frac{\hat K}{\rho(A_T)\beta}}^{\,s}\beta^{(T+1)\varepsilon}
\)
to $0$. Also note that Assumption \ref{assump:4}, together with Theorem \ref{thrm:main}, implies that
\(
\sqrt{\frac{\hat K}{\rho(A_T)\beta}} > 1
\)
for sufficiently large $T$. Thus, when $T$ is sufficiently large, the error bound between the finite-horizon MFE and the infinite-horizon MFE grows exponentially in $t$ and decays exponentially in $T$, for both the $Q$-functions and the state measures.
\end{remark}

Next, we show that for any infinite-horizon non-stationary MFE there exists a stationary MFE as a limit point when {the infinite-horizon MFGs are contractive, i.e., we have $\bar K + \frac{\bar L K_1}{\rho(1-\beta)}<1$. Under more relaxed conditions, in the next section, we will also show that a non-stationary infinite-horizon MFE has a stationary MFE as an accumulation point if, and only if, the state-measure flow obtained from that MFE is weakly convergent. The contraction condition $\bar K + \frac{\bar LK_1}{\rho(1-\beta)}<1$ gives a sufficient condition for the convergence of state-measures of infinite-horizon non-stationary MFE to stationary MFE.}

\begin{theorem}\label{thrm:s2}
Suppose that Assumptions \ref{assump:1} and \ref{assump:4} hold and that $\bar K + \frac{\bar LK_1}{\rho(1-\beta)} <1$. Let $\mathrm{MFG}_{\mathrm{ns}} = (X,A,C+\Omega,P,\mu_0)$ and $\mathrm {MFG}_{\mathrm s} = (X,A,C+\Omega,P)$. Let $(\pmb \pi , \pmb \mu) = ((\pi_t)_t,(\mu_t)_t) \in \mathrm {MFE}_{\mathrm {ns}}$. Then, the limit of $\pmb \pi \otimes \pmb \mu:=(\pi_t \otimes \mu_t)_t$ exists under the {weak convergence}, the limit belongs to $\mathrm{MFE}_{\mathrm s}=\{ (\pi,\mu)\},$ and it holds that
\[
\sup_{k \geq t+1}\|\mu- \mu_{k} \|_{\mathrm{TV}} \leq \left(\frac{\bar L K_1}{\rho(1-\beta)}+\bar K\right)^{t+1} \sup_{k  \geq 0} \|\mu-\mu_k\|_{\mathrm{TV}}.
\]
Furthermore, minimizers of the family of $Q$ functions $(Q_t)_t$ defined as \(Q_t = H_{1,t}(Q_{t+1},\mu_t)\), converge to the minimizer of the $Q$-function obtained under the MFE of $\mathrm{MFG}_{\mathrm s}.$ Let $s>t$. Then, it also holds that
\[
\|Q-Q_t\|_{\infty} \leq \beta^{s-t} \frac{M}{1-\beta}+\bar L \frac{ 1- \beta^{s-t}}{1-\beta}\left(\frac{\bar L K_1}{\rho(1-\beta)}+\bar K\right)^{t} \sup_{k  \geq 0} \|\mu-\mu_k\|_{\mathrm{TV}},
\]
where by $Q$ we denote the $Q$-function that corresponds to state-measure $\mu$ in the limit MFE $(\pi,\mu)$, i.e., $Q = H_{1,t}(Q,\mu)$ for any $t \ge 0$.
\end{theorem}
\begin{proof}
Note that any stationary MFE $(\pi,\mu)$ induces an infinite-horizon non-stationary MFE $(\pmb \pi, \pmb \mu)$ such that $\pmb \pi = (\pi,\pi,\cdots)$ and $\pmb \mu = (\mu,\mu,\cdots)$. Thus, the function $Q$ satisfies Lemma \ref{lem:10}. Since $\mu = H_{2,t}(Q,\mu)$ and $\mu_{t+1}=H_{2,t}(Q_t,\mu_t)$, for all $t \in \mathbb N,$ arguing as in Lemma \ref{lem:11}, we obtain
\begin{equation}\label{eq:a}
\| \mu - \mu_{t+1} \|_{\textrm {TV}} \leq  \frac{K_1}{2\rho} \| Q - Q_{t} \|_{\infty} + \bar K \| \mu - \mu_{t} \|_{\textrm {TV}}.
\end{equation}
Since $Q=H_{1,t}(Q,\mu)$ and $Q_t=H_{1,t}(Q_{t+1},\mu_t)$, by Lemma \ref{lem:f2}, it follows that
\begin{equation}\label{eq:b}
\| Q - Q_{t} \|_{\infty} \le 2\bar L\|\mu -\mu_t\|_{\mathrm{TV}}+\beta \|Q-Q_{t+1}\|_{\infty}.
\end{equation}
In particular, inequalities \eqref{eq:a} and \eqref{eq:b} imply that 
\[
\| \mu - \mu_{t+1} \|_{\textrm {TV}}\leq \left(\frac{\bar L K_1}{\rho(1-\beta)}+\bar K\right) \sup_{k \geq t} \|\mu-\mu_k\|_{\textrm{TV}},
\]
which leads to
\[
\sup_{k \geq t+1}\|\mu- \mu_{k} \|_{\mathrm{TV}} \leq \left(\frac{\bar L K_1}{\rho(1-\beta)}+\bar K\right) \sup_{k  \geq t} \|\mu-\mu_k\|_{\mathrm{TV}}.
\]
Repeating the inequality for the terms on the right-hand side above, we obtain
\begin{equation}\label{eq:ssss1}
    \sup_{k \geq t+1}\|\mu- \mu_{k} \|_{\mathrm{TV}} \leq \left(\frac{\bar L K_1}{\rho(1-\beta)}+\bar K\right)^{t+1}\sup_{k  \geq 0} \|\mu-\mu_k\|_{\mathrm{TV}}.
\end{equation}
Therefore, we have that $\limsup_{t \to \infty} \|\mu- \mu_{t} \|_{\mathrm{TV}}=0$. Uniform convergence of the minimizers of $(Q_t)_t$ follows from the uniform convergence of $(Q_t)_t$ to the $Q$-function of MFG$_{\mathrm s}$. 

Using Corollary \ref{cor:vv}, we can obtain the error bound between $Q_t$ and $Q$ as follows:
\begin{align*}
    \| Q-Q_t\|_{\infty} &\leq \beta \|Q-Q_{t+1}\|_{\infty} + 2\bar L \|\mu -\mu_t\|_{\mathrm {TV}} \\&\leq \beta^{s-t} \frac{M}{1-\beta} + 2\bar L \frac{ 1- \beta^{s-t}}{1-\beta} \sup_{k \geq t} \|\mu-\mu_t\|_{\mathrm{TV}}
    \\& \leq \beta^{s-t} \frac{M}{1-\beta}+2\bar L \frac{ 1- \beta^{s-t}}{1-\beta}\left(\frac{\bar L K_1}{\rho(1-\beta)}+\bar K\right)^{t} \sup_{k  \geq 0} \|\mu-\mu_k\|_{\mathrm{TV}}.
\end{align*}
{
It remains to show that $(\pi_t \otimes \mu_t)_t$ converges in $\mathcal P(X\times \mathcal P(A))$. Since
\(
\lim_{t \to \infty} \|Q_t - Q\|_{\infty} = 0,
\)
and since both $Q_t$ and $Q$ are strongly convex with respect to the $1$-norm, \cite[Lemma 4]{aydin2023robustness} applies to $(Q_t)_t$ and $Q$. Thus, we have
\(
\lim_{t \to \infty} \pi_t \otimes \mu_t = \pi \otimes \mu
\)
weakly, as desired.
}
\end{proof}

Our final error bound combines the results of Theorem \ref{thrm:7} and Theorem \ref{thrm:s2} to obtain a finite-time error bound between finite-horizon MFE and infinite-horizon MFE.

\begin{corollary}
    Suppose that the assumptions of Theorem \ref{thrm:7} and Theorem \ref{thrm:s2} hold. With the same notation as above, for any sufficiently large $T$ we have
    \[
    \| \mu - \mu^T_t \|_{\mathrm{TV}} \leq O \left( \left( \frac{\bar L K_1}{\rho(1-\beta)}+\bar K\right)^{t}+\frac 1{(T-t)}\left(\sqrt{ \frac{\hat K}{ \rho(A_T) \beta}}\right)^t (T+1)\beta^{T\varepsilon}\right)
    \]
    and
    \[
    \| Q-Q^T_t\|_{\infty} \leq O \left( \frac {s-t}{(T+1-t)}\left(\sqrt{ \frac{\hat K}{ \rho(A_T) \beta}}\right)^s (2T+1)\beta^{(T+1)\varepsilon} + \beta^{s-t} + \left(\frac{\bar L K_1}{\rho(1-\beta)}+\bar K\right)^{t}\right)
    \]
    where $T>s>t$.
\end{corollary}
\begin{proof}
    This result follows from Theorem \ref{thrm:7} and Theorem \ref{thrm:s2} after a straightforward triangle inequality.
\end{proof}

\section{Approximation of Infinite-Horizon Mean-field Equilibria with Finite-Horizon Equilibria}\label{sect:approx}

Let \( \mathrm{MFG}_{\mathrm T} = (X, A, c, p, \mu_0, T)\), \( \mathrm{MFG}_{\textrm{ns}} = (X, A, c, p, \mu_0)\), and \( \mathrm{MFG}_s = (X, A, c, p)\). Throughout this section, we assume that $(X,d_X)$ is a compact metric space, and $A$ is a compact convex subset of some Euclidean space $\mathbb R^m.$ In this section, we establish general convergence results regarding the approximation of infinite-horizon non-stationary and stationary MFE via finite-horizon MFE. In the previous section, in the presence of a regularizer, we obtained the rate of convergence between finite-horizon MFE and infinite-horizon MFE under the contraction of the finite-horizon MFE and Assumption \ref{assump:4}. Similarly, we have established the convergence of the non-stationary infinite-horizon MFE to stationary MFE under contraction of these MFG. In this section, we will study the convergence between these different notions of MFE under weaker assumptions, but without any error bound available. The questions that we will tackle in this section can be summarized as follows:
\begin{enumerate}
    \item In Subsection \ref{sect:4.1}, we show that accumulation points (in the time-horizon) of ``extensions'' of finite-horizon MFE are infinite-horizon non-stationary MFE.
    \item In Subsection \ref{sect:4.2}, we provide a characterization of the convergence of infinite-horizon non-stationary MFE to a stationary MFE.
    \item In Subsection \ref{sect:4.3}, as a byproduct of the results above, we will show that finite-horizon MFE can be used to approximate stationary MFE when the fixed-point iteration holds for MFE$_{\text s}$. In particular, we will show that when the fixed-point iteration holds for MFG$_{\text s},$ by learning the MFE of finite-horizon games, one can also learn an approximate MFE for MFE$_{\text s}$.
\end{enumerate}

\subsection{Approximation of Stationary Equilibria with Discounted Finite-Horizon Equilibria}\label{sect:4.1}
In this subsection, we study the relationship between the finite-horizon MFE and the infinite-horizon non-stationary MFE under the same initial-state measure.

First, we will give a description of the mode of the convergence. Let $\pi:X \to \mathcal P(A)$ be a policy and $\mu \in \mathcal P(X)$ be a state-measure. By $\pi\otimes \mu$ we denote the joint probability measure $\pi(da|x)\mu(dx)$. With this notation, we will show that for any given \((\pi^T_t,\mu^T_t)_{t=0}^{T-1}\in\mathrm{MFE}_T\), after extending them to $\mathcal P(X\times A)^{\infty}:=\prod_{t=0}^{\infty}\mathcal P(X\times A)$, accumulation points of the family $\{(\pi^T_t\otimes \mu^T_t)_t:T\in \mathbb N\}$ in $\mathcal P(X\times A)^{\infty}$, where each section $\mathcal P(X\times A)$ is endowed with the topology of the weak convergence, induces a MFE of \(\mathrm{MFG}_{\textrm{ns}}\) after disintegration. Using the definition of a MFE, if $\lim_T \pi^T_t\otimes \mu^T_t = \pi_t\otimes \mu_t$, this amounts to showing that $\pi_t$ is an optimal policy given $\pmb \mu = (\mu_t)_{t \in \mathbb N}$ at time $t$ for all $t \in \mathbb N$ and 
\[
\mu_{t+1}(\cdot) = \int_X\int_A p(\cdot|x,a,\mu_t)\pi_t(da|x)\mu_t(dx).
\]

To achieve this, we will first extend the MFE of \(\mathrm{MFG}_T\) to infinite flows, since the MFE of \(\mathrm{MFG}_{\textrm{ns}}\) are defined as infinite flows. Then, we will extract convergent subsequences of $Q$-functions and state-measures from these extended finite-horizon MFE to pass to the infinite-horizon case. We will need the convergence of $Q$-functions to ensure that accumulation points of each $(\pi^T_t \otimes \mu^T_t)_T$ concentrates on the optimal state-action pairs of the limiting $Q$-function, so that the limiting stochastic kernel $\pi_t$ obtained from disintegration is indeed an optimal policy. Here, the optimal state-action pairs at time $t$ refer to pairs of states and actions, $(x,a),$ at time $t$ such that $a$ is a minimizer of the corresponding $Q$-function at time $t$ under the state $x$.

To preserve the dynamic programming structure between the convergent $Q$-functions required for a MFE and the evolution structure required from state-measures (see \eqref{eq:stat-evolve}), we will have to interchange limits and integrals, for which we need sufficient regularity conditions on the transition probability $p$. Furthermore, we will also have to control the behavior of the limiting optimal policies, for which we need a regularity conditions on the one-stage cost function $c$. For this purpose, throughout this subsection, we impose the following assumption on our system components:
\begin{assumption}\label{assump:2}
Equip $X \times A \times \mathcal P(X)$ is endowed with the product topology. For any convergent sequence $(x_n,a_n,\mu_n) \subset X \times A \times \mathcal P(X)$ such that 
\[
\lim_{n \to \infty} (x_n,a_n,\mu_n) = (x,a,\mu) \in X \times A \times \mathcal P(X)
\]
the following holds:
\begin{enumerate}
    \item  The one stage cost function $c$ satisfies $\lim_n c(x_n,a_n,\mu_n)=c(x,a,\mu).$
    \item The transition probability $p$ satisfies $\lim_n p(\cdot |x_n,a_n,\mu_n) = p(\cdot|x,a,\mu)$ weakly.
\end{enumerate}
\end{assumption}

We refer to the assumptions above as \emph{continuous convergence} of $c$ (resp. \emph{weakly continuous convergence} of $p$), which can be shown easily to coincide with the uniform convergence on compact sets on their respective spaces using the separability. Since we assumed that our state space $X$ is compact, we will work with the $1$-Wasserstein metric over $\mathcal P(X)$, which metrizes the weak convergence. We recall that the $1$-Wasserstein metric is defined as 
\[
W_1( \mu, \nu ) := \sup_{ \|g \|_{\textrm{Lip}} \leq 1} \left | \int_X g(x)\mu(dx) - \int_X g(x) \nu(dx) \right|,
\]
where $\|g \|_{\textrm{Lip}}$ denotes the Lipschitz coefficient of the function $g: X \to \mathbb R$. It follows that $(\mathcal P(X),W_1)$ is a compact space. Thus, we can extract convergent subsequence of state-measures in our setting as it is.

To ensure that from a given family of $Q$-functions obtained under a MFE we can extract a convergent subsequence, we will impose the following Lipschitz conditions on the system components to ensure that they all lie on a common compact space.
\begin{assumption}\label{assump:3}
    \begin{enumerate}[(a)]
    \item The one-stage reward function \( c \) satisfies the following Lipschitz bound:
    \[
    \sup_{\substack{x,\hat{x}\in X\\ x\neq \hat{x}}}\sup_{(a,\mu) \in A \times \mathcal P(X)}\frac{\left| c(x,a,\mu) - c(\hat{x},a,\mu) \right|}{d_X(x,\hat x)}
    \leq \tilde L_1 .
    \]

    \item The stochastic kernel \( p(\cdot | x,a,\mu) \) satisfies the following Lipschitz bound:
    \[
    \sup_{\substack{x,\hat{x}\in X\\ x\neq \hat{x}}} \sup_{(a,\mu) \in A \times \mathcal P(X)}\frac{W_1(p(\cdot | x,a,\mu), p(\cdot | \hat{x},a,\mu))}{d_X(x,\hat x)}
    \leq \tilde K_1.
    \]

    \item We have $\beta \tilde K_1 <1$.
\end{enumerate}
\end{assumption}
{
\begin{remark}
Compared with the regularized case, in which Lipschitz continuity properties are also imposed on the system in Theorem~\ref{prop:2}, the condition $\beta \tilde K_1 < 1$ is less restrictive. Moreover, we do not require the additional strong uniqueness condition on the cost function.
\end{remark}
}
The condition $\beta \tilde K_1 < 1$ is imposed to ensure the existence of a compact domain for the $Q$-functions that is invariant under $H_{1,t}(\cdot,\mu)$ for every $\mu \in \mathcal P(X)$. The condition $\beta \tilde K_1 < 1$ will play a similar role to the implicit condition $\beta \frac {K_1}2<1$ that we have used in the previous section; see Remark \ref{rem:discount}. This is sufficient to guarantee that for any $Q$-function arising from finite-horizon and infinite-horizon MFE, say $Q$, the map
\(
x \mapsto \min_{a \in A} Q(x,a)
\)
is Lipschitz on $X$ with Lipschitz constant that is uniform across all the possible $Q$-functions. As a consequence, the collection of such $Q$-functions is contained in a compact set. {To prove this, we will show the sequential compactness of the $Q$-functions obtained under a given family of state-measures, which is equivalent to compactness in the case of compact metric spaces. To argue this, for a sequence of $Q$-functions $(Q^n_t)_{t=0}^T$, we will use Assumption \ref{assump:3} to extract convergent subsequences of $(\min_{a \in A}Q^n_t(\cdot,a))_{t=1}^T$ and use the compactness of $\mathcal P(X)$ to extract a convergent subsequence of state-measures.
}

In contrast with the infinite-horizon case, convergence of the $Q$-functions is not automatic in the finite-horizon setting: even if the state measures of finite-horizon MFEs converge to those of an infinite-horizon MFE, the corresponding $Q$-functions need not converge. The source of this difficulty is the finite-time termination of the finite-horizon $Q$-functions in the finite-horizon case. To ensure convergence, we therefore impose the additional requirement that the $Q$-functions belong to a compact set.

In general, the condition $\beta \hat K_1<1$ is not restrictive. For instance, when $X$ is a finite discrete space equipped {with the discrete metric}, one can always take $\tilde K_1\le 1$, as $W_1(\mu,\nu) \le1$ for all $\mu,\nu \in \mathcal P(X)$ in this case. Indeed, for any coupling $\mu \otimes \nu$ of $\mu$ and $\nu$,
\[
\int_{X \times X} 1_{\{x \not = \tilde x\}} \mu\otimes \nu(dx,d\tilde x) \le 1,
\]
which implies $W_1(\mu,\nu)\le 1$.

\begin{remark}
    If $\beta \tilde K_1<1$ is not satisfied, one can impose a stronger regularity condition
    \[
    \| p(\cdot|x,a,\mu)-p(\cdot|\hat x,\hat a,\mu)\|_{\mathrm{TV}} \le \tilde K_1\left( d_X(x,\hat x) + \| a - \tilde a \| \right),
    \]
    on $p$ to ensure that $Q$-functions will lie in a compact set. Since the bounded functions work as test functions under the total-variation norm, this will be enough to ensure that $\min_{a \in A}Q(\cdot,a)$ are Lipschitz continuous with a uniform Lipschitz bound across all the feasible $Q$-functions.
\end{remark}

Since $(\mathcal P(X),W_1)$ is a compact metric space, Assumptions \ref{assump:2} and \ref{assump:3} imply that the cost function $c$ is bounded above by some constant $M$. Throughout the rest of this section, we assume that Assumptions \ref{assump:2} and \ref{assump:3} hold.

Let \(\{(\pi_t^T, \mu_t^T)\}_{t=1}^T \in \mathrm{MFE}_{\mathrm T}\). We extend this MFE to an infinite flow, \(\{(\tilde{\pi}_t^T, \tilde{\mu}_t^T)\}_{t=1}^\infty\), by defining
\[
\tilde{\pi}_t^T = 
\begin{cases}
\pi_t^T, & \text{for } t \leq T, \\
\pi_T^T, & \text{for } t > T,
\end{cases}
\quad
\tilde{\mu}_t^T = 
\begin{cases}
\mu_t^T, & \text{for } t \leq T, \\
\mu_T^T, & \text{for } t > T.
\end{cases}
\]
We will call the pairs $(\tilde \pi^T_t,\tilde \mu^T_t)_t$ as extended MFE, and denote the set of extended MFE as $\mathrm{EMFE}_T$.
We remark that our results in this subsection will be asymptotic; hence, the exact extension of \(\{(\pi_t^T, \mu_t^T)\}_{t=1}^T\) does not affect our conclusions.

The following result is the main result of this subsection and shows that any accumulation point of the joint probability measures generated by the flow $(\tilde \pi^T_t ,\tilde \mu^T_t)_{t=1}^{\infty}$ in the horizon $T$ indeed gives us a MFE for the non-stationary game after disintegration.

\begin{theorem}\label{thrm:1}
Suppose Assumptions \ref{assump:2} and \ref{assump:3} hold. Any accumulation point of the family 
\[
\{(\tilde \pi^T_t \otimes \tilde \mu^T_t)_t: (\tilde \pi^T_t \otimes \tilde \mu^T_t)_t \in\mathrm{EMFE}_{\mathrm T},   T\geq 1\} \subset \mathcal P(X\times A)^{\infty}
\]
in $\mathcal P(X \times A)^{\infty}$ under the product topology generated by the weak convergence of probability measures leads to a non-stationary MFE for $\mathrm {MFG}_{\mathrm{ns}}$ after disintegration.
\end{theorem}

\begin{remark}
When $\mathrm{MFG}_{\mathrm{ns}}$ admits a unique MFE, Theorem \ref{thrm:1} implies that any family of finite-horizon MFEs converges to the infinite-horizon non-stationary MFE. Moreover, as observed in Remark \ref{rem:1} and Theorem \ref{thrm:s2}, there are situations in which the convergence of the iterations for $\mathrm{MFG}_T$ can be justified for every $T$, even though the contraction condition fails for $\mathrm{MFG}_{\mathrm{ns}}$. In addition, fixed-point iteration is generally intractable for $\mathrm{MFG}_{\mathrm{ns}}$ because of the infinite-horizon nature of the problem. Therefore, when $\mathrm{MFE}_{\mathrm{ns}}$ is a singleton, Theorem \ref{thrm:1} shows that finite-horizon MFEs provide a tractable approximation scheme for $\mathrm{MFE}_{\mathrm{ns}}$ under weaker learning conditions than those required for $\mathrm{MFG}_{\mathrm{ns}}$. We also note that, in the setting of Section \ref{sect:fixpt}, fixed-point iteration for finite-horizon MFGs may converge even when $\bar K>1$, whereas this is not possible for infinite-horizon MFGs, in either the non-stationary or the stationary setting.
\end{remark}

Before proceeding to the proof of Theorem \ref{thrm:1}, we will show that $Q$-functions obtained under MFE (both in finite and infinite-horizon) lie in a common compact metric space and establish a technical lemma to ensure that limiting joint probability measures obtained from finite-horizon MFE gives an optimal policy after disintegration.

The main reason why we have assumed Lipschitz continuity of the system components in Assumption \ref{assump:3} is the following lemma, which shows that all the $Q$-functions we obtain via dynamic programming (in both finite and infinite horizon settings) belong to a common compact subset of $(C(X\times A),\| \cdot \|_{\infty})$, the space of all continuous functions over $X \times A$ equipped with the uniform norm, which will be essential for our subsequent arguments. The constants in the assumption above will not be important whatsoever regarding our results, unlike in the case of fixed-point iterations, where we required our Lipschitz coefficients to be sufficiently small.

\begin{lemma}\label{lem:f6}
    Suppose Assumptions \ref{assump:2} and \ref{assump:3} hold. For all $T \in \mathbb N$, let $(\mu^T_t)_t \in \{\mu_{0,T}\}\times \mathcal P(X)^{T-1}$ be a family of state-measures that induce finite-horizon MFE starting from $\mu_{0,T}$. Let $(\tilde \mu^T_t)_t$ be the corresponding EMFE. By $(Q^T_t)_t$ denote the corresponding family of $Q$-functions to $(\mu^T_t)_t$. {Then, for each $t \in \mathbb N$, the sequences $(Q^T_t)_{T \ge t}$ lie in a (common) separable compact space $(\mathcal C,\|\cdot\|_{\infty})$.}
    
    In particular, if $(\tilde \mu^T_t)_t$ is weakly convergent in $\mathcal P(X)^{\infty}$ to some $(\hat \mu_t)_t \in \{\hat \mu_0\}\times \mathcal P(X)^{\infty}$, then if $(\hat Q_t)_t$ is the family of $Q$-functions that correspond to $(\hat \mu_t)_t$, we have that
    \(
    \lim_{T \to \infty}\| Q^T_t-\hat Q_t\|_{\infty} =0.
    \)
\end{lemma}
\begin{proof}
Due to the length of the proof, we postpone the proof to Appendix \ref{sect:lem:f6}.
\end{proof}

\begin{notation}
    Unlike the previous subsection, the $Q$-functions obtained under a finite-horizon MFE might not belong to a common compact space with the $Q$-functions of the infinite-horizon non-stationary MFE. Thus, we are deducing the convergence properties of these $Q$-functions directly by using the convergence properties of the corresponding state-measures.
\end{notation}

Next, we will provide a technical statement that ensures that limits of our extended MFE gives a non-stastionary MFE after disintegration.
Under a convergent sequence of joint-probability measures, since the flow of state-measures can be obtained as a marginal of these joint-probabilities, it automatically holds that state-flows also converge weakly to that of the target joint probability measure. However, it is well known that convergence of the policies might not hold in general. When MFGs are regularized, as in Section \ref{sect:fixpt}, due to $(\hat \pi_t)_t$ being a flow of Dirac measures, we can further show that we can also establish the convergence of the policies $(\tilde \pi^T_t)_t$ to $(\hat \pi_t)_t$; see \cite[Lemma 4]{aydin2023robustness}.

Although convergence of the joint probability measures implies the convergence of their marginals, such convergence provides no information regarding the support of the limit, which we will need to verify that limiting flows obtained from the joint probability measures generated by finite-horizon MFE are indeed MFE. For this reason, we will need the following technical lemma, which is inspired from \cite[Proposition 3.9]{SaBaRaSIAM}.

\begin{lemma}\label{lem:1}
Suppose Assumptions \ref{assump:2} and \ref{assump:3} hold. Let $(Q_n)_n$ be a family of uniformly bounded continuous real-valued functions over $X\times A$ such that $\lim_{n \to \infty} Q_n(x_n,a_n) = Q(x,a)$ for all $(x_n,a_n)_n,(x,a) \in X \times A$ such that $\lim_n (x_n,a_n) = (x,a)$. Then, if $\pi_n \otimes \mu_n \in \mathcal P(X \times A)$ concentrates on the optimal state-action pairs of the continuous function $Q_n : X \times A \to \mathbb R$ for all $n$, then the weak limit of $(\pi_n \otimes \mu_n)_n$ in $\mathcal P(X\times A)$ also concentrates on the optimal state-action pairs of the function $Q$, provided the limit  of $(\pi_n \otimes \mu_n)_n$ exists.
\end{lemma}
\begin{proof}
    Due to the length of the proof, we postpone the proof to Appendix \ref{sect:lem:1}.
\end{proof}

Lastly, to show that limits of state-measures satisfy the desired evolution condition, we will need the following extension of the dominated convergence theorem.
\begin{lemma}\label{lem:2}
    Suppose Assumptions \ref{assump:2} and \ref{assump:3} hold. Let $(\pi_t)_t$ be a family of policies such that $\pi_t: X \to \mathcal P(A)$. Let $(\mu_t)_t \subset \mathcal P(X)$. Suppose that $\lim_{t \to \infty} \pi_t \otimes \mu_t$ is weakly convergent in $\mathcal P(X \times A)$ (in the sense of probability measures). Then, we have
    \[
    \lim_{t \to \infty} \int_{X \times A} p(\cdot|x,a,\mu_t)\pi_t \otimes \mu_t(da,dx) = \int_{X \times A} p(\cdot|x,a,\lim_{t \to \infty}\mu_t) \lim_{t\to \infty} \pi_t \otimes \mu_t(da,dx),
    \]
    where the limits are taken under the topology of weak convergence in $\mathcal P(X)$.
\end{lemma}
\begin{proof}
    We postpone the proof to Appendix \ref{sect:lem:2}.
\end{proof}

With the tools introduced above, below we provide the promised proof of Theorem \ref{thrm:1}.

\begin{proof}[Proof of Theorem \ref{thrm:1}]
Let $(\tilde \pi^T_t,\tilde \mu^T_t)_{t=0}^{\infty}$ denote the extended MFE obtained from the $T$-horizon MFG. Let $(Q^T_t)_{t=0}^T$ denote the family of action-value functions that satisfy
\[
Q^T_t(x,a) = E^{\tilde \pi^T_t}\left[ \sum_{n=t}^T \beta^{n-t
}c(x_n,a_n,\tilde \mu^T_n)\bigg|x_t=x, a_t=a \right]
\]
for all $t\leq T.$ We extend them to the infinite-horizon as
\[
\tilde Q^T_t(x,a) = Q^T_t(x,a) 1_{\{t \leq T\}}(t) + Q^T_{T}(x,a)1_{\{t >T\}}(t).
\]
The family $\{(\tilde \pi^T_t\otimes \tilde \mu^T_t)_t:T \geq 1\}$ lies in the compact set $\mathcal P(X\times A)^{\infty}$ as $X$ and $A$ are compact. Since for any $t \in \mathbb N$, $(\tilde Q^T_t)_T$ also lies in a compact set by Lemma \ref{lem:f6}, as a consequence of the Arzel\`a-Ascoli theorem, using a diagonal extraction, for any subsequence of $\{(\tilde \pi^T_t \otimes \tilde \mu^T_t)_t:T \geq 1\} \times \{ (Q^T_t)_t : T\geq 1 \}$ we can find a convergent subsequence in $\mathcal P(X\times A)^\infty \times \prod_{t=0}^{\infty} \mathcal C$, which we will write using the horizon indices $(T_n)_n$. {Here, the set $\mathcal C$ is as defined in Lemma \ref{lem:f6}}. Let $\lim_{n \to\infty} ((\tilde \pi^{T_n}_t\otimes \tilde \mu^{T_n}_t)_t,(\tilde Q^{T_n}_t)_t)=:( (\hat \pi_t \otimes \hat \mu_t)_t, (Q_t)_t).$ 

Recall that for any $t+1< T_n$ we have
\[
\tilde Q^{T_n}_t(x,a) = c(x,a,\tilde \mu^{T_n}_t) + \beta \int_X \min_{b \in A} \tilde Q^{T_n}_{t+1}(y,b)p(dy|x,a,\tilde \mu^{T_n}_t),
\]
and
\[
\tilde \mu^{T_n}_{t+1}(\cdot) = \int_X \int_A p(\cdot | x,a,\tilde \mu^{T_n}_t)\tilde \pi^{T_n}_t(da | x)\tilde \mu^{T_n}_t(dx).
\] 
Since $(\tilde Q^{T_n}_t)_n$ converges uniformly to $Q_t$, using the continuous weak convergence of $p$ (Assumption \ref{assump:3}), we satisfy the running conditions of \cite[Theorem 3.5]{serfozo1982convergence}, which implies that
\[
Q_t(x,a) = c(x,a,\hat \mu_t) + \beta \int_X \min_{b \in A} Q_{t+1}(y,b)  p(dy | x,a,\hat\mu_t).
\]
As a consequence of Lemma \ref{lem:2}, we also have the evolution 
\[
\hat \mu_{t+1}(\cdot) = \int_X \int_A p(\cdot | x,a,\hat \mu_t)\hat \pi_t(da | x)\hat \mu_t(dx)
\]
satisfied by the flow $(\hat \mu_t)_t$.

The joint probability measure $\tilde \pi^{T_n}_t \otimes \tilde \mu^{T_n}_t$ concentrates on the optimal state-action pairs of $\tilde Q^{T_n}_t$ for $t<T_n$ by definition, and thus by Lemma \ref{lem:1} it follows that $\pi_t \otimes \mu_t$ must concentrate on the optimal state-action pairs of $Q_t$ for all $t$.
Therefore, by \cite[Theorem 3.6]{SaBaRaSIAM} it follows that $(\pi_t \otimes \mu_t)_{t=0}^{\infty}$ must be an MFE of MFG$_{\textrm{ns}}$, as desired.
\end{proof}

\subsection{Approximation of $\text{ MFE}_{\text{s }}$ with $\text{ MFE}_{\text{ns}}$}\label{sect:4.2}
{
Theorem \ref{thrm:7} shows that the gap between the state measure of a finite-horizon MFE can be of the form $O(a^tb^T)$, where $a>1$ and $0<b<1$. Thus, we might not be able to approximate an infinite-horizon non-stationary MFE via finite-horizon MFE by solving a finite-horizon for a fixed-horizon length $T$. A tractable approximation scheme for infinite-horizon non-stationary MFE would be to construct an approximate infinite-horizon non-stationary MFE by solving a finite-horizon MFG of horizon length $T$.}

To identify situations in which such an approximation is possible, in this subsection we study the relationship between MFE$_{\mathrm s}$ and MFE$_{\mathrm{ns}}$. In particular, our goal is to establish conditions under which stationary MFE can be approximated by finite-horizon MFE uniformly over time in the general setting of this section. This will also allow us to obtain a tractable estimation scheme for infinite-horizon non-stationary MFE as explained above. These results also provide tractable approximation schemes for stationary MFE based on finite-horizon MFE.

The property that we will study to relate $\mathrm{MFE}_{\mathrm{ns}}$ and $\mathrm{MFE}_{\mathrm s}$ is the so-called \emph{asymptotically discount optimality} condition.
\begin{definition}
    Let $\mathrm {MFG}_{\mathrm{ns}}$ be a non-stationary infinite-horizon MFG $(X,A,c,p,\mu_0)$. We say that $(\hat \pi_t,\hat \mu_t)_{t=0}^{\infty}$ is an \emph{asymptotically discount optimal MFE} ($\mathrm{ADOMFE}$ for short) in $\mathrm {MFG}_{\mathrm{ns}}$ if $\lim_{t \to \infty} \hat \pi_t \otimes \hat \mu_t = \pi \otimes \mu$ exists in $(\mathcal P(X\times A),W_1)$ for some $\pi\otimes \mu \in \mathcal P(X \times A)$ that is an MFE for the stationary MFG $(X,A,c,p)$.
\end{definition}

\begin{remark}
    In Theorem \ref{thrm:s2}, it is shown that when the stationary MFG is contractive, infinite-horizon non-stationary MFE obtained under any initial state-measure is indeed ADOMFE.
\end{remark}

Let $(\hat \pi_t,\hat \mu_t)_t$ be an $\mathrm{ADOMFE}$. Then, by definition, the state marginal flow $(\hat \mu_t)_t$ converges weakly to some $\mu \in \mathcal P(X)$; that is, $
\hat \mu_t \Rightarrow \mu$ as $t \to \infty.$
At first glance, this suggests that if the family of joint probability measures
\(
\{(\tilde \pi_t^T \otimes \tilde \mu_t^T)_T : t \geq 0\}
\)
converges asymptotically to $(\hat \pi_t\otimes\hat \mu_t)_t$ as $T\to \infty$ in $\mathcal P(X\times A)^{\infty}$, then the corresponding state marginals should inherit the same asymptotic behavior. More precisely, if
\[
\lim_{T\to\infty} \tilde \pi_t^T \otimes \tilde \mu_t^T
=
\hat \pi_t \otimes \hat \mu_t
\qquad \text{for all } t,
\]
then, for sufficiently large $T$, the terminal marginal $\tilde \mu_T^T$ should be close, in the weak topology on $\mathcal P(X)$, to $\mu$ for all $t \geq T$, since the flow $(\hat \mu_t)_t$ accumulates around $\mu$.

However, if $(\hat \pi_t,\hat \mu_t)_t$ is not an $\mathrm{ADOMFE}$, the sequence $(\hat \mu_t)_t$ need not converge. Furthermore, the set of accumulation points of $(\hat \mu_t)_t$ in $(\mathcal P(X),W_1)$ may be sparse. Consequently, approximating {any} infinite-horizon MFE by a finite-horizon MFE in a tractable manner becomes significantly more subtle and may require additional arguments. We also emphasize that, in general, the limit
\(
\lim_{T\to\infty} \tilde \pi_t^T(\cdot | x)
\)
does not necessarily exist {even when $(\tilde \pi^T_t \otimes \tilde \mu^T_t)_t$ is convergent in $\mathcal P(X \times A)^{\infty}$ as $T \to \infty$.}

\begin{remark}
The concept of $\mathrm{ADOMFE}$ is inspired by the asymptotically discount optimal policies of MDPs, which are known to exist for any infinite-horizon discounted cost MDP and are widely used in the framework of adaptive learning \cite{hernandez2012adaptive} and reinforcement learning \cite{adelman2024thompson}.
\end{remark}

Unlike the previous subsection, in this subsection we will not need Assumption \ref{assump:3}. In the infinite-horizon case, the $Q$-functions that correspond to a flow of state-measures are automatically convergent if the flow of state-measures are convergent under Assumption \ref{assump:2}, which allows us to bypass the need for identifying a compact set for our $Q$-functions. First, define
\[
H_1(Q,\mu)(x,a) = c(x,a,\mu) + \beta \int_X \min_{b \in A} Q(y,b)p(dy|x,a,\mu).
\]

\begin{lemma}\label{lem:f7}
        Suppose Assumption \ref{assump:2} holds. For all $n\in \mathbb N$, let $\pmb \mu_n \in \{\mu_{0,n}\}\times\mathcal P(X)^{\infty},$ where $(\mu_{0,n})_n$ is weakly convergent in $\mathcal P(X)$. For all $n \in \mathbb N$ and $t \in \mathbb N$, by \(Q_{n,t}=H_{1}(Q_{n,t+1},\mu_{n,t})\)
    denote the corresponding family of $Q$-functions to $\pmb \mu_n$, which are uniquely determined by the flow $\pmb \mu_n$. Then, if $\pmb \mu_n$ is weakly convergent in $\mathcal P(X)^{\infty}$ to some $\pmb \mu=(\mu_t)_t$, we have that 
    \[
    \lim_{n \to \infty} \|Q_{n,t}-Q_t\|_{\infty} =0,
    \]
    where $(Q_t)_t$ is the sequence of $Q$-functions that correspond to the flow $\pmb \mu$, i.e.,
    \(
    Q_t = H_{1}(Q_{t+1},\mu_t).
    \)
\end{lemma}
\begin{proof}
    See Appendix \ref{sect:f7}.
\end{proof}

Let $(\hat \pi_t,\hat \mu_t)_t \in \text{MFE}_{\text {ns}}.$ The following proposition shows that when $\lim_t (\hat \pi_t \otimes \hat \mu_t)_t$ exists in $\mathcal P(X\times A)$ under the weak convergence, the limit must belong to MFE$_{\text s}$, which shows that $(\hat \pi_t,\hat \mu_t)_t$ is an ADOMFE.

\begin{theorem}\label{thrm:a}
    Suppose Assumption \ref{assump:2} holds. Let $(\hat \pi_t,\hat \mu_t)_t \in \mathrm{MFE}_{\mathrm {ns}}$. If $\lim_{t \to \infty} (\hat \pi_t \otimes \hat \mu_t)_t$ exists in $(\mathcal P(X\times A),W_1)$, then $\lim_{t \to \infty} (\hat \pi_t \otimes \hat \mu_t)_t \in \mathrm{MFE}_{\mathrm s}$ (after disintegration).
\end{theorem}
\begin{proof}
    Let $\pi\otimes \mu = \lim_t \hat \pi_t \otimes \hat \mu_t$. Recall that, for all $t \in \mathbb N$ we have
    \begin{equation}\label{eq:ns}
    \hat \mu_{t+1}(\cdot) = \int_X \int_A \,p(\cdot | x,a,\hat \mu_t)\,\hat \pi_t(da | x) \,\hat \mu_t(dx).
    \end{equation}
    Then by Lemma \ref{lem:2}, we obtain that 
    \begin{equation}\label{eq:ns1}
    \mu (\cdot) = \int_X \int_Ap(\cdot | x,a,\mu)\, \pi(da | x) \mu(dx).
    \end{equation}
    Let $Q_t$ be the family of $Q$-functions defined via the relation $H_1(Q_{t+1},\hat \mu_t) = Q_t$.
    
    Now consider the $k$-slided infinite-horizon non-stationary MFE generated by the flow $(\rho^k_t,\nu^k_t)_t$ for any $k \in \mathbb N$, where $\nu^k_t := \tilde \mu_{t+k}$ and $\rho^k_t := \tilde \pi_{t+k},$ $t \ge 0$. The corresponding $Q$-functions to the $k$-slided state-measures $(\nu^k_t)_t$ are indeed of the form $\tilde Q^k_t := Q_{t+k}$. It then follows that for any given $t\in \mathbb N$, $\lim_{k \to \infty} \rho^k_t \otimes \nu^k_t$ exists weakly in $\mathcal P(X\times A)$. Thus, Lemma \ref{lem:f7} is applicable to the slided MFE. Using Lemma \ref{lem:f7} on the family $\{(\rho^k_t,\nu^k_t)_t:k \in \mathbb N\}$, we obtain that there exists $\tilde Q_t := \lim_{k \to \infty} \tilde Q^k_t$, where the limit is under the uniform norm over $X \times A$, such that $\tilde Q_t = H_1(\tilde Q_{t+1},\mu).$
    Using the contractivity of the Bellman operator $H_1(\cdot,\mu)$ under the discount $\beta<1$, 
    we see that $\tilde Q_t = \tilde Q$ for all $t \in \mathbb N$ for some $Q$-function $\tilde Q$. Thus, $\lim_{k \to \infty} \tilde Q^k_t=\tilde Q$ for all $t \in \mathbb N$. Uniqueness of the fixed point of the Bellman operator indeed implies that $\tilde Q = Q$. By Lemma \ref{lem:1}, we obtain that $\pi\otimes \mu$ concentrates on the optimal state action pairs of the function $Q$, which in turn implies that $(\pi,\mu)$ is a stationary MFE, by definition.
\end{proof}

In case of MDPs, obtaining a stationary policy from finite-horizon optimal policies is often not difficult due to the continuous convergence of the $Q$-functions to a stationary one, which implies that accumulation points of the minimizers sampled from finite-horizon optimal policies are optimal for the stationary $Q$-function. However, in the case of MFGs, the system components also depend on the state-measures. The next theorem shows that the behavior of the state-measure flow is indeed crucial when relating infinite-horizon non-stationary MFE and stationary MFE.

It is clear from the definition of an $\textrm{ADOMFE}$ that the state-measure flow obtained from an $\textrm{ADOMFE}$ is also weakly convergent. The next proposition shows that the converse is also true when the optimal policies of MFE obtained from MFG$_s$ are Dirac measures; that is, an MFE is an $\textrm{ADOMFE}$ if, and only if, the state-measure flow obtained from the MFE converges weakly.

\begin{theorem}\label{thrm:q1}
Suppose that Assumption \ref{assump:2} holds. Let $(\hat \pi_t,\hat \mu_t)_t \in \mathrm {MFE}_{\mathrm {ns}}$ be such that $(\hat \mu_t)_t$ is weakly convergent in $\mathcal P(X)$. Suppose that for any $(\pi,\mu) \in \mathrm{MFE}_{\mathrm s},$ there exists a unique policy $\pi$ that corresponds to $\mu$. Then, the following are equivalent:
\begin{enumerate}
    \item $(\hat \pi_t,\hat \mu_t)_t$ is an $\mathrm{ADOMFE}$,
    \item $(\hat \mu_t)_t$ is weakly convergent.
\end{enumerate}
\end{theorem}
\begin{proof}
As discussed prior to the statement of the theorem, the implication ``$1 \implies 2$'' is straightforward. We will now show that ``$2 \implies 1$''. Thus, suppose that $(\hat \mu_t)_t$ is weakly convergent in $\mathcal P(X)$ to some $\mu^*$. In this case, since $\lim_{t \to \infty}\hat \mu_t = \mu^*$ weakly, and for $t \ll T$, $\hat \mu_t$ satisfies \eqref{eq:ns}, we obtain that $\mu^*$ satisfies \eqref{eq:ns1}, by Lemma \ref{lem:2}. Thus, the difficult part will be to control the ``convergence'' of the optimal policies, for which, we will use Lemma \ref{lem:1}.

 Under our assumption, by \cite{SaBaRaSIAM}, there exists a MFE for MFG$_{\textrm{ns}}$, say $(\hat \pi_t ,\hat \mu_t)_{t=0}^{\infty}$. Then, as $\prod_{t=0}^{\infty} \mathcal P(X\times A)$ is compact, any subsequence of the family of joint probability measures $(\hat \pi_t \otimes \hat \mu_t)_t$ has a convergent subsequence of the joint probability measure, say $(\hat \pi_{n_t} \otimes \hat \mu_{n_t})_t$. Let 
\[
Q_{t}(x,a) = c(x,a,\hat \mu_t) + \beta \int_X \min_{a \in A} Q_{t+1}(y,a) p(dy|x,a,\hat \mu_t) 
\]
for all $t \geq 0$. Since $(\hat \mu_t)_t$ is weakly convergent, by Lemma \ref{lem:f7}, we also obtain that the sequence of $Q$-functions $(Q_t)_t$ is convergent. Let $\lim_t Q_t = Q^*.$ Then, by Lemma \ref{lem:1}, it holds that $\lim_t \hat \pi_{n_t} \otimes \hat \mu_{n_t} :=  \pi^*\otimes \mu^*$ concentrates on the optimal state-action pairs of $Q^*$, meaning that $( \pi_*, \mu^*)$ is a stationary MFE. Since for any $\mu$ there exists a unique policy $\pi$ that corresponds to it, it follows that the weak limit
\[
\lim_{ t \to \infty } \hat \pi_t \otimes\hat \mu_t = \pi^* \otimes \mu^*
\]
must hold as desired due to the sequential compactness of $\mathcal P(X \times A)$ under the weak convergence.
\end{proof}

As an application of the ADOMFE property, we show that if a family of EMFE converges to an ADOMFE, then the convergence is uniform over $t\in \mathbb N$ in the horizon length $T$.

\begin{proposition}\label{prop:5}
Suppose that Assumptions \ref{assump:2} and \ref{assump:3} hold. Let $(\tilde \pi^T_t,\tilde \mu^T_t)_{t=0}^{\infty} \in \mathrm{EMFE}_{\mathrm T}$ for all $T$. If $(\lim_{T} \tilde \pi^T_t \otimes \tilde \mu^T_t)_t = (\hat \pi_t \otimes \hat \mu_t)_t\in \mathrm{ADOMFE},$ then for all $\varepsilon > 0$ there exists a sufficiently large $\tilde t(\varepsilon)$, $\tilde T(\tilde t(\varepsilon))$, and $T(\varepsilon)$ such that 
    \[
    W_1( \tilde \pi^{\tilde T(\tilde t(\varepsilon))}_{\tilde t(\varepsilon)} \otimes \tilde \mu^{\tilde T(\tilde t(\varepsilon))}_{\tilde t(\varepsilon)}, \hat \pi_t \otimes \hat \mu_t) < \varepsilon
    \]
    for all $t \ge T(\varepsilon)$.
{In particular, tail of the ADOMFE $(\hat \pi_t \otimes \hat \mu_T)_{t=0}^{\infty}$ can be uniformly approximated by the EMFE $\{(\tilde \pi^T_t,\tilde \mu^T_t)_{t=0}^{\infty} : T \in \mathbb N\}$.}
\end{proposition}
\begin{proof}
Let $\lim_t \hat \pi_t \otimes \hat \mu_t = \pi \otimes \mu.$ Then, for a given $\varepsilon>0$, we have
\[
W_1( \pi \otimes \mu , \hat \pi_t \otimes \hat \mu_t) < \frac{\varepsilon}4
\]
for all sufficiently large $t\ge T(\varepsilon)$ for some $T(\varepsilon)$. Thus, picking a sufficiently large $\tilde T(\tilde t(\varepsilon))$ so that for some $T(\varepsilon) \le \tilde t (\varepsilon) \ll  \tilde T(\tilde t(\varepsilon))$ we have
\[
W_1(\pi^{\tilde T(\tilde t(\varepsilon))}_{\tilde t(\varepsilon)} \otimes \mu^{\tilde T(\tilde t(\varepsilon))}_{\tilde t(\varepsilon)}, \hat \pi_{\tilde t(\varepsilon)} \otimes \hat \mu_{\tilde t(\varepsilon)}) < \frac{\varepsilon}2,
\]
we obtain that 
\begin{align*}
&W_1(\pi^{\tilde T(\tilde t(\varepsilon))}_{\tilde t(\varepsilon)} \otimes \mu^{\tilde T(\tilde t(\varepsilon))}_{\tilde t(\varepsilon)}, \hat \pi_t \otimes \hat \mu_t)
\\&< W_1(\pi^{\tilde T(\tilde t(\varepsilon))}_{\tilde t(\varepsilon)} \otimes \mu^{\tilde T(\tilde t(\varepsilon))}_{\tilde t(\varepsilon)}, \hat \pi_{\tilde t(\varepsilon)} \otimes \hat \mu_{\tilde t(\varepsilon)}) + W_1(\pi \otimes \mu, \hat \pi_{\tilde t(\varepsilon)} \otimes \hat \mu_{\tilde t(\varepsilon)}) + W_1(\pi \otimes \mu, \hat  \pi_t \otimes \hat \mu_t) < \varepsilon
\end{align*}
for all $t\ge T$ by the triangle inequality.
\end{proof}

\subsection{Applications to Learning Theory}\label{sect:4.3}

In this subsection, we discuss the implications of approximating infinite-horizon MFE by finite-horizon MFE in learning theory. In particular, we are interested in demonstrating that we can learn stationary MFE of an underlying unknown system by learning the finite-horizon MFE first. A particular setting of interest where learning stationary MFE is not directly applicable is the model-based learning frameworks. In these frameworks, the learned model is often tailored to the solution concept of interest \cite{pasztor2023efficient,ramponi2024imitation}. These methods typically generate data through simulation and then update the model using the simulated outcomes \cite{huang2024statistical,pasztor2023efficient}. In discrete-time MFGs, such simulations are usually carried out in the finite-horizon setting for reasons of tractability. As a result, most model-based methods for discrete MFGs learn models that approximate the true finite-horizon MFE rather than the infinite-horizon one directly. However, since the learned model is based on its fit to the finite-horizon MFE of the underlying setting, this does not provide us any information whether the learned model can be used in the infinite-horizon setting. Using the results of the previous section, we show that models calibrated by using finite-horizon MFGs in this way can also serve as accurate proxies for infinite-horizon MFE.

So far, we have only discussed the approximation of a non-stationary MFE via a finite-horizon MFE. Hidden in the proof of the result above is the following approximation property that relates a finite MFE to a stationary MFE, which relies on the approximation of non-stationary MFE via finite-horizon MFE. The proof is essentially similar to that of Proposition \ref{prop:5}.

\begin{corollary}\label{cor:2}
Suppose that Assumptions \ref{assump:2} and \ref{assump:3} hold. Let
\((\tilde \pi_t^T,\tilde \mu_t^T)_t \in \mathrm{EMFE}\) be the equilibrium
sequence associated with \(\mathrm{MFG}_T\). Assume that
\(
(\tilde \pi_t^T \otimes \tilde \mu_t^T)_t \to (\hat \pi_t \otimes \hat \mu_t)_t\)
as \(T \to \infty,
\)
where \((\hat \pi_t,\hat \mu_t)_t \in \mathrm{ADOMFE}\). By definition, there exists a stationary MFE $(\pi,\mu)$ such that
\(
\hat \pi_t \otimes \hat \mu_t \to \pi \otimes \mu\) as \(t \to \infty.
\)
Then, for every \(\varepsilon > 0\),
there exists \(t_\varepsilon\) such that for every \(t \ge t_\varepsilon\),
there exists \(T(t,\varepsilon)\) such that, for all \(T \ge T(t,\varepsilon)\),
\[
W_1\!\bigl(\tilde \pi_t^T \otimes \tilde \mu_t^T,\ \pi \otimes \mu\bigr)
< \varepsilon.
\]
\end{corollary}
\begin{proof}
    By Theorem \ref{thrm:q1}, we have $\lim_t \hat \pi_t \otimes  \hat \mu_t=\pi\otimes \mu$ under the metric $W_1$. Let $\varepsilon>0$. Then there exists $t_{\varepsilon}>0$ such that 
    \[
    W_1(\hat \pi_t \otimes \hat \mu_t,\pi \otimes \mu) < \frac{\varepsilon}2
    \]
    for all $t \ge t_{\varepsilon}$. Fix any $t_* \ge t_{\varepsilon}$. Then, we have
    \[
    \lim_{T \to \infty} W_1(\tilde \pi^T_{t_*} \otimes \tilde \mu^T_{t_*} , \hat \pi_{t_*} \otimes \hat \mu_{t_*})=0,
    \]
    meaning that there exists $T_{t_*}$ such that for all $T \ge T_{t_*}$ we must have
    \[
    W_1(\tilde \pi^T_{t_*} \otimes \tilde \mu^T_{t_*},\hat \pi_{t_*}\otimes  \hat \mu_{t_*}) \le \frac{\varepsilon}2.
    \]
    Using triangle inequality, for any $T \ge T_{t_*}$ we obtain 
    \[
    W_1(\tilde \pi^T_{t_*}\otimes \tilde \mu^{T}_{t_*},\pi\otimes \mu) \le W_1(\tilde \pi^T_{t_*}\otimes \tilde \mu^{T}_{t_*}, \hat \pi_{t_*}\otimes \hat \mu_{t_*}) + W_1(\hat \pi_{t_*}\otimes\hat \mu_{t_*},\pi\otimes\mu) \le \varepsilon,
    \]
    as desired.
\end{proof}

\begin{remark}
    Corollary \ref{cor:2} indeed implies the convergence of the state-measures induced by the finite-horizon MFE to that of the stationary MFE. One can also deduce the convergence of the corresponding $Q$-functions in a similar manner. However, in general, we might not have the convergence of the policies themselves. To recover the convergence of the policies, we need additional assumptions, such as the corresponding stationary MFE having a unique corresponding optimal policy; see Lemma \ref{lem:f6}.
\end{remark}

As an application of the corollary above, we consider a hypothetical reinforcement learning scenario in which one learns an approximate model parameter for a finite-horizon MFG whose MFE is close to the original stationary MFE. We show that the stationary MFE induced by a model learned through finite-horizon MFE can then be used as a proxy for the true stationary MFE.

\begin{corollary}\label{cor:6}
    Suppose that Assumptions \ref{assump:2} and \ref{assump:3} hold. For sufficiently large $T$, suppose that for the finite-horizon games $\mathrm {MFG}_{\mathrm T}=(X,A,C,P,\mu_0,T)$ and $\overline {\mathrm{MFG}}_{\mathrm T}=(X,A,\overline C,\overline P,\mu_0,T)$ we have that the flow of joint probability measures generated by the MFE of $(\pi^T_t,\mu^T_t)_t\in\mathrm {MFG}_{\mathrm T}$ and $(\bar \pi^T_t,\bar \mu^T_t)\in\overline {\mathrm{MFG}}_{\mathrm T}$ satisfy $W_1(\pi^T \otimes \mu^T_t,\bar \pi^T_t \otimes \bar \mu^T_t) < \varepsilon$ for sufficiently large $T$. We extend $(\pi^T_t,\mu^T_t)_t$ and $(\bar \pi^T_t,\bar \mu^T_t)_t$ an infinite-flow, and denote these extensions with the same notation. Let $\lim_{T \to \infty} (\pi^T_t \otimes \mu^T_t)_t =(\pi_t,\mu_t)_t,$ and $\lim_{T \to \infty}(\bar \pi^T_t \otimes \bar \mu^T_t) =(\bar \pi^T_t, \bar \mu^T_t)_t$ be such that $(\pi_t,\mu_t)_t,(\bar \pi_t,\bar \mu_t)_t\in \mathrm{ADOMFE}.$ Let $\lim_t \pi_t\otimes \mu_t = \pi\otimes \mu$ and $\lim_{t \to \infty} \bar \pi_t \otimes \bar \mu_t = \bar \pi \otimes \bar \mu$. Then, we have $(\pi,\mu) \in \mathrm {MFG}_{\mathrm s},$ $(\bar \pi ,\bar \mu) \in \overline{\mathrm{MFG}}_{\mathrm s},$ and $W_1( \pi \otimes \mu, \bar \pi \otimes \bar \mu) < 3\varepsilon$.
\end{corollary}

\begin{proof}
    By Corollary \ref{cor:2}, both $(\tilde \pi^T_t)_t$ and $(\overline \pi^T_t)_t$ converge to the optimal policies of MFG$_{\mathrm s}$ and $\overline { \mathrm {MFG}}_{\mathrm s}$, respectively. So, for a given $\varepsilon>0$, there exists a threshold $T(\varepsilon).$
    
    Then, by the triangle inequality, for all $t$, we have that
    \[
    W_1( \pi \otimes \mu, \overline \pi \otimes \overline \mu) < W_1( \pi \otimes \mu, \pi^T_t \otimes \mu^T_t) + W_1(\overline \pi \otimes \overline \mu, \overline \pi^T_t \otimes \overline \mu^T_t) + W_1(\pi^T_t \otimes \mu^T_t,\overline \pi^T_t \otimes \overline \mu^T_t),
    \]
    where $\pmb {\mu}$ and $\pmb {\overline \mu}$ are defined as in Corollary \ref{cor:2}.
    The terms $W_1( \pi_t \otimes \mu_t, \pi^T_t \otimes \mu^T_t)$ and $W_1(\overline \pi_t \otimes \overline \mu_t, \overline \pi^T_t \otimes \overline \mu^T_t)$ are small by Corollary \ref{cor:2} for large $t$ and sufficiently large $T$. By our running assumption, we have that $W_1(\pi^T_t \otimes \mu^T_t,\overline \pi^T_t \otimes \overline \mu^T_t)$ is also small, and thus $W_1( \pi_t \otimes \mu_t, \overline \pi_t \otimes \overline \mu_t)$ must be sufficiently small for sufficiently large $t$.
\end{proof}
{
\begin{remark}
The ADOMFE property required in Corollary \ref{cor:6} follows from the stability of the dynamical system
\[
\begin{bmatrix}
    Q_t \\
    \mu_t
\end{bmatrix}
=
\begin{bmatrix}
    H_{1,t}(Q_{t+1},\mu_t) \\
    H_{2,t}(Q_{t-1},\mu_{t-1})
\end{bmatrix}.
\]
In the regularized case, Theorem \ref{prop:2} provides a sufficient condition for the global asymptotic stability of this dynamical system via contraction.
\end{remark}
}
\section{Conclusion}
In this work, we have established improved contraction rates for finite-horizon MFGs and shown that finite-horizon MFGs can still be contractive even when the contraction conditions found in the literature for the infinite-horizon setting can fail. We have demonstrated that accumulation points of finite-horizon MFE are non-stationary MFE under mild conditions. Furthermore, we have studied the relationship between stationary MFE and finite-horizon MFE and provided conditions under which we can approximate stationary MFE with finite-horizon MFE. Using the contraction result established for finite-horizon MFE, we have obtained finite-time error bounds between finite-horizon and infinite-horizon MFE. Through these error bounds we have also obtained a new uniqueness criterion for infinite-horizon non-stationary MFE. As an application, we have shown that when two MFGs have finite-horizon MFE that are close under the $W_1$ metric, the corresponding stationary MFE are also close under $W_1$.
\printbibliography
\appendix
\section{Proofs of Subsection \ref{sect:3.1}}\label{app:a}
In this appendix, we will provide complete proofs for the technical results presented in Subsection \ref{sect:3.1}. Almost all of the proofs here are slight modifications of those in \cite{anahtarci2023q}. The differences are due to the fact that we are working with finite-horizon MFGs. Nevertheless, these technical results are fundamental for the purpose of Subsection \ref{sect:3.1}. For the sake of completeness, we will provide complete proofs for these results.
\subsection{Proof of Lemma \ref{lem:f0}}\label{sect:lem:f0}
\begin{proof}[Proof of Lemma \ref{lem:f0}]
    For $t=T-1$, we have
    \begin{align*}
    \sup_{u \in \mathcal P(A)} |Q^{\pmb \mu^T}_{T-1}(x,u) - Q^{\pmb \mu^T}_{T-1}(\tilde x,u)| &= \sup_{u \in \mathcal P(A)}|C(x,u,\mu_{T-1}^T)-C(\tilde x, u ,\mu_{T-1}^T)|
    \\&\leq L_1 1_{\{x \not = \tilde x\}}\leq \frac{L_1}{1-\frac{\beta K_1}2} 1_{\{x \not = \tilde x\}}.
    \end{align*}
    We suppose that up until time $t+1$, the $Q$-functions $Q^{\pmb \mu}_T$ satisfy the Lipschitz property
    \[
    \sup_{x,\tilde x \in X}\left|Q^{\pmb \mu^T}_{t+1}(x,u) -Q^{\pmb \mu^T}_{t+1}(\tilde x,u)\right| \le \frac{L_1}{1-\frac{\beta K_1}2}1_{\{x \not = \tilde x\}}.
    \]
    We recall the concentration inequality 
    \begin{align*}
     &\left|\sum_{y \in X} \min_{b \in A}Q^{\pmb\mu^T}_{t+1}(y,b)\left( P(y|x,u,\mu_t^T)-P(y|\tilde x,u,\mu^T_t) \right) \right|
     \\& \le \frac{\sup_{x \in X}\min_{b \in A}Q^{\pmb\mu^T}_{t+1}(y,b)-\min_{y \in X}\min_{b \in A}Q^{\pmb\mu^T}_{t+1}(y,b)}{2} \| P(\cdot|x,u,\mu_t^T) - P(\cdot|\tilde x,u,\mu_t^T)\|_1
     \\& \le \frac{L_1}{1-\frac{\beta K_1}2} \frac 12  \| P(\cdot|x,u,\mu_t^T) - P(\cdot|\tilde x,u,\mu_t^T)\|_1,
    \end{align*}
    where the last line follows from the induction hypothesis above.
    
    Going backwards in time, using Lemma \ref{lem:a} and the assumption above, we obtain
    \begin{align*}
    &\sup_{u \in \mathcal P(A)} | Q^{\pmb\mu^T}_t(x,u) - Q^{\pmb \mu^T}_t(\tilde x,u)| 
    \\&\leq \sup_{u \in \mathcal P(A)}\left(|C(x,u,\mu_t^T)-C(\tilde x, u ,\mu_t^T)| +
    \beta \left|\sum_{y \in X} \min_{b \in A}Q^{\pmb\mu^T}_{t+1}(y,b)\left( P(y|x,u,\mu_t^T)-P(y|\tilde x,u,\mu_t^T) \right) \right|\right)
    \\& \le L_1 1_{\{x \not = \tilde x\}} + \sup_{u \in \mathcal P(A)}\beta \left|\sum_{y \in X} \min_{b \in A}Q^{\pmb\mu^T}_{t+1}(y,b)\left( P(y|x,u,\mu_t^T)-P(y|\tilde x,u,\mu_t^T) \right) \right|
    \\&\leq \left [L_1 + \frac{\beta}2 \frac{L_1K_1(1-(\frac{\beta K_1}2)^{T-t})}{\left(1-\frac{\beta K_1}2\right)}\right] 1_{\{x \not = \tilde x\}} \leq \left [L_1 + \frac{\beta}2 \frac{L_1K_1}{(1-\frac{\beta K_1}2)}\right] 1_{\{x \not = \tilde x\}}.
    \end{align*}
    
    Lipschitz continuity of the minimizers diretly follows from \cite[Lemma 3]{anahtarci2023q}.
\end{proof}
\subsection{Proof of Lemma \ref{lem:f1}}\label{sect:lem:f1}
\begin{proof}[Proof of Lemma \ref{lem:f1}]
    The proof closely follows that of \cite[Proposition 2]{anahtarci2023q} and provide here an outline of the proof for completeness, not including some of the details in the calculations. Fix any $\pmb \mu, \pmb {\tilde{\mu}} \in \{\mu_0\}\times P(X)^{T-1}$. For a given $Q$-function, $Q$, let $f(Q)(x) = \arg\min_{u \in \mathcal P(A)} Q(x,u)$. Using Lemma~\ref{lem:f0}, we obtain the following by applying the triangle inequality
\begin{align*}
    &  \| H_{2,t}(Q^{\pmb {\mu}^T}_{t-1},\mu_{t-1}^T) - H_{2,t}(Q^{\pmb {\tilde \mu}^T}_{t-1},\tilde \mu_{t-1}^T) \|_{\mathrm{TV}} 
    \\&= \frac 12\sum_y \bigg|\sum_x P(y|x, f(Q^{\pmb {\mu}^T}_{t-1})(x),\mu_{t-1}^T)\mu_{t-1}^T(x) - \sum_x P(y|x,f(Q^{\pmb {\tilde \mu}^T}_{t-1})(x), \tilde{\mu}_{t-1}^T) \tilde{\mu}_{t-1}^T(x)\bigg| \\
    &\leq \frac 12\sum_y \bigg|\sum_x P(y|x,f(Q^{\pmb {\mu}^T}_{t-1})(x),\mu_{t-1}^T)\mu_{t-1}^T(x) - \sum_x P(y|x, f(Q^{\pmb {\tilde \mu}^T}_{t-1})(x), \tilde{\mu}_{t-1}^T)\mu_{t-1}^T(x)\bigg| \\
    &\quad + \frac 12\sum_y \bigg|\sum_x P(y|x, f(Q^{\pmb {\tilde \mu}^T}_{t-1})(x), \tilde{\mu}_{t-1}^T)\mu_{t-1}^T(x) 
    - \sum_x P(y|x, f(Q^{\pmb {\tilde \mu}^T}_{t-1})(x), \tilde{\mu}_{t-1}^T) \tilde{\mu}_{t-1}^T(x)\bigg|
    \\& = (a).
\end{align*}
Note that 
\begin{align*}
    &\sum_y \bigg|\sum_x P(y|x,f(Q^{\pmb {\mu}^T}_{t-1})(x),\mu^T_{t-1})\mu^T_{t-1}(x) 
    - \sum_x P(y|x, f(Q^{\pmb {\tilde \mu}^T}_{t-1})(x), \tilde{\mu}_{t-1}^T)\mu_{t-1}^T(x)\bigg|
    \\& \le \sum_x \left\| P(\cdot|x,f(Q^{\pmb { \mu}^T}_{t-1})(x),\mu_{t-1}^T) - P(\cdot|x,f(Q^{\pmb {\tilde \mu}^T}_{t-1})(x), \tilde{\mu}_{t-1}^T) \right\|_1 \mu_{t-1}^T(x),
\end{align*}
and using \cite[Lemma A.1]{kontorovich2008concentration} we obtain
\begin{align*}
    &\sum_y \bigg|\sum_x P(y|x, f(Q^{\pmb {\tilde \mu}^T}_{t-1})(x), \tilde{\mu}_{t-1}^T)\mu_{t-1}^T(x) 
    - \sum_x P(y|x, f(Q^{\pmb {\tilde \mu}^T}_{t-1})(x), \tilde{\mu}^T_{t-1}) \tilde{\mu}^T_{t-1}(x)\bigg|
    \\& \le \sum_y\left( \frac{\sup_{x \in X}P(y|x,f(Q^{\pmb {\tilde \mu}^T}_{t-1})(x), \tilde{\mu}^T_{t-1})-\inf_{x \in X}P(y|x,f(Q^{\pmb {\tilde \mu}^T}_{t-1})(x), \tilde{\mu}^T_{t-1})}{2}\right)\|\mu_{t-1}^T-\tilde \mu_{t-1}^T\|_1
    \\& \le \frac 12\sup_{x,x'\in X}\| P(\cdot|x,f(Q^{\pmb {\tilde \mu}^T}_{t-1})(x), \tilde{\mu}^T_{t-1})-P(\cdot|x',f(Q^{\pmb {\tilde \mu}^T}_{t-1})(x'), \tilde{\mu}^T_{t-1})\|_1 \|\mu_{t-1}^T-\tilde \mu^T_{t-1}\|_1.
\end{align*}
By Assumption~\ref{assump:1}, we must have
\begin{align*}
&\| P(\cdot|x,f(Q^{\pmb {\tilde \mu}^T}_{t-1})(x), \tilde{\mu}^T_{t-1})-P(\cdot|x',f(Q^{\pmb {\tilde \mu}^T}_{t-1})(x'), \tilde{\mu}^T_{t-1})\|_1 
\\& \le K_1 \left( 1_{\{x\not = x'\}}+\| f(Q^{\pmb {\tilde \mu}^T}_{t-1})(x)-f(Q^{\pmb {\tilde \mu}^T}_{t-1})(x')\|_1\right).
\end{align*}
Furthermore, by \cite[Lemma 3]{anahtarci2023q}, we obtain
\begin{align*}
    \| f(Q^{\pmb {\tilde \mu}^T}_{t-1})(x)-f(Q^{\pmb {\tilde \mu}^T}_{t-1})(x')\|_1 \le \frac{\bar L}{\rho} 1_{\{x \not = x'\}}.
\end{align*}
Thus, as desired, it holds that
\begin{align*}
    (a) &\leq \frac 12\sum_x \left\| P(\cdot|x,f(Q^{\pmb { \mu}^T}_{t-1})(x),\mu_{t-1}) - P(\cdot|x,f(Q^{\pmb {\tilde \mu}^T}_{t-1})(x), \tilde{\mu}_{t-1}) \right\|_1 \mu_{t-1}^T(x) 
    \\&\qquad + \frac{K_1}4\left( 1 + \frac{\bar L}{\rho}\right) \|\mu_{t-1}^T-\tilde{\mu}_{t-1}^T\|_1 \\
    &\leq \frac 12K_1\left( \sup_x \left\|f(Q^{\pmb { \mu}^T}_{t-1})(x) - f(Q^{\pmb {\tilde \mu}^T}_{t-1})(x) \right\|_1 + \|\mu_{t-1}^T - \tilde{\mu}_{t-1}^T\|_1 \right)
    +\frac{K_1}4\left( 1 + \frac{\bar L}{\rho}\right) \|\mu_{t-1}^T-\tilde{\mu}^T_{t-1}\|_1 \\
    &\leq \frac{K_1}{2\rho}\| Q^{\pmb \mu^T}_{t-1} - Q^{\pmb {\tilde \mu}^T}_{t-1}\|_{\infty}+\bar K \|\mu_{t-1}^T-\tilde{\mu}^T_{t-1}\|_{\mathrm{TV}}.
\end{align*}

For the term $\| H_{2,1}(Q^{\pmb { \mu}^T}_0,\mu^T_0) - H_{2,1}(Q^{\pmb {\tilde \mu}^T}_0,\mu_0^T) \|_{\mathrm{TV}},$ since $\mu_0$ is fixed, the same proof as above yields the desired result.
\end{proof}
\subsection{Proof of Lemma \ref{lem:f2}}\label{sect:lem:f2}
\begin{proof}[Proof of Lemma \ref{lem:f2}]
    Since the initial state-measures of the flows $\pmb \mu^T$ and $\pmb {\tilde \mu}^T$ are the same, for $t=0$, we obtain the following:
    \[
        \| Q^{\pmb \mu^T}_0 - Q^{\pmb {\tilde \mu^T}}_0 \|_{\infty} \leq \beta \left|\sum_{y \in X} (Q^{\pmb \mu^T}_1 - Q^{\pmb {\tilde \mu}^T}_1)P(y|x,a,\mu_0) \right| \leq \beta \| Q^{\pmb {\mu}^T}_1 - Q^{\pmb {\tilde \mu}^T}_1 \|_{\infty}.
    \]
    For the last term, we similarly obtain
    \[
    \| Q^{\pmb {\mu}^T}_{T-1} - Q^{\pmb {\tilde \mu}^T}_{T-1} \|_{\infty} \leq \sup_{(x,a) \in X \times A}|C(x,a,\mu_{T-1}^T)-C(x,a,\tilde \mu_{T-1}^T)| \leq L_1 \| \mu^T_{T-1} - \tilde \mu^T_{T-1}\|_{\mathrm{TV}}.
    \]
    For the intermediate terms, using similar observations as above, we have
    \begin{align*}
       & \| Q^{\pmb {\mu}^T}_t - Q^{\pmb {\tilde \mu}^T}_t \|_{\infty} 
       \\&\leq \sup_{(x,u)\in X \times \mathcal P(A)}| C(x,u,\mu_t^T) - C(x,u,\tilde \mu_t^T)| 
        \\&  + \beta \sup_{(x,u)\in X \times \mathcal P(A)}\left|\sum_{y \in X}  \min_{v \in \mathcal P(A)}Q^{\pmb \mu^T}_{t+1}(y,v)P(y|x,u,\mu_t^T) - \sum_{y \in X}  \min_{b \in \mathcal P(A)}Q^{\pmb {\tilde \mu}^T}_{t+1}(y,v)P(y|x,u,\tilde \mu_t^T) \right|.
    \end{align*}
    As a consequence of Lemma \ref{lem:a}, we obtain
    \[
    | C(x,u,\mu_t^T) - C(x,u,\tilde \mu_t^T)| \le 2L_1 \| \mu_t^T - \tilde \mu^T_{t}\|_{\mathrm{TV}}.
    \]
    For the remainder term, by the triangle inequality, we obtain
    \begin{align*}
     &\left|\sum_{y \in X}  \min_{v \in \mathcal P(A)}Q^{\pmb \mu^T}_{t+1}(y,v)P(y|x,u,\mu_t^T) - \sum_{y \in X}  \min_{b \in \mathcal P(A)}Q^{\pmb {\tilde \mu}^T}_{t+1}(y,v)P(y|x,u,\tilde \mu_t^T) \right|
     \\&\le \underbrace{\left|\sum_{y \in X}  \min_{v \in \mathcal P(A)}Q^{\pmb \mu^T}_{t+1}(y,v)P(y|x,u,\mu_t^T) - \sum_{y \in X}  \min_{v \in \mathcal P(A)}Q^{\pmb \mu^T}_{t+1}(y,v)P(y|x,u,\tilde \mu_t^T) \right|}_{(a)}
     \\& \qquad +\underbrace{\left|\sum_{y \in X}  \min_{v \in \mathcal P(A)}Q^{\pmb \mu^T}_{t+1}(y,v)P(y|x,u,\tilde \mu_t^T) - \sum_{y \in X}  \min_{v \in \mathcal P(A)}Q^{\pmb {\tilde \mu}^T}_{t+1}(y,v)P(y|x,u,\tilde \mu_t^T) \right|}_{(b)}.
    \end{align*}
    
    To control the term $(a)$, we recall the following concentration inequality for probability measures:
    \[
    \left|\sum_{x \in X} F(x)\mu(x)  - \sum_{x \in X}F(x)\tilde \mu(x)\right| \le \frac{\sup_{x \in X}F(x) - \inf_{x \in X}F(x)}2\| \mu-\tilde \mu\|_1,
    \]
    see \cite{kontorovich2008concentration}.
    Thus,
    \[
    (a) \le 2\left(\sup_{x \in X}\min_{v \in \mathcal P(A)}Q^{\pmb \mu^T}_{t+1}(x,v)-\inf_{x \in X}\min_{v \in \mathcal P(A)}Q^{\pmb \mu^T}_{t+1}(x,v)\right) K_1 \|\mu_t^T-\tilde  \mu_t^T\|_{\mathrm{TV}}
    \]
    by Lemma \ref{lem:a}, where the factor $2$ is due to the total variation norm. Using Lemma \ref{lem:f0}, we obtain
    \begin{align*}
    &\sup_{x \in X}\min_{v \in \mathcal P(A)}Q^{\pmb \mu^T}_{t+1}(x,v)-\inf_{x \in X}\min_{v \in \mathcal P(A)}Q^{\pmb \mu^T}_{t+1}(x,v) \le  \sup_{x,\tilde x \in X}\left| \min_{v \in \mathcal P(A)}Q^{\pmb \mu^T}_{t+1}(x,v)-\min_{v \in \mathcal P(A)}Q^{\pmb \mu^T}_{t+1}(\tilde x,v) \right|
    \\& \le \sup_{u \in \mathcal P(A)}\sup_{x,\tilde x \in X}\left| Q^{\pmb \mu^T}_{t+1}(x,v)-Q^{\pmb \mu^T}_{t+1}(\tilde x,v)\right|\le \frac{L_1}{1-\frac{\beta K_1}2}1_{\{x \not = \tilde x\}};
    \end{align*}
    thus, 
    \[
    (a) \le 2 \bar L K_1\| \mu_t^T-\tilde \mu_t^T\|_{\mathrm{TV}}.
    \]

    It remains to upper bound $(b)$. Since the stochastic kernel $P$ is evaluated under $(x,u,\tilde \mu_t),$ we obtain
    \begin{align*}
        (b) &\le \sup_{x \in X}\left|\min_{v \in \mathcal P(A)}Q^{\pmb \mu^T}_{t+1}(x,v)-  \min_{v \in \mathcal P(A)}Q^{\pmb {\tilde \mu}^T}_{t+1}(x,v)\right|
        \\& \le \sup_{x \in X}\sup_{u \in \mathcal P(A)}\left|Q^{\pmb \mu^T}_{t+1}(x,v)- Q^{\pmb {\tilde \mu}^T}_{t+1}(x,v)\right|\le \|Q^{\pmb \mu^T}_{t+1} -Q^{\pmb {\tilde \mu}^T}_{t+1}\|_{\infty}.
    \end{align*}

    Gathering all the bounds above, we obtain the desired inequality as follows:
    \begin{align*}
        \| Q^{\pmb {\mu}^T}_t - Q^{\pmb {\tilde \mu}^T}_t \|_{\infty} &\le \left(2L_1 -2 \frac{\beta L_1K_1}{1-\frac{\beta K_1}2} \right)\|\mu_t^T-\tilde \mu_t^T\|_{\mathrm{TV}}+\beta \| Q^{\pmb {\mu}^T}_{t+1} - Q^{\pmb {\tilde \mu}^T}_{t+1}\|
        \\& \le 2\bar L\|\mu_t^T -\tilde \mu_t^T\|_{\mathrm{TV}}+\beta \| Q^{\pmb {\mu}^T}_{t+1} - Q^{\pmb {\tilde \mu}^T}_{t+1}\|.
    \end{align*}
\end{proof}

\section{Proofs of Subsection \ref{sect:b}}\label{app:b}
In this appendix, we will provide proofs for the technical results stated in Subsection \ref{sect:b}.
\subsection{Proof of Lemma \ref{lem:g1}}\label{sect:lem:g1}
\begin{proof}[Proof of Lemma \ref{lem:g1}]
    If $h=(h_1,\cdots,h_T)$, $A_Th= \rho(A_T)h$ implies that
    \[
    \left(\bar K + \frac{K_1}{\rho}\bar L\right)h_1 + \sum_{j=2}^{T-1}\frac{K_1}{\rho}\bar L \beta^{j-1} h_j + \frac{K_1}{\rho}\bar L \beta^{T-1}h_T = \rho(A_T)h_2
    \]
   and thus we have 
\[
    \rho(A_T)h_1 = \sum_{j=1}^{T-1}\frac{K_1}{\rho}\bar L \beta^j h_j + \frac{K_1}{\rho}\bar L \beta^{T}h_T= \beta \left( \rho(A_T)h_2 - \bar K h_1 \right).
\]
    Similarly, for $j=2 \cdots, T-1$, it also holds that
    \(
    \rho(A_T)h_j = \hat K h_{j-1} - \bar K \beta h_j + \rho(A_T)\beta h_{j+1}.
    \)
    The last row of $T(\rho(A_T))$ follows directly from that of $A_T$.
\end{proof}
\subsection{Proof of Lemma \ref{lem:g2}}\label{sect:lem:g2}
\begin{proof}[Proof of Lemma~\ref{lem:g2}]
    Suppose not. Let $k$ be the largest positive eigenvalue of $T(\rho(A_T))$, which exists as $\rho(A_T)>0$ is an eigenvalue of $T(\rho(A_T))$. Then we have 
    \[
    \left(T(\rho(A_T))+(\bar K \beta + \varepsilon) I\right)g = (k+(\bar K \beta + \varepsilon) )g
    \]
    for a unique positive vector $g$ (up to constant multiplicity) by the Perron-Frobenius theorem. However, we also have \[
    \left(T(\rho(A_T))+(\bar K \beta + \varepsilon) I\right)h = (\rho(A_T)+\bar K\beta + \varepsilon)h\] for some positive vector $h$ by Lemma \ref{lem:g1}. However, note that eigenvectors of $T(\rho(A_T))+(\bar K \beta + \varepsilon) I$ and $A_T$ that corresponds to $\rho(A_T)$ are the same. Thus, we must have $\rho(A_T) = k$ as the span of positive eigenvectors must be one-dimensional by the Perron-Frobenius theorem.
\end{proof}
\subsection{Proof of Lemma \ref{lem:222}}\label{sect:lem:222}
\begin{proof}[Proof of Lemma~\ref{lem:222}]
    We recall a simple well-known fact that the eigenvalues of $T(\rho(A_T))$ are real \cite[3.1.P22]{horn2012matrix} as the terms on the sub-diagonal and super-diagonal are positive, which can be proven directly by using a diagonal transformation. By \cite{yueh2005eigenvalues}, eigenvalues of $ T(\rho(A_T))$ are of the form
\[
- \bar K \beta + 2 \sqrt{ \rho(A_T) \beta \hat K} \cos \theta,
\]
where $\theta \in \mathbb C$ (so $\cos \theta$ can take any value as it an entire complex function). For the eigenvalue $\lambda = \rho(A_T)$, for some $z \in \mathbb C$, using the identity $z= e^{i \theta}$, we can write
\(
z + \bar z = 2 \cos \theta,
\)
where $\bar z$ is the complex conjugate of $z$, because $\cos\theta$ must be a real number in this case \cite[Eq. 4]{yueh2005eigenvalues}. We claim that, the same resul holds also for any given eigenvalue $\lambda$ of $T(\rho(A_T)).$ In particular, we want to show that any $\lambda$ can be written in the form
\begin{equation}\label{eq:q24}
\lambda =- \bar K \beta + \sqrt{ \rho(A_T) \beta \hat K} \left( z + z^{-1} \right)
\end{equation}
for some $z \in \mathbb C$ because for every $\tilde z \in \mathbb C$, there exists $z$ such that $\tilde z = z + z^{-1}$. We aim to identify the polynomial which gives $z + z^{-1}$ as roots in order to identify the eigenvalues of $T(\rho(A_T))$.
Fix $z$ and $\lambda$ as above. Let $u=(u_1,u_2,\cdots,u_{T-1})$. Then, if 
\begin{equation}\label{eq:ssssss5}
T(\rho(A_T))u = \lambda u,
\end{equation}
for all $1<k \leq T-1$, we want to find $(u_k)_{k=1}^{T-1}$ such that
\begin{equation}\label{eqref:q23}
\hat K u_{k-1} - \bar K \beta u_k + \rho(A_T)\beta u_{k+1} = \lambda u_k,
\end{equation}
which can be deduced from the recursive relations that arise from the relation \eqref{eq:ssssss5}.

To solve the difference equations above, we propose two solutions for \eqref{eqref:q23} as $u_k= y^k_1$ and $u_k = y^k_2$, where $y_1 = \sqrt{\frac{ \hat K }{\rho(A_T)\beta}}z$ and $y_2 = \sqrt{\frac{ \hat K }{\rho(A_T)\beta}} z^{-1}$. For $u_k = y^k_1$, we have
\[
y^k_1\left( \hat Ky_1^{-1} - \bar K \beta  + \rho(A_T)\beta y_1  \right) = \lambda y^k_1.
\]
Since 
\[
\hat Ky^{-1}_1 - \bar K \beta  + \rho(A_T)\beta y_1= -\bar K \beta + \sqrt{\hat K \rho(A_T)\beta} z^{-1}+ \sqrt{\rho(A_T)\beta \hat K} z,
\]
$y_1$ satisfies \eqref{eqref:q23}. Similarly, for $u_k = y^k_2$, we have
\[
y^k_2 \left( \sqrt{\hat K \rho(A_T)\beta} z -\bar K \beta +\sqrt{ \rho(A_T)\beta \hat K} z^{-1} \right) = \lambda u_k.
\]
Thus, $u_k= y^k_1,y^k_2$ are candidates for solutions for \eqref{eqref:q23} for $1< k < T$. It follows that $u_k = y^k_1+cy^k_2$ also satisfies \eqref{eqref:q23} for all $1<k \leq T-1$, where  $c\in \mathbb R$. We want to find $c$ such that $u_k$ satisfies \eqref{eqref:q23} also for $k= T-1$ and $k=1$.

Let $u_0 =0$. Then, for the recursion relation at $k=1$, we have
\begin{align*}
-\bar K \beta u_1 + \rho(A_T)\beta u_2 
&= -\bar K \beta  \sqrt{\frac{ \hat K }{\rho(A_T)\beta}}(z +c z^{-1}) + \hat K(z^2+cz^{-2})
\\&=\left( - \bar K \beta + \sqrt{\rho(A_T) \beta \hat K} (z+ z^{-1}) \right) \left( \sqrt{\frac{ \hat K }{\rho(A_T)\beta}}(z +c z^{-1})\right)
\end{align*}
by \eqref{eq:q24}. Matching the terms, we get
\(
\hat K \left(z^2+c z^{-2}\right) = \hat K \left(z^2 + 1 + c + c z^{-2}\right);
\)
and hence, $c=-1$ must hold.

For the case $k=T-1$, expanding the relation \eqref{eqref:q23} we obtain
\begin{align*}
&\hat K u_{T-2} + (-\bar K \beta +r)u_{T-1} 
\\&=\hat K  \left(\sqrt{\frac{ \hat K }{\rho(A_T)\beta}}\right)^{T-2}\left(z^{T-1}-\frac 1{z^{T-2}}\right) +(- \bar K \beta+r) \left(\sqrt{\frac{ \hat K }{\rho(A_T)\beta}}\right)^{T-1}\left(z^{T-1}-\frac 1{z^{T-1}}\right)
\\&= \lambda u_{T} = ( -\bar K \beta + \sqrt{ \rho(A_T) \beta \hat K} \left(z + z^{-1}\right) ) \left(\sqrt{\frac{ \hat K }{\rho(A_T)\beta}}\right)^{T-1}\left(z^{T-1}-\frac 1{z^{T-1}}\right)
\\& = -\bar K \beta \left(\sqrt{\frac{ \hat K }{\rho(A_T)\beta}}\right)^{T-1}\left(z^{T-1}-\frac 1{z^{T-1}}\right) + \hat K\left(\sqrt{\frac{ \hat K }{\rho(A_T)\beta}}\right)^{T-2}
\left(z^{T}+z^{T-2}-\frac 1{z^{T}}-\frac 1{z^{T-2}}\right).
\end{align*}
The relationship we have established above can be simplified to the following:
\begin{equation}\label{eq:g25}
z^{T}- z^{-T} = \frac{ r}{\sqrt{\rho(A_T)\beta \hat K}}\left(z^{T-1}-z^{-(T-1)}\right).
\end{equation}
Writing $z = u e^{i \theta '}$, the above equation can be rewritten as a root finding problem for a degree $T$ trigonometric polynomial, which will have $2T$ roots.

We can rewrite \eqref{eq:g25} as
\begin{equation}\label{eq:g26}
p(z) := z^{2T} - \frac{ r }{\sqrt{ \rho(A_T) \beta \hat K}} z^{2T-1} + \frac{ r }{\sqrt{ \rho(A_T) \beta \hat K}} z -1=0.
\end{equation}
Note that $p(1)=p(-1)=0$. Let $p(z) = (z^2-1)q(z)$. Then, since we have
    \[
z^{2T} -1 = (z^2-1)(z^{2T-2}+z^{2T-4}+\cdots+1)
    \]
    and
    \[
z^{2T-1} - z = z(z^2-1)(z^{2T-4}+z^{2T-6}+\cdots+1),
    \]
    it follows that
    \[
q(z) = z^{2(T-1)} - \frac{ r }{\sqrt{ \rho(A_T) \beta \hat K}}z^{2T-3}+z^{2T-4}- \frac{ r }{\sqrt{ \rho(A_T) \beta \hat K}}z^{2T-5}+\cdots -  \frac{ r }{\sqrt{ \rho(A_T) \beta \hat K}} z +1.
    \]
    Therefore, $q(z) = z^{2T-2}q\left( z^{-1}\right),$ i.e., $q$ is a degree $(2T-2)$-palindromic polynomial. Since $q$ is a palindromic polynomial of even degree, for $\omega = z + z^{-1}$, we can write $q(z) = z^{T-1}Q(\omega)$, where $Q$ is a polynomial of degree $T-1$ \cite[Theorem 2.1]{conrad2016roots}. Thus, for every pair $z+z^{-1}$ that is a root of $Q$, we can backtrack the relations above to show that $\lambda = -\bar K \beta + \sqrt{\rho(A_T)\beta \hat K}\left(z + z^{-1}\right)$ is an eigenvalue, and this completes the proof.
\end{proof}
\subsection{Proof of Lemma \ref{lem:h24}}\label{sect:lem:h24}
\begin{proof}[Proof of Lemma \ref{lem:h24}]
If $\frac{r}{\sqrt{ \rho(A_T)\beta \hat K}}= 1,$ then we have $p(z) = (z^{2T-1}+1)(z-1)$, and thus all the roots of $p$ are on the (complex) unit circle.

We will next consider the case $j := \frac{r}{\sqrt{ \rho(A_T)\beta \hat K}} < 1$. Any $z \in \mathbb C$ that satisfies $p(z)=0$ also satisfies $h(z):=\frac{z^{2T-1}(z-j)}{(1-jz)}=1$. Note that as $j<1$, $h$ is a holomorphic function in the region $|z|<1$. Now, for $|z|<1$, we have that
\(
|h(z)| < \left | \frac{z-j}{1-jz}\right|.
\)
Since $j<1$, the map $z \mapsto \frac{z-j}{1-jz}$ is a Möbius transformation that maps the unit circle into itself; thus, $|h(z)|\leq 1$ for all $|z|\leq 1$. This stems from the fact that $\left |\frac{z-j}{1-jz}\right| <1$ for $|z|<1$. As a consequence of the maximum modulus principle \cite[Thrm 10.24]{rudin1987real}, $h(z) = 1$ is only feasible on the unit circle $|z|=1$ for $|z| \leq 1$. 

Now, observe that
\[
\left | \frac{z-j}{1-jz}\right|^2-1 = \frac{|z-j|^2-|1-jz|^2}{|1-jz|^2},
\]
and that
\[
|z-j|^2-|1-jz|^2 = (1-j^2)(|z|^2-1).
\]
Since $j<1$, this implies that for any $|z|>1$, we have $\left | \frac{z-j}{1-jz}\right|>1$. In particular, $|h(z)|>1$ for $|z|>1$. Therefore, by the minimum modulus principle, it follows that for $|z| \ge 1$, $h(z)=1$ is only possible for $|z|=1$.
\end{proof}
\section{Proofs of Section \ref{sect:approx}}\label{app:c}
In this appendix, we will provide the proofs of technical results stated in Section \ref{sect:approx}. These results can be easily translated to the regularized setting; however, since the setting in Section \ref{sect:approx} is more general, we will stick to that case.
\subsection{Proof of Lemma \ref{lem:f6}}\label{sect:lem:f6}
Let 
\[
\mathcal S := \left\{ Q:X \times A: |\min_{a \in A}Q(x,a)-\min_{a \in A}Q(\tilde x,a)| \le \frac{\tilde L_1}{1-\beta \tilde K_1}d_X(x,\tilde x) \, , \|Q\|_{\infty} \le \frac M{1-\beta}\right\}.
\]
Let
\(
H_1(Q,\mu)(x,a) := c(x,a,\mu) + \beta \int_X \min_{b \in A} Q(y,b)p(dy|x,a,\mu).
\)
We claim that $H_1(\cdot,\mu)$ maps $\mathcal S$ into itself. Indeed,
\[
|H_1(Q,\mu)(x,a)| \le \| c(\cdot,\cdot,\cdot)\|_{\infty} + \beta \left| \int_X \min_{b \in A} Q(y,b)p(dy|x,a,\mu)\right| \le M + \beta \frac{M}{1-\beta}=\frac{M}{1-\beta},
\]
so the same uniform bound holds.

To obtain the Lipschitz bound, using the Lipschitz bound of $\min_{b \in A} Q(\cdot,b)$, we obtain
\begin{align*}
    &|\min_{b \in A}H_1(Q,\mu)(x,b)-\min_{b \in A}H_1(Q,\mu)(\tilde x,b)| \le \sup_{b \in A}| H_1(Q,\mu)(x,b)-H_1(Q,\mu)(\tilde x,b)|
    \\& \le \sup_{b \in A}| c(x,b,\mu)-c(\tilde x,b,\mu)|+\beta \sup_{b \in A}\left| \int_X\min_{b \in A} Q(y,b)p(dy|x,a,\mu)-\int_X\min_{b \in A} Q(y,b)p(dy|\tilde x,a,\mu) \right|
    \\& \le \tilde L_1 + \beta \frac{\tilde L_1}{1-\beta \tilde K_1} \le \frac{\tilde L_1}{1-\beta \tilde K_1}.
\end{align*}
Thus, $H_1(\cdot,\mu)$ maps $\mathcal S$ into itself for any $\mu \in \mathcal P(X)$.

Now, suppose that there exists a family of $Q$-functions $(Q^m_n)_n\subset \mathcal S$ and $(\mu^m_n)_n\subset \mathcal P(X)$ such that
\(
H_1(Q^m_{n+1},\mu^m_n)=Q^m_n 
\)
for all $n \ge 0$ and $m\ge0$. Suppose $(\mu^m_n)_n \in \mathcal P(X)^{\infty}$ converges to $(\mu_n)_n$ weakly as $m \to \infty$. Suppose there exists a unique family of $Q$-functions $(Q_n)_n$ such that 
\(
H_1(Q_{n+1},\mu_n)=Q_n,
\)
which we will prove to be the case later on.
We want to show that $\lim_{m \to \infty} \| Q^m_n -Q_n\|_{\infty}=0$. First, we take any subfamily $\{(Q^{m_k}_n)_n: k \in \mathbb N\}$ of the family $\{ (Q^m_n)_n: m \in \mathbb N\}$. Now take a further subfamily of the family $\{(Q^{m_k}_n)_n: k \in \mathbb N\}$, say $\{( Q^{m_{k_l}}_n)_n: l \in \mathbb N\}$ such that $\min_{b \in A} Q^{m_{k_l}}_n(\cdot,b)$ are convergent under the uniform norm, which is possible as the set
\[
\mathcal S_X:=\left\{ u:X \to \mathbb R: |u(x)-u(\tilde x)| \le \frac{\tilde L_1}{1-\beta \tilde K_1}d_X(x,\tilde x) \, , \|u\|_{\infty} \le \frac M{1-\beta}\right\}
\]
is compact by the Arzel\`a-Ascoli theorem. 

By Assumption \ref{assump:2}, we must have
\(
\lim_{l \to \infty} \|c(\cdot,\cdot,\mu^{m_{k_l}}_n)-c(\cdot,\cdot,\mu_n)\|_{\infty} = 0.
\)
Furthermore, note that
\begin{align*}
    &\left|\int_X \min_{b \in A}Q^{m_{k_l}}_n(y,b)p(dy|x,a,\mu^{m_{k_l}}_n) -\int_X \min_{b \in A} Q_n(y,b) p(dy|x,a,\mu_n) \right|
    \\& \le \underbrace{\left|\int_X \min_{b \in A}Q^{m_{k_l}}_n(y,b)p(dy|x,a,\mu^{m_{k_l}}_n) -\int_X \min_{b \in A} Q_n(y,b) p(dy|x,a,\mu^{m_{k_l}}_n) \right|}_{=(a_l)(x,a)}
    \\& \qquad \qquad+\underbrace{\left| \int_X \min_{b \in A} Q_n(y,b) p(dy|x,a,\mu^{m_{k_l}}_n) -\int_X \min_{b \in A} Q_n(y,b) p(dy|x,a,\mu_n)  \right|}_{=(b_l)(x,a)}.
\end{align*}
By definition of the $W_1$-metric, as a consequence of Assumption \ref{assump:2} and \cite[Theorem 3.5]{serfozo1982convergence}, we obtain
\[
\sup_{(x,a) \in X \times A}(a_l)(x,a) \le \sup_{(x,a) \in X\times A}\frac{\tilde L_1}{1-\beta \tilde K_1} W_1(p(\cdot|x,a,\mu^{m_{k_l}}_n),p(\cdot|x,a,\mu_n)) \to 0.
\]
Furthermore, dominated convergence theorem implies that 
\(
\lim_{l \to \infty} \sup_{(x,a) \in X \times A}(b_l)(x,a) =0.
\)
Thus, a straightforward triangle inequality implies that 
\[
\sup_{(x,a) \in X \times A}\left|\int_X \min_{b \in A}Q^{m_{k_l}}_n(y,b)p(dy|x,a,\mu^{m_{k_l}}_n) -\int_X \min_{b \in A} Q_n(y,b) p(dy|x,a,\mu_n) \right| \to 0.
\]
Now, if the initial sequence $(Q^m_n)_{n \in \mathbb N}$ were to be chosen as $(Q^m_n)_{0 \le n \le f(m)}$, where $f(m) \to \infty$ as $m \to \infty$, the same argument above holds. 

To complete the proof, it remains to show that for any given $\pmb \mu \in \mathcal P(X)^{\infty}$ there exists a unique family of $Q$-functions in $\mathcal S$.
First, we recall some definitions. For a given $\pmb \mu \in \{\mu_0\}\times \mathcal P(X)^{\infty}$, we define $T^{\pmb \mu}_t$ acting on functions $u:X \to \mathbb R$ as 
\[
T^{\pmb \mu}_t u(x) = \min_{a \in A} \left(c(x,a,\mu_t) + \beta \int_X u(y)p(dy|x,a,\mu_t)\right).
\]
Let 
\[
\hat{\mathcal S}=\left\{ u:X\to \mathbb R: \| u\|_{\infty} \le \frac{M}{1-\beta}\right\},
\]
where $M$ is the upper bound of $c$, which exists as $X \times A$ is compact due to Assumption \ref{assump:2}. Similar to \cite[Lemma 3.4]{SaBaRaSIAM}, we have the following:
\begin{lemma}\label{lem:app1}
    For any $t\ge 0$, $T^{\pmb \mu}_t$ maps $\hat{\mathcal S}$ into itself. Furthermore, for any $u,v \in \hat{\mathcal S}$, we have
    \[
    \| T^{\pmb \mu}_tu-T^{\pmb \mu}_tv\|_\infty \le \beta \| u-v\|_{\infty}.
    \]
\end{lemma}

Let $\hat{\mathcal S}^{\infty} := \prod_{t=0}^{\infty} \hat{\mathcal S}$. Then, for any $\pmb u \in \hat{\mathcal S}^\infty$, we define a new operator $T^{\pmb \mu}:\hat{\mathcal S}^\infty \to \hat{\mathcal S}^\infty$ as
\(
(T^{\pmb \mu}\pmb u)_t =T^{\pmb \mu}_tu_{t+1}
\)
for $t \ge 0$. Clearly, $T^{\pmb \mu}$ is well defined. Let $\mathfrak q>1$ be such that $\mathfrak q \beta <1$. Then, on $\hat{\mathcal S}^\infty$ we define the metric $d(\pmb u,\pmb v) := \sum_{t=0}^{\infty} \mathfrak q^{-t} \| u_t-v_t\|_{\infty}.$ Lemma \ref{lem:app1} implies the following:
\begin{lemma}
    There exists $c<1$ such that for all $\pmb u,\pmb v \in \hat{\mathcal S}^{\infty}$ we have
    \(
    d(T^{\pmb \mu}\pmb u,T^{\pmb \mu}\pmb v) \le c d(\pmb u,\pmb v).
    \)
\end{lemma}
\begin{proof}
    Directly follows from Lemma \ref{lem:app1}.
\end{proof}
Note that $\hat{\mathcal S}^{\infty}$ is complete under the metric.  Thus, for all $\pmb \mu$, there exists a unique fixed point $\pmb u^{\pmb \mu} \in \hat{\mathcal S}$ as a consequence of Banach fixed-point theorem. Note that $\pmb u_0 \equiv (0,0,\cdots) \in \hat{\mathcal S}^{\infty}.$ We recursively define $\pmb u^{\pmb \mu}_{k+1} = T^{\pmb \mu}\pmb u_k$ for $k\ge0$. Then, Banach fixed-point theorem implies that 
\(
\lim_{k \to \infty}d(\pmb u^{\pmb \mu}_k,\pmb u^{\pmb \mu}_*) =0,
\)
where $T^{\pmb \mu}u^{\pmb \mu}_* = u^{\pmb \mu}_*$. Since $T^{\pmb \mu}$ also maps $\mathcal S_X$ into itself, and the null v ector is in $\mathcal S_X$, it follows that $u^{\pmb \mu}_* \in \mathcal S_X$.

The argument above implies that for any given $\pmb \mu$, there exists a unique family of $Q$-functions, as desired. This completes the proof of Lemma \ref{lem:f6}.

\subsection{Proof of Lemma \ref{lem:1}}\label{sect:lem:1}
\begin{proof}[Proof of Lemma \ref{lem:1}]
For our purposes, only the second half of the proof of \cite[Proposition 3.9]{SaBaRaSIAM} is relevant, and we will closely follow it. Let 
\[
A_n := \{ (x,a) \in X \times A : Q_n(x,a) = \min_a Q_n(x,a) \},
\]
and
\[
A := \{ (x,a) \in X \times A: Q(x,a) = \min_a Q(x,a) \}.
\]
Note that as $Q_n$ are continuous, and $\lim_n Q_n = Q$ happens continuously, it holds that $Q$ is continuous over $X \times A$. Furthermore, by assumption we have $\pi_n \otimes \mu_n(A_n) =1$ for all $n$. Since $Q_n,\min_{a \in A}Q_n(\cdot,a),$ $\min_{a \in A} Q(\cdot,a),$ and $Q$ are continuous, we have that the sets $A_n$ and $A$ are closed. 

For all $n$, since $\pi_{n} \otimes \mu_{n}$ concentrates on optimal state-action pairs of $Q_n$, for all $n$ it holds that
\(
\pi_{n} \otimes \mu_{n} (A_n) = 1.
\)
Note that the continuity of the maps $Q_n$, $\min_{a \in A}Q_n(\cdot,a), Q$, and $\min_{a \in A}Q(\cdot,a)$ gives us that $A_n$ and $A$ are closed sets in $X \times A$. For $c > 0$, we define the open level sets
\[
A(\infty,c) = \left\{ (x,a) : Q_n(x,a) > \min_{a \in A}Q_n(x,a) + c \right\}.
\]

Using the continuity of \( Q_n \) and \( \min_{ a\in A}Q_{n}(\cdot,a) \), we have 
\[
\partial A(\infty, c) \subset (Q_n - \min_{a\in A}Q_{n}(\cdot,a))^{-1}(\{c\}),
\]
where \(\partial V\) denotes the topological boundary of a set \(V\), and \((Q_n - \min_{a\in A}Q_{n}(\cdot,a))^{-1}(\{c\})\) is the preimage of the set \(\{c\}\) under the function \(Q_n - \min_{a\in A}Q_{n}\). Since \(\pi_n \otimes \mu_n\) is a probability measure for all $n$, the pushforward measure 
\(
\rho_n := \pi_n \otimes \mu_n \circ (Q_n - \min_{a \in A}Q_{n}(\cdot,a))^{-1}
\)
is also a Borel probability measure. Therefore, for all \(n\), \(\rho_n\) has at most countably many atoms. If \(\rho_n\) had uncountably many atoms, the sum of uncountably many positive numbers would necessarily diverge, leading to a contradiction since \(\rho_n\) is a finite measure.
Hence, the set 
\[
A(\infty) = \{c>0 : \pi \otimes \mu (\partial A(\infty, c))>0\}
\]
is at most countable. Consequently, there exist uncountably many \(c>0\) such that~~\(\pi_n \otimes \mu_n (\partial A(\infty, c)) = 0.\)

Using the observation above, we can construct a decreasing sequence of positive numbers \((c_n)_{n \in \mathbb N}\) such that \(\lim_{n \to \infty} c_n = 0\) and 
\(
\pi_t \otimes \mu_t (\partial A(\infty, c_n)) = 0,
\)
with \(A(\infty, c_n) \subset A(\infty, c_{n+1})\) for all \(n\). Define 
\[
A^{<}(\infty, k) := (A \cup A(\infty, k))^c.
\]
With these notations, we can decompose the whole \(X \times A\) into disjoint Borel measurable sets as follows:
\[
X \times A = A \cup A(\infty, c_m) \cup (A \cup A(\infty, c_m))^c =: A \cup A(\infty, c_m) \cup A^{<}(\infty, c_m),
\]
for all \(m\). Finally, we note that the sublevel set \(A_t \cup A^{<}(\infty, c_n)\) is closed in \(X \times A\) for all $n$.

 Observing that
\[
1 = \pi_{n} \otimes \mu_{n} (A_n ) = \pi_{n} \otimes \mu_{n}(A \cap A_n) + \pi_{n} \otimes \mu_{n}(A^{<}(\infty,c_m) \cap A_n) + \pi_{n} \otimes \mu_{n} ( A(\infty,c_m) \cap A_n)
\]
for all $n$ and $m$, we obtain
\[
1 \leq \limsup_{n \rightarrow \infty} \left( \pi_{n} \otimes \mu_{n}(A \cap A_n) + \pi_{n} \otimes \mu_{n}A^{<}(\infty,c_m) \cap A_n) + \pi_{n} \otimes \mu_{n} ( A(\infty,c_m) \cap A_n) \right)
\]
for all $m$; thus,
\begin{equation}\label{eqqq6}
1 \leq \liminf_{m \rightarrow \infty} \limsup_{n \rightarrow \infty} \left( \pi_{n} \otimes \mu_{n}((A \cup A^{<}(\infty,c_m)) \cap A_n) + \pi_{n} \otimes \mu_{n} ( A(\infty,c_m) \cap A_n) \right).
\end{equation}

First, we aim to show that 
\(
\limsup_{n \to \infty} \pi_{n} \otimes \mu_{n}(A(\infty, c_m) \cap A_n) = 0.
\)
Observing that 
\[
0 \leq \pi_{n,t} \otimes \mu_{n,t}(A(\infty, c_m) \cap A_n) \leq \pi_{n} \otimes \mu_{n} \left( (\partial A(\infty, c_m) \cup A(\infty, c_m)) \cap A_n \right),
\]
and noting that \((\partial A(\infty, c_m) \cup A(\infty, c_m)) \cap A_n\) is closed, we claim that the indicator function \(1_{(\partial A(\infty, c_m) \cup A(\infty, c_m)) \cap A_n}\) converges continuously to 0.

Let \((x_n, a_n) \in A_n \cap (\partial A(\infty, c_m) \cup A(\infty, c_m))\) for all \(n\), and suppose \((x_n, a_n) \to (x, a) \in X \times A\). Since \(\partial A(\infty, c_m) \cup A(\infty, c_m)\) is closed in \(X \times A\), it follows that \((x, a) \in \partial A(\infty, c_m) \cup A(\infty, c_m)\). Consequently, 
\[
\lim_{n \to \infty} Q_n(x_n, a_n) = Q(x, a) \geq \min_{ a\in A}Q(x,a) + c_m = \lim_{n \to \infty} \min_{a \in A}Q_n(x_n,a) + c_m.
\]
Thus, for any sufficiently large \(n\), we have \(Q_n(x_n, a_n) \not =  \min_{a \in A}Q_n(x_n,a)\), implying that 
\[
\
\lim_{n\rightarrow \infty} 1_{(\partial A(\infty, c_m) \cup A(\infty, c_m)) \cap A_n}(x_n, a_n) = 0.
\]
Then, as \(\pi_n \otimes \mu_n(\partial A(\infty, c_m)) = 0\) for all \(m\), from Portmanteau Theorem \cite[Theorem 2.1.-(iii)]{billingsley2013convergence} it follows that for each $m$, we have
\[
\begin{aligned}
0 &= \lim_{n \to \infty} \pi_n \otimes \mu_n \left( (\partial A(\infty, c_m) \cup A(\infty, c_m)) \cap A_n \right) \\
&= \limsup_{n \to \infty} \pi_{n} \otimes \mu_{n} \left( (\partial A(\infty, c_m) \cup A(\infty, c_m)) \cap A_n \right),
\end{aligned}
\]
where we used the dominated convergence theorem on the first line and \cite[Theorem 3.5]{langen1981convergence} on the second line.
This reduces \eqref{eqqq6} to
\begin{align*}
1 &\leq \liminf_{m \rightarrow \infty} \limsup_{n\rightarrow \infty} \left(\pi_{n} \otimes \mu_{n}((A \cup A^{<}(\infty,c_m))\cap A_n) + \pi_{n} \otimes \mu_{n}(A(\infty,c_m)\cap A_n)\right)
\\&= \liminf_{m \rightarrow \infty} \limsup_{n\rightarrow \infty} \pi_{n} \otimes \mu_{n}((A \cup A^{<}(\infty,c_m))\cap A_n)
\\&\leq \liminf_{m \rightarrow \infty} \limsup_{n\rightarrow \infty} \pi_{n} \otimes \mu_{n}(A \cup A^{<}(\infty,c_m)).
\end{align*}
Since each $A \cup A^{<}(\infty,c_m)$ is closed, using again the Portmanteau Theorem \cite[Theorem 2.1.-(iii)]{billingsley2013convergence}, we have that
\begin{align*}
1 &\leq \liminf_{m \rightarrow \infty} \limsup_{n \rightarrow \infty} \pi_{n} \otimes \mu_{n} \left( A \cup A^{<}(\infty,c_m) \right)  \leq \liminf_{m \rightarrow \infty}\pi \otimes \mu \left( A \cup A^{<}(\infty,c_m) \right) .
\end{align*}
Note that $A \cup A^{<}(\infty,c_m)$ satisfies the property 
\[
\bigcap_{m \in \mathbb{N}} (A \cup A^{<}(\infty,c_m)) = A
\]
as a consequence of the continuity of $Q$ and $\min_{a \in A}Q(\cdot,a)$. 
Together with the monotone convergence theorem, this observation gives us 
\[
1 \leq \liminf_{m \rightarrow \infty}\pi \otimes \mu \left( A \cup A^{<}(\infty,c_m) \right)  = \pi \otimes \mu (A),
\]
which demonstrates $\pi \otimes \mu (A)=1$ as desired.
\end{proof}
\subsection{Proof of Lemma \ref{lem:2}}\label{sect:lem:2}
\begin{proof}
    Let $\lim_t \pi_t \otimes \mu_t = \pi \otimes \mu$ under the weak convergence, which exists by our running assumption.
    Since the $1$-Wasserstein metric over $\mathcal P(X)$ metrizes the weak convergence in our setting, we can $g:X \to \mathbb R$ that are $1$-Lipschitz as test functions to verify this property. Let $g$ be such a function. Since $X$ is compact, it follows that $g$ is bounded. Then, using the Fubini theorem, we obtain
\begin{align*}
\int_X g(x)\mu_{t}(dx) 
&= \int_X \int_{X \times A} g(y)p(dy|x,a,\mu_{t})\pi_{t}\otimes \mu_{t}(da,dx)
\\&=\int_{X \times A}  \int_X g(y)p(dy|x,a,\mu_{t})\pi_{t}\otimes \mu_{t}(da,dx).
\end{align*}
Now, using the triangle inequality, we obtain
\begin{align*}
    &\bigg|\int_{A \times X}  \int_X g(y)p(dy|x,a,\mu_{t})\pi_{t}\otimes  \mu_{t}(da,dx)
     -\int_{A \times X}\int_X g(y) p(y|x,a,\mu)\pi \otimes \mu(da,dx)\bigg|
    \\& \le \underbrace{\bigg|\int_{A \times X}  \int_X g(y)p(dy|x,a,\mu_t)\pi_t\otimes \mu_t(da,dx)-\int_{A \times X}\int_Xg(y) p(y|x,a, 
    \mu)\pi_{t} \otimes \mu_{t}(da,dx)\bigg|}_{(a_t)}
    \\& \quad+\underbrace{\bigg|\int_{X \times A}\int_Xg(y) p(y|x,a, 
    \mu)\pi_{t} \otimes \mu_{t}(da,dx) - \int_{X \times A}\int_Xg(y) p(y|x,a,\mu) \pi \otimes \mu(da,dx) \bigg|}_{(b_t)}.
\end{align*}
By Assumption \ref{assump:2}, for any $(x_n,a_n,\mu_n)_n \subset X \times A \times \mathcal P(X)$ such that 
\(
\lim_{n \to \infty} (x_n,a_n,\mu_n)_n = (y,v,\nu)
\)
in $X \times A \times \mathcal P(X)$, using the fact that $W_1$ metrizes the weak convergence in our case, we obtain that
\[
\lim_{n \to \infty} \left|\int_X g(y) p(dy|x_n,a_n,\mu_n), \int_X g(y)p(dy|y,v,\nu) \right| \le \lim_{n \to \infty}W_1 \left( p(\cdot|x_n,a_n,\mu_n),p(\cdot|y,v,\nu)\right)=0,
\]
i.e., \(\int_X g(y)p(dy|\cdot,\cdot,\cdot)\) is continuously convergent in \(X \times A \times \mathcal P(X)\).

By \cite[Theorem 3.5]{serfozo1982convergence}, it directly follows that 
\(
\lim_{t \to \infty} (a_t)  =0.
\)
Since $\int_X g(y)p(dy|\cdot,\cdot,\mu)$ is continuously convergent in $X \times A \times \mathcal P(X)$, we have that $\int_X g(y)p(dy|\cdot,\cdot,\mu)$ is bounded as $X\times A$ is compact. Thus, as a consequence of the dominated convergence theorem, we obtain
\(
\lim_{t \to \infty} (b_t) =0 
\)
as $(\pi_t\otimes \mu_t)_t$ is weakly convergent to $\pi\otimes \mu$ in $\mathcal P(X\times A)$. Therefore,
\[
\bigg|\int_{A \times X}  \int_X g(y)p(dy|x,a,\mu_{t})\pi_{t}\otimes  \mu_{t}(da,dx)
     -\int_{A \times X}\int_X g(y) p(y|x,a,\mu)\pi \otimes \mu(da,dx)\bigg| \to 0
\]
for any $1$-Lipschitz $g:X \to \mathbb R$.
Since $g$ are test functions for the convergence under the $1$-Wasserstein metric, it follows that
\[
W_1\left( \int_{A \times X}p(\cdot|x,a,\mu_t)\pi_t\otimes \mu_t(da,dx), \int_{A \times X}p(\cdot|x,a,\mu)\pi\otimes \mu(da,dx)\right) \to 0,
\]
as desired.
\end{proof}

\subsection{Proof of Lemma \ref{lem:f7}}\label{sect:f7}
The proof heavily relies on \cite[Lemma A.1]{SaBaRaSIAM} and \cite[Proposition 3.10]{SaBaRaSIAM}, and thus we will recall some technical tools introduced there. Part of the tools are also presented in the proof of Lemma \ref{lem:f6}. Let $\pmb \mu = (\mu_t)_t \in \{\mu_0\}\times \mathcal P(X)^{\infty}$ and $\pmb \mu^n = (\mu_{t}^n)_t \in \{\mu_0\} \times \mathcal P(X)^{\infty}$. Suppose $\pmb \mu^n$ converges to $\pmb \mu$ in $\mathcal P(X)^{\infty}$. In \cite[Proposition 3.10]{SaBaRaSIAM}, it is shown that $Q$-functions $(Q^n_t)_t$ and $(Q_t)_t$ defined as
\[
Q^n_t(x,a) = c(x,a,\mu^n_t) + \beta \int_X \min_{b \in A}Q^n_{t+1}(y,b)p(dy|x,a,\mu^n_t),
\]
and
\[
Q_t(x,a) = c(x,a,\mu_t) + \beta \int_X \min_{b \in A}Q_{t+1}(y,b)p(dy|x,a,\mu_t)
\]
satisfy the convergence
\[
\lim_{n \to \infty} \sup_{x \in X} \left|\min_{b \in A}Q^n_{t}(x,b) -\min_{b \in A}Q_{t}(x,b)\right| =0.
\]
We recall that in the proof of Lemma \ref{lem:f6}, we also demonstrated that for a given family of flows $\pmb \mu$, such $Q$-functions are uniquely determined by the flow $\pmb \mu$; thus, the limits above are not subject to the choice of representation.

When $\pmb \mu^n \in \{\mu^n_0\} \times \mathcal P(X)^{\infty}$, one can check that the same proof as in \cite[Proposition 3.10]{SaBaRaSIAM} holds if $\mu^n_0 \to \mu_0$ weakly in $\mathcal P(X)$. Using this observation, we obtain the following:
\begin{lemma}
    Suppose that Assumption \ref{assump:2} holds. Let $\pmb \mu^n \in \{\mu^n_0\} \times \mathcal P(X)^{\infty}$, $\pmb \mu \in \{\mu_0\} \times \mathcal P(X)^{\infty}$ be such that $\lim_n \pmb \mu^n = \pmb \mu$ in $\mathcal P(X)^{\infty}$. Then, we have
    \[
    \lim_{n \to \infty}\sup_{(x,a)} | Q_t(x,a)-Q^n_t(x,a)| = 0.
    \]
\end{lemma}
\begin{proof}
    As argued above, for all $t \in \mathbb N$, \cite[Proposition 3.10]{SaBaRaSIAM} implies
    \[
\lim_{n \to \infty} \sup_{x \in X} \left|\min_{b \in A}Q^n_{t}(y,b) -\min_{b \in A}Q_{t}(y,b)\right| = 0.
\]
Thus, \cite[Theorem 3.5]{serfozo1982convergence} implies that 
\[
\lim_{n \to \infty} \int_X \min_{b \in A} Q^n_{t+1}(y,b)p(dy|x_n,a_n,\mu^n_t) = \int_X \min_{b \in A} Q_{t+1}(y,b)p(dy|x,a,\mu_t)
\]
whenever $(x_n,a_n)\subset X \times A$ converges to $(x,a)$. Thus, we have
\[
\lim_{n \to \infty} \sup_{ (x,a) \in X \times A} \left|  \int_X \min_{b \in A} Q^n_{t+1}(y,b)p(dy|x,a,\mu^n_t) - \int_X \min_{b \in A} Q_{t+1}(y,b)p(dy|x,a,\mu_t)\right| =0
\]
as $X\times A$ is compact. Furthermore, Assumption \ref{assump:2} also implies that 
\[
\lim_{n \to \infty} | c(x_n,a_n,\mu^n_t) - c(x,a,\mu_t) |=0
\]
whenever $(x_n,a_n) \subset X \times A$ converges to $(x,a)$, meaning that
\[
\lim_{n \to \infty} \sup_{(x,a) \in X \times A} | c(x,a,\mu^n_t) - c(x,a,\mu_t) | = 0.
\]
Thus, as a consequence of a straightforward triangle inequality, now it is easy to see that 
\[
    \lim_{n \to \infty}\sup_{(x,a)} | Q_t(x,a)-Q^n_t(x,a)| = 0
    \]
    holds.
\end{proof}
Lemma \ref{lem:f7} directly follows from the lemma above.
\end{document}